\documentclass[11pt,a4paper]{amsart}
\usepackage[top=1.33in, bottom=1.33in, left=1.25in, right=1.25in]{geometry}
\usepackage{amsmath, amsxtra, amsthm, amsfonts, amssymb, mathtools}
\usepackage{mathrsfs}
\usepackage{graphicx}
\usepackage{url}
\usepackage[dvipsnames]{xcolor}
\usepackage{bbm}
\usepackage{bm}
\usepackage{tikz-cd}
\usepackage{tikz}
\usetikzlibrary{backgrounds}
\usepackage[cal=boondoxo,scr=euler]{mathalfa}
\usepackage{float}
\usetikzlibrary{decorations.pathreplacing, positioning}
\usetikzlibrary{decorations.markings}
\usetikzlibrary{shapes,positioning,intersections,quotes}
\usetikzlibrary{shapes,fit}
\usetikzlibrary{arrows.meta, calc}
\definecolor{curvyarrow}{RGB}{200, 100, 150}
\usepackage{comment}
\usepackage[shortlabels]{enumitem} 
\usepackage{enumitem}
\usepackage{relsize}
\usepackage{stmaryrd}

\usepackage[all,cmtip]{xy}	
\xyoption{arrow}
\usepackage{mlmodern} 
\usepackage[T1]{fontenc}
\usepackage{enumitem}

\usepackage[backref=page,linktocpage]{hyperref} 
\hypersetup{
    colorlinks,
    allcolors=teal
}

\usepackage[capitalize]{cleveref}

\newtheoremstyle{results}
{8pt}
{8pt}
{\itshape}
{}
{\bfseries}
{}
{.5em}
{}
\theoremstyle{results}
\newtheorem{thm}{Theorem}[section]
\newtheorem{proposition}[thm]{Proposition}
\newtheorem{corollary}[thm]{Corollary}
\newtheorem{lemma}[thm]{Lemma}
\newtheorem{maintheorem}[thm]{Theorem}

\newtheoremstyle{definitions}
{7pt}
{7pt}
{}
{}
{\bfseries}
{}
{.5em}
{}

\theoremstyle{definitions}
\newtheorem{definition}[thm]{Definition}
\newtheorem{conjecture}[thm]{Conjecture}

\newtheorem{example}[thm]{Example}

\newtheorem{question}[thm]{Question}

\crefname{thm}{theorem}{theorems}
\Crefname{thm}{Theorem}{Theorems}
\crefname{lemma}{lemma}{lemmas}
\Crefname{lemma}{Lemma}{Lemmas}
\crefname{proposition}{proposition}{propositions}
\Crefname{proposition}{Proposition}{Propositions}
\crefname{corollary}{corollary}{corollaries}
\Crefname{corollary}{Corollary}{Corollaries}
\crefname{definition}{definition}{definitions}
\Crefname{definition}{Definition}{Definitions}
\crefname{conjecture}{conjecture}{conjectures}
\Crefname{conjecture}{Conjecture}{Conjectures}
\crefname{example}{example}{examples}
\Crefname{example}{Example}{Examples}
\crefname{remark}{remark}{remarks}
\Crefname{remark}{Remark}{Remarks}

\def\Z{\mathbb{Z}}

\def\R{\mathbb{R}}
\def\C{\mathbb{C}}
\def\F{\mathbb{F}}
\def\P{\mathbb{P}}

\def\Z{\mathbb{Z}}

\def\calA{\mathcal{A}}

\def\calC{\mathcal{C}}

\def\calL{\mathcal{L}}

\def\calS{\mathcal{S}}

\newcommand{\abs}[1]{\left\lvert#1\right\rvert}

\newcommand{\QnG}{\textrm{Q}_n(\textrm{G})}

\newcommand{\CH}{\mathrm{CH}}
\newcommand{\uH}{\underline{\mathrm{H}}}
\newcommand{\ChowGr}
{\underline{\mathrm{CH}}}

\newcommand{\G}{\textrm{G}}

\newcommand{\Tr}{\operatorname{Tr}}

\newcommand{\trunc}{\Tr}

\newcommand{\uCH}{\underline{\textrm{CH}}}

\newcommand{\rk}{\operatorname{rk}}

\newcommand{\M}{\mathrm{M}}

\newcommand{\floor}[1]{\lfloor #1 \rfloor}

\renewcommand{\L}{\mathcal{L}}
\newcommand{\zero}{\hat{0}}

\renewcommand{\S}{\calS}

\newcommand{\des}{\mathrm{des}}

\newcommand{\un}{\hat{1}}
\newcommand{\At}{\mathrm{At}}
\newcommand{\Qd}{\mathrm{Q}}

\newcommand{\Edge}{\mathcal{E}}

\newcommand{\Latt}{\mathrm{L}}
\newcommand{\medbullet}{\mathbin{\vcenter{\hbox{\scalebox{0.6}{$\bullet$}}}}}

\renewcommand{\L}{\mathcal{L}}
\newcommand{\Ind}{\mathrm{Ind}}

\newcommand{\Harr}{\mathcal{H}}
\newcommand{\Arr}{\mathcal{A}}
\newcommand{\Conf}{\mathrm{Conf}}
\newcommand{\Des}{\mathrm{Des}}

\newcommand{\ndd}{\mathrm{ndd}}

\newcommand{\Pdual}{\P^{\ast}}
\renewcommand{\P}{\mathrm{P}}
\renewcommand{\H}{\mathrm{H}}
\renewcommand{\hat}{\widehat}

\newcommand{\ind}{\mathrm{Ind}}

\newcommand{\asce}{\mathsmaller{\pmb{\smallnearrow}}}
\newcommand{\desc}{\mathsmaller{\pmb{\smallsearrow}}}
\newcommand{\smallnearrow}{\tikz[scale=0.18, baseline={(0, 0)}, >={Stealth}]\draw[->, thick, line width=0.4pt] (0, 0) -- (1,1);}
\newcommand{\smallsearrow}{\tikz[scale=0.18, baseline={(0, 0)}, >={Stealth}]\draw[->, thick, line width=0.4pt] (0,1) -- (1,0);}

\title{Chow polynomials of rank-uniform labeled posets}

\author{Basile Coron}
\address{B. Coron, Centre de Mathématiques Laurent Schwartz, Polytechnique, Paris, France}
\email{basile.coron@polytechnique.edu}

\author{Luis Ferroni}
\address{L. Ferroni, Dipartimento di Matematica, Universit\`a di Pisa, Pisa, Italy}\email{luis.ferroni@unipi.it}

\author{Shiyue Li}
\address{S. Li, Department of Mathematics, University of Michigan, Ann Arbor, USA}\email{shiyueli@umich.edu}
\date{\today}

\begin{document}
\begin{abstract}
     We introduce and develop the theory of UMEL-shellable posets. These are posets equipped with an edge-lexicographical labeling satisfying certain uniformity and monotonicity properties. This framework encompasses classical families of combinatorial geometries, including uniform matroids, projective and affine geometries, braid matroids of type A and B, and all Dowling geometries. It also comprises all rank-uniform semimodular supersolvable lattices, and therefore also all rank-uniform distributive lattices. Our main result establishes real-rootedness phenomena for the Chow polynomials, the augmented Chow polynomials, and the chain polynomials associated with those posets, thus making simultaneous progress towards conjectures by Ferroni--Schr\"oter, Huh--Stevens, and Athanasiadis--Kalampogia-Evangelinou. In the special case of lattices of flats of matroids, the (augmented) Chow polynomials coincide with the Hilbert--Poincar\'e series of the Chow ring associated to the smooth and generally noncompact toric varieties of the (augmented) Bergman fan of the matroid, whereas the chain polynomial encodes the Hilbert--Poincar\'e series of the Stanley--Reisner ring of the Bergman complex of the matroid. Therefore, these real-rootedness results are tightly linked to the study of these algebro-geometric structures in matroid theory.
\end{abstract}

\maketitle
\setcounter{tocdepth}{1}\tableofcontents

\section{Introduction}

\subsection{Real-rootedness in algebraic and geometric combinatorics} 
Over the past few years, the algebraic study of matroids and geometric lattices has been transformed by a surge of ideas bridging combinatorics, geometry, and algebra. What began as the search for combinatorial proofs of positivity conjectures has grown into a rich and interconnected theory, informed by deep insights from algebraic and polyhedral geometry (see \cite{ardila-icm22,huh-icm22,eur}). Yet, this progress has also given rise to new challenges. Several algebraically defined polynomial invariants of matroids have been recently conjectured to have only real roots, bringing analytic tools into play.

Among \emph{positivity} phenomena for polynomials with nonnegative coefficients, real-rootedness stands as the strongest: it is equivalent to the Toeplitz matrix of a polynomial being totally nonnegative (\cite[Theorem~1.3]{branden2015unimodality}) and, by Newton’s inequalities \cite[Theorem~2]{stanley-unimodality}, it implies the log-concavity and the unimodality of the coefficients. Many fascinating connections of this property with the notions of $\gamma$-positivity and Koszulity of rings (see, e.g., \cite[Section~4]{reiner-welker}, \cite[Section~5.4]{ferroni-matherne-stevens-vecchi}) have also motivated new avenues of research.

While many algebraic and geometric invariants of matroids are known to be log-concave, unimodal, or $\gamma$-positive, the property of real-rootedness remains more elusive. A central motivating question of the present work is the following conjecture of Ferroni and Schröter \cite[Conjecture~8.18]{ferroni-schroter}.

\begin{conjecture}\label{conj:real-rootedness-chow}
    The Chow polynomial of a matroid has only nonpositive real roots.
\end{conjecture}

This conjecture, and a counterpart for \emph{augmented} Chow polynomials, has been independently proposed by Huh and Stevens \cite[Conjecture~4.3.3]{stevens-bachelor}. 

The Chow polynomial of $\M$ is originally defined as the Hilbert--Poincar\'e series of the Chow ring of $\M$. The Chow ring of $\M$ was introduced by Feichtner and Yuzvinsky~\cite{feichtner2004chow} and played a crucial role in the resolution of the Heron--Rota--Welsh conjecture by Adiprasito, Huh, and Katz~\cite{adiprasito2018hodge}. When the matroid $\M$ is represented by a  hyperplane arrangement $\calA$ of size $n$, the Chow ring of $\M$ is the Chow ring of the wonderful variety, or the wonderful compactification of the complement of $\calA$, introduced by De Concini and Procesi in~\cite{de1995wonderful}. The wonderful variety is a subvariety of the smooth and projective permutohedral toric variety $X_{\underline{A}_n}$, and its Chow ring admits a set of generators given by base-point-free divisor classes pulled back from $X_{\underline{A}_n}$. As a graded vector space over any infinite field, the Chow ring has two combinatorial monomial bases, given by nested sets in the lattice of flats \cite{feichtner2004chow} and  nested matroid quotients \cite{backman2023simplicial}. The augmented Chow polynomial is defined analogously as the Hilbert--Poincar\'e series of the augmented Chow ring of $\M$; this ring was introduced by Braden, Huh, Matherne, Proudfoot, and Wang \cite{BRADEN2022108646} and was instrumental in the proof of the Dowling--Wilson top-heavy conjecture by the same authors \cite{braden2020singular}. When the matroid $\M$ is represented by a hyperplane arrangement $\calA$, the augmented Chow ring of $\M$ is the Chow ring of an iterated blowup, called the augmented wonderful variety, of the matroid Schubert variety of $\M$ introduced by Ardila and Boocher \cite{ardila2016closure}. The wonderful variety of $\calA$ appears as a divisor inside the augmented wonderful variety of $\calA$. The augmented wonderful variety is a subvariety of the smooth and projective stellahedral toric variety $X_{A_n}$, and its Chow ring admits a set of generators given by base-point-free divisor classes pulled back from $X_{A_n}$ \cite{eur2022stellahedral}. The two monomial bases of the Chow ring of $\M$ immediately generalize to augmented monomial bases of the augmented Chow ring of $\M$. Therefore, the (augmented) Chow polynomial can be viewed as the generating polynomials of these monomial bases. Despite the rich structural results about these varieties, we do not know algebro-geometric explanations for why their Chow rings exhibit real-rooted Hilbert--Poincaré series in extensive computations. In contrast, closely related projective spaces, Grassmannians and flag varieties are well-known to have Chow rings with Hilbert--Poincaré series possessing non-real roots. 

In recent work of Ferroni, Matherne, and Vecchi \cite{ferroni-matherne-vecchi} an alternative combinatorial definition of Chow and augmented Chow polynomials of graded posets is introduced, circumventing the necessity of taking the Hilbert--Poincar\'e series of a ring. When the poset equals $\mathcal{L}(\M)$, the lattice of flats of a matroid $\M$, these notions agree with the Hilbert--Poincar\'e series of the Chow ring and the augmented Chow ring of $\M$ (we recapitulate the key definitions in Section~\ref{sec:chow-poly-recap}). In a nutshell, the Chow polynomial $\uH_\P(x)$ of a graded bounded poset $\P$ is the value that the inverse of the reduced characteristic function takes at $[\widehat{0},\widehat{1}]$ in the incidence algebra of $\P$ (here $\widehat{0}$ and $\widehat{1}$ denote the minimum and maximum elements of $\P$). The augmented Chow polynomial $\P$, denoted $\H_\P(x)$, is the Chow polynomial of the poset obtained by adding a minimum below $\P$. This poset-theoretic perspective enables us to formulate our first main result.
\begin{maintheorem}
\label{thm:main1-intro}
     Let $\P$ be a UMEL-shellable poset. 
    Then the following statements hold: 
    \begin{enumerate}[\normalfont (i)]
        \item The Chow polynomial $\uH_{\P}(x)$ only has real and nonpositive roots. 
        \item The augmented Chow polynomial $\H_{\P}(x)$ only has real and nonpositive roots. 
    \end{enumerate}
\end{maintheorem}
The precise definition of a \emph{UMEL-shellable} poset is presented in Section~\ref{sec:rank-uniform-labeled-posets}. This new class of posets is a subclass of EL-shellable posets \cite{bjorner-el-shellable}, i.e., posets admitting an edge-lexicographic labeling (for more details, see the recapitulation in Section~\ref{sec:definitions}). The letter `U' in the acronym UMEL stands for posets being \emph{rank-uniform}: this means that for any two elements $s,t\in \P$ of the same rank, the upper intervals $[s,\widehat{1}]$ and $[t,\widehat{1}]$ have the same Whitney numbers of the second kind, i.e., the same number of rank $j$ elements for each $j$. When $\P$ is a geometric lattice these posets were studied in great detail by Brylawski\footnote{Brylawski~\cite[p.~223]{brylawski} uses the term \emph{upper combinatorially uniform} (or plainly uniform) for the posets we call rank-uniform.} \cite{brylawski}. The letter `M' stands for a technical monotonicity condition that the edge-lexicographic labeling must satisfy. These conditions altogether provide us with new recurrence relations for Chow polynomials, which we employ along with the machinery of interlacing families to deduce several real-rootedness results. The proof of the above theorem also enables us to uncover strong interlacing phenomena between the roots of the foregoing polynomials. For UMEL-shellable posets, the first of the five statements below settles \cite[Conjecture~5.7]{ferroni-matherne-stevens-vecchi}. The other four statements settle a private conjecture of Li (see also~\Cref{conj:contraction-interlace} and \Cref{conj:truncation-interlace}). In the case of Boolean lattices, the fourth statement recovers that the consecutive Eulerian polynomials interlace \cite[Theorem 2.7]{savage2015}; the fifth statement recovers that the $n$th Eulerian polynomial interlaces with the $n$th derangement polynomial {\cite[Proof of Theorem 3.6]{branden2021symmetric}}. 

\begin{maintheorem}
\label{thm:main-interlacing-intro}
    Let $\P$ be a UMEL-shellable poset. The following statements hold: 
    \begin{enumerate}[\normalfont(i)]
        \item The roots of $\uH_{\P}(x)$ and $\H_{\P}(x)$ interlace. 
        \item For an atom $a$ of $\P$, the roots of $\H_{[a, \un]}(x)$ and $\uH_{\P}(x)$ 
        interlace. 
        \item For an atom $a$ of $\P$, the roots of $\H_{[a, \un]}(x)$ and $\H_{\P}(x)$ 
        interlace.
        \item For an atom $a$ of $\P$ with a mild condition (see \Cref{thm:gamma-interlacing-aug-contract}), the roots of $\uH_{[a, \un]}(x)$ and $\uH_{\P}(x)$ 
        interlace. 
        \item Let $\Tr(\P)$ be the truncation of $\P$. The polynomials $\H_{\Tr(\P)}(x)$, and $\uH_{\Tr(\P)}(x)$ are real-rooted, and the following interlacing relations hold (where $\preceq$ indicates root interlacing): 
        \[
        \begin{tikzcd}
        \uH_{\Tr(\P)} \arrow[r, "\preceq"] \arrow[d, "\preceq"] & \uH_{\P} \arrow[d, "\preceq"] \\
        \H_{\Tr(\P)} \arrow[r, "\preceq"] & \H_{\P}.
        \end{tikzcd}
        \] 
    \end{enumerate}
\end{maintheorem}

For the precise results implying the properties in the preceding list, see \Cref{thm:gamma-interlacing-aug-contract} and \Cref{thm:gamma-truncation-interlacing}. Several interesting classes of matroids are covered by the preceding results (see Theorem~\ref{thm:main-matroid-specializations} below). Furthermore, the above result also makes progress towards a more general conjecture on the real-rootedness of Chow polynomials appearing in \cite[Conjecture~1.5]{ferroni-matherne-vecchi} by Ferroni, Matherne, and Vecchi. That conjecture proposes that the Chow polynomial of any Cohen--Macaulay poset is real-rooted. The posets considered in this article are Cohen--Macaulay as a consequence of being EL-shellable.

\medskip 

Let us denote by $\widehat{\mathcal{L}}(\M)$ the set of all proper nonempty flats of $\M$. Since the poset $\mathcal{L}(\M)$ is a geometric lattice, the order complex $\Delta(\widehat{\mathcal{L}}(\M))$, often called the \emph{Bergman complex of $\M$}, displays many desirable combinatorial and topological features---for example it is shellable and thus Cohen--Macaulay.
The $f$-vector of the Bergman complex of $\M$ encodes the number of flags of flats in $\widehat{\mathcal{L}}(\M)$ according to size, and so its generating function, the $f$-polynomial of the complex, is often called the \emph{chain polynomial} of $\widehat{\mathcal{L}}(\M)$. A recent conjecture by Athanasiadis and Kalampogia-Evangelinou \cite[Conjecture~1.2]{athanasiadis-kalampogia} proposes that the chain polynomial of a geometric lattice is a real-rooted polynomial. This conjecture can be recast in terms of $h$-polynomials (which are linear transformations of $f$-polynomials) as follows.

\begin{conjecture}\label{conj:real-rootedness-chain}
    The $h$-polynomial of the order complex of the lattice of flats (or, equivalently, the Bergman complex) of a matroid has only nonpositive real roots.
\end{conjecture}

The $h$-polynomial of a $d$-dimensional simplicial complex $\Delta$ will be denoted by $h(\Delta;y)$. 
It can also be defined by
\[
\operatorname{Hilb}(S/I_{\Delta}; y) \coloneqq \sum_{i\geqslant 0} \dim (S/I_{\Delta})_i\, y^i = \frac{h(\Delta; y)}{(1-y)^{d+1}}, 
\]
where $S/I_{\Delta}$ is the Stanley-Reisner ring of $\Delta$ over any field, and $\operatorname{Hilb}(\,\cdot\,;y)$ stands for Hilbert--Poincar\'e series. The Chow ring $\ChowGr(\M)$ can be presented as a quotient of the Stanley--Reisner ring of $\Delta(\widehat{\mathcal{L}}(\M))$. However, no general relationship between the Hilbert--Poincar\'e series of these rings is known---except for the special case of uniform matroids (see \cite[Conjecture~6.2]{hameister-rao-simpson} and \cite[Corollary~3.17]{ferroni-matherne-stevens-vecchi}).

Unlike the case of augmented Chow polynomials for which real-rootedness is conjectured to be true, the augmented counterpart of Conjecture~\ref{conj:real-rootedness-chain} is known to be false. The \emph{augmented Bergman complex}, introduced in \cite{BRADEN2022108646}, does not in general have a real-rooted $h$-polynomial---in fact, by results of Athanasiadis and Ferroni \cite[Section~4]{athanasiadis-ferroni} even unimodality fails for this polynomial. 

The following is another main result. It provides a proof of Conjecture~\ref{conj:real-rootedness-chain} for the class of UMEL-shellable posets.

\begin{maintheorem}
\label{thm:main2-intro}
    Let $\P$ be a UMEL-shellable poset. The $h$-polynomial of the order complex $\Delta(\P)$ of $\P$ only has real and nonpositive roots.
\end{maintheorem}

The proof of this result is very similar in spirit to the proof of Theorem~\ref{thm:main1-intro}. As is hinted from the above discussion, it is unexpected that one can prove instances of both Conjectures~\ref{conj:real-rootedness-chow} and \ref{conj:real-rootedness-chain} using essentially the same framework. In fact, before carrying out the proof of Theorem~\ref{thm:main1-intro} in Section~\ref{sec:chow-RR}, we present the proof of Theorem~\ref{thm:main2-intro} in Section~\ref{sec:h-RR}. The reason is that the ideas used to prove the latter are instrumental to prove the former, which is more technical.

Specializing Theorem~\ref{thm:main2-intro} to the matroids listed below, in Theorem~\ref{thm:main-matroid-specializations}, yields the real-rootedness of the $h$-polynomial of the Bergman complex of several families of matroids. This retrieves recent results appearing in the literature, especially in the work of Athanasiadis and Kalampogia-Evangelinou \cite{athanasiadis-kalampogia} and Br\"and\'en and Saud-Maia-Leite \cite{branden-saud1,branden-saud2} (see below for a historical remark).

\subsection{Rank selection, truncations, and duality}

If $n$ is a positive integer, we denote $[n] := \{1,\ldots,n\}$. For a graded bounded poset $\P$ of rank $r$ and a subset $S\subseteq [r-1]$ the \emph{rank-selected} poset $\P_S$ is the subposet of $\P$ induced on the subset $\{x\in \P: \rk(x)\in S\cup\{0,r\}\}$. In the special case in which $S = [r-2]$, the poset $\P_S$ is customarily called the \emph{truncation} of $\P$, and we denote it by $\trunc(\P)$. When $\P$ is the lattice of flats of a matroid $\M$, the truncation of $\P$ agrees with the lattice of flats of the matroid truncation of $\M$. 

In Proposition~\ref{prop:truncation} we show that truncations preserve the property of being a UMEL-shellable poset. In the framework of Chow polynomials this is especially interesting because of the tight connection between truncations and poset duality. Let $\P^{\ast}$ be the dual poset of $\P$, obtained by reversing the partial order of $\P$. Then we have 
\[ \uH_{\P^*}(x) = \H_{\trunc(\P)}(x) \qquad \text{ and } \qquad \H_{\P^*}(x) = \H_{\P}(x). \]
These identities were one of the main results in \cite[Theorem~1.2 and Corollary~4.2]{branden-vecchi1}, and allow us to derive many results about duals of UMEL-shellable posets. Nonetheless we give simple and self-contained proofs in this paper (see Lemma~\ref{lemma:duality}).

It is natural to ask whether all rank selections preserve the property of being UMEL-shellable. While the question remains open (see Question~\ref{question:rank-selected} below), we nevertheless show that Theorem~\ref{thm:main1-intro} and Theorem~\ref{thm:main2-intro} are true for any rank selection of a UMEL-shellable poset. The following is among the most general statements we prove in the present article.

\begin{thm}\label{thm:main3-intro}
    Let $\P$ be UMEL-shellable poset and let $\P_S$ be a rank-selected subposet of $\P$. Then the following holds:
    \begin{enumerate}[\normalfont (i)]
        \item The Chow polynomials $\uH_{\P_S}(x)$ and $\uH_{\P_S^*}(x)$ have nonpositive real roots.
        \item The augmented Chow polynomial $\H_{\P_S}(x)$ has nonpositive real roots.
        \item The roots of the polynomials $\uH_{\P_S}(x)$ and $\H_{\P_S}(x)$ interlace.
        \item The $h$-polynomial of the order complex $\Delta(\P_S)$ has nonpositive real roots.
    \end{enumerate}
\end{thm}

Some ideas employed to prove this result are modifications and generalizations of tools used by Athanasiadis, Douvropoulos, and Kalampogia-Evangelinou \cite{athanasiadis-douvropoulos-kalampogia}.

\subsection{Supersolvable lattices}
The class of supersolvable lattices, introduced by Stanley \cite{stanley1972supersolvable}, serves as a far-reaching generalization of both lattices of subgroups of supersolvable finite groups, as well as intersection lattices of fiber-type arrangements (e.g. braid arrangements of type A and B, see \cite{OT_1992}), among many other examples. In this article we will restrict our attention to supersolvable lattices that are also assumed to be semimodular. A semimodular supersolvable lattice can be simply defined as a semimodular lattice containing a maximal chain of modular elements. From this definition one can see that the class of supersolvable lattices encompasses the class of modular lattices and thus in particular all distributive lattices. A classical result of Björner \cite{bjorner-el-shellable} establishes that any choice of a maximal chain of modular elements in a lattice induces an explicit EL-labeling on that lattice, and thus semimodular supersolvable lattices are EL-shellable. 

Rank-uniform semimodular supersolvable lattices were studied by Damiani, D'Antona, and Regonatti~\cite{Damiani1994}. This class of posets contains remarkable families of geometric lattices, prominent examples being all braid matroids of type A and type B, and more generally all Dowling geometries \cite{dowling1973class,dowling1973q}. Our main result concerning rank-uniform semimodular supersolvable lattices is the following. 

\begin{thm}\label{thm:main4-intro}
    Every rank-uniform semimodular supersolvable lattice is UMEL-shellable.
\end{thm}

The above is a central step towards deducing the specializations to matroids discussed below.

\subsection{Specializations to matroids}

As we show in this paper, many famous classes of matroids have lattices of flats that are UMEL-shellable, and are thus covered by our framework. The following is a specialization of Theorems~\ref{thm:main1-intro},~\ref{thm:main2-intro},~and~\ref{thm:main3-intro} to matroids.

\begin{thm}\label{thm:main-matroid-specializations}
    The following classes of matroids have real-rooted Chow polynomials, augmented Chow polynomials, and $h$-polynomials of their Bergman complexes.
    \begin{itemize}[leftmargin=23pt,itemsep=-1pt]
        \item Uniform matroids.
        \item Projective geometries and $q$-niform matroids.
        \item Affine geometries and their rank-selected subposets.
        \item Braid matroids of type A and their rank-selected subposets.
        \item Braid matroids of type B and their rank-selected subposets.
        \item Dowling geometries and their rank-selected subposets (generalizations of the preceding two families).
        \item Supersolvable rank-uniform geometric lattices and their rank-selected subposets (generalizations of the preceding families).
    \end{itemize}
    In particular, Conjectures~\ref{conj:real-rootedness-chow} and \ref{conj:real-rootedness-chain} hold for these matroids.
\end{thm}

As explained earlier, a motivating feature of the class of posets addressed in the present article is that it enables a simultaneous approach for both Conjecture~\ref{conj:real-rootedness-chow} and Conjecture~\ref{conj:real-rootedness-chain}. It is worth mentioning at this point the work by Br\"and\'en and Saud-Maia-Leite \cite{branden-saud1,branden-saud2} and by Br\"and\'en and Vecchi \cite{branden-vecchi2} (the latter written concurrently and independently of the present manuscript). The main results in these works establish, respectively, the real-rootedness of the chain polynomial and the Chow polynomial of certain posets called \emph{TN-posets}. A TN-poset is a \emph{lower rank-uniform} poset; instead of requiring that \emph{upper} intervals of the same rank have the same Whitney numbers of the second kind, they require it for the \emph{lower} intervals. The acronym TN stands for totally nonnegative: it is required that a certain matrix constructed from these lower rank-uniform posets is totally nonnegative. The classes of TN-posets and UMEL-shellable posets appear to share many interesting features (after passing to the dual). It is unclear, however, how these two classes compare. In the special case of geometric lattices, we conjecture (see Conjecture~\ref{conj:tn-geom-lattices-are-umel}) that all TN-geometric lattices are in fact UMEL-shellable, whereas the converse is easily seen to be false (partition lattices are UMEL-shellable but not TN-posets). 

Some of the classes of matroids listed in the preceding theorem were approached in the literature recently by different sets of authors.

\begin{itemize}
[leftmargin=18pt,itemsep=-1pt]
    \item Uniform matroids were proved to have real-rooted augmented Chow polynomials in \cite[Theorem~1.10]{ferroni-matherne-stevens-vecchi}. The real-rootedness of the Chow polynomial of uniform matroids is the main result of Br\"and\'en and Vecchi in \cite{branden-vecchi1}; subsequently a different proof was provided by Hoster and Stump in \cite{hoster-stump}. The real-rootedness of the $h$-polynomial of the Bergman complex of a uniform matroid follows from results of Brenti and Welker \cite{brenti-welker}.
    \item Projective geometries and their truncations (also known as $q$-niform matroids \cite{proudfoot-qniform} or subspace lattices) were proved to have a real-rooted chain polynomial in work of Athanasiadis and Kalampogia-Evangelinou \cite[Proposition~3.1]{athanasiadis-kalampogia}. 
    \item Projective and affine geometries were proved to have real-rooted Chow and augmented Chow polynomials by Br\"and\'en--Vecchi \cite[Section~6]{branden-vecchi2} and real-rooted $h$-polynomials of their Bergman complexes by Br\"and\'en and Saud-Maia-Leite \cite[Theorem~3.7]{branden-saud1}.
    \item Braid matroids of type A and B and in fact all Dowling geometries were proved to have real-rooted augmented Chow polynomials by Br\"and\'en and Vecchi in \cite{branden-vecchi2}. The real-rootedness of the $h$-polynomial of the Bergman complex for braid matroids of type A and B was proved by Athanasiadis and Kalampogia-Evangelinou in \cite[Section~3]{athanasiadis-kalampogia}; an alternative proof for these two classes of matroids and Dowling lattices was provided by Br\"and\'en and Saud-Maia-Leite \cite[Theorem~4.1]{branden-saud1}. 
\end{itemize}

The real-rootedness of the Chow polynomial and the augmented Chow polynomial for type A and type B braid matroids (and more generally, for Dowling geometries) was the original new result that motivated the discovery of UMEL-shellable posets\footnote{After we communicated our discovery, Br\"and\'en--Vecchi noticed an alternative proof of this fact, which is now contained in \cite[Section~6]{branden-vecchi2}).}. The Chow and the augmented Chow polynomial of the type A braid matroid have been in the center of intensive study in recent years, see \cite{gaiffi-serventi,ferroni-matherne-stevens-vecchi,ferroni-matherne-vecchi,liao,stump,eur-ferroni-matherne-pagaria-vecchi,kannan-kuhne}.

\subsection{Structure of the paper} 
In \Cref{sec:definitions}, we recapitulate the main notions on EL-shellable posets, face and flag enumeration on posets, as well as expressions of the polynomials in question using these enumerations. We also include properties of interlacing polynomials, a key tool to prove real-rootedness of single variate polynomials. In \Cref{sec:rank-uniform-labeled-posets}, we present and develop the framework of UMEL-shellable posets, along with technical properties that follow from the definitions. In \Cref{sec:h-RR} and \Cref{sec:chow-RR}, we introduce refinements of $h$-polynomials of order complexes and $\gamma$-polynomials of Chow polynomials into interlacing sequences of polynomials, using the recursive structures given by face and flag enumerations. These two sections contain the proofs of the main real-rootedness results, Theorem~\ref{thm:main1-intro} and Theorem~\ref{thm:main2-intro}. In \Cref{sec:supersolvable}, we show that rank-uniform supersolvable lattices are UMEL-shellable, which enables us to conclude real-rootedness results for those posets. In combination with the results on rank selection, (mainly Theorem~\ref{thm:main3-intro}), presented in \Cref{sec:rank-selection}, we establish the real-rootedness of the foregoing families of polynomials for the matroids listed in Theorem~\ref{thm:main-matroid-specializations}. We also expand upon Theorem~\ref{thm:main1-intro} to \Cref{thm:gamma-truncation-interlacing} by proving an interlacing diagram that illuminates the relationship between the truncation, the dual poset, and the original poset.  We finish the article with a list of open problems and conjectures in Section~\ref{sec:open-questions}.
\subsection*{Acknowledgments} 

We thank the organizers and participants of the special year on Algebraic and Geometric Combinatorics at the Institute for Advanced Study, during which this collaboration was initiated. We benefited from helpful conversations with Spencer Backman, Petter Br\"{a}nd\'{e}n, Mark Goresky, Matt Larson, Jonathan Leake, Jacob Matherne, Karola M\'{e}sz\'{a}ros, Shayan Oveis Gharan, David E Speyer, and Lorenzo Vecchi. We are grateful to Christos Athanasiadis and Petter Br\"{a}nd\'{e}n for suggestions that improved the manuscript. During the early stages of this project, LF was a Minerva Research Fellow at the IAS, and SL was supported by the NSF Grant DMS-1926686 at the IAS. 

\section{Preliminaries}
\label{sec:definitions}
\subsection{Posets and edge labelings}
Let $(\P, \leqslant)$ be a finite graded bounded poset with rank function $\rk \colon \P \to \mathbb{Z}_{\geqslant 0}$. We denote by $\un$ the upper bound of $\P$ and by $\zero$ its lower bound. We denote by $\At(\P)$ the set of elements of rank $1$, called the atoms of $\P$. For all $s\leqslant t \in \P$, the interval between $s$ and $t$ is denoted by 
\[
[s, t] \coloneqq \{ u \in \P \mid s \leqslant u \leqslant t\}. 
\]
If $s < t \in \P$ are such that no $u \in \P$ satisfies $s < u < t$, we say that $s$ is covered by $t$, or $t$ covers $s$, and we write $s \lessdot t$. We denote by $\Edge(\P)$ the set of pairs $(s, t) \in \P \times \P$ such that $s \lessdot t.$ An edge labeling of $\P$ is a map $\lambda: \Edge(\P) \rightarrow \Lambda$, with $\Lambda$ a totally ordered set with total order denoted $\vartriangleleft$. An edge labeling of $\P$ allows us to associate to every saturated chain $C$ of the form 
\[
s = u_0 \lessdot u_1 \lessdot \cdots \lessdot u_k = t 
\] between two elements $s \leqslant t$ the list of labels $\lambda(C) \coloneqq \lambda(u_0, u_1)\ldots \lambda(u_{k-1}, u_k).$

Bj\"{o}rner and Wachs introduced the notion of edge-lexicographical labeling on a finite graded bounded poset $\P$ in \cite{bjorner-el-shellable, bjorner1983lexicographically}. 
\begin{definition}
An edge labeling $\lambda: \Edge(\P) \rightarrow \Lambda$ of $\P$ is called \textit{edge-lexicographical} (EL), if for all $s \leqslant t$ in $\P$, there exists a unique maximal chain $C$
\[
s = u_0 \lessdot u_1 \lessdot \cdots \lessdot u_{k} = t
\]
between $s$ and $t$ with weakly increasing labels as the rank increases: 
\[
\lambda(u_0, u_1) \trianglelefteqslant \lambda(u_1, u_2) \trianglelefteqslant \cdots \trianglelefteqslant \lambda(u_{k-2}, u_{k-1}) \trianglelefteqslant \lambda(u_{k-1}, u_k),
\] and the maximal chain $C$ is lexicographically minimal among all maximal chains from $s$ to $t$. \footnote{For the definition of EL-labelings, different conventions exist in the literature: we can require the unique maximal chain to have either weakly or strictly increasing labels. One must then adapt the definition of descent (Subsection \ref{subsec:flag-enum}) accordingly.  
}
\end{definition}

An \emph{EL-labeled poset} $(\P,\lambda)$ is a pair consisting of a poset $\P$ and an EL-labeling 
\[
\lambda:\Edge(\P)\rightarrow \Lambda.
\] For all $s \leqslant t$ in $\P$ with EL-labeling $\lambda$, the interval $[s,t]$ has a natural EL-labeling $\lambda_{s, t}$ given by the restriction of $\lambda$ to the covering relations in $[s,t]$.

It is well-known that the order complex $\Delta(\P)$ of an EL-labeled poset $(\P, \lambda)$ is shellable, with shelling order on the maximal faces given by the lexicographic order induced by $\lambda$ on the maximal chains of $\P$. For general background on posets and topology, we refer to \cite{wachs2007poset}. 

\subsection{Face enumeration on simplicial complexes}

Let $\Delta$ be a finite simplicial complex of dimension $d-1$. For each integer $i\in \{0,1,\ldots,d\}$ let us denote by $f_i(\Delta)$ the number of simplices in $\Delta$ of cardinality $i$ (that is, of dimension $i-1$). 
The $f$-vector of $\Delta$ is the $(d+1)$-tuple $(f_0(\Delta), f_1(\Delta), \dots, f_{d}(\Delta))$. The $h$-vector of $\Delta$, denoted $(h_0(\Delta),\ldots,h_{d}(\Delta))$, is defined by the following equation
\[
  \sum_{i=0}^d f_{i}(\Delta)(y-1)^{d-i}
  \;=\;
  \sum_{i=0}^d h_i(\Delta) y^{d-i}.
\]
The $f$-polynomial $f(\Delta; x)$ and the $h$-polynomial $h(\Delta; y)$ are defined as 
\[
  f(\Delta; x) \coloneqq \sum_{i=0}^d f_{i} \: x^i,
  \quad \text{ and } \quad
  h(\Delta; y) \coloneqq \sum_{i=0}^d h_i \: y^i.
\]
By definition, the relation between the $f$-polynomial and the $h$-polynomial can be written as follows: 
\[
  h(\Delta; y) \;=\; (1-y)^d \, f\left(\frac{y}{1-y}\right).
\]

\subsection{Flag enumeration on posets}\label{subsec:flag-enum}

In this section, we recall the notion of flag $f$-vectors of posets. For general background on flag vectors and further information on the topics covered in this section, we refer to \cite[Chapter~3]{stanley-ec1}.

Let $\P$ be a finite graded poset with rank function $\rk$, and assume that $\P$ has rank $n$. For a subset $S=\{r_1<\cdots<r_m\}\subseteq \{0,\dots,n\}$, the 
\textit{flag $f$-vector} of $\P$ is the map $\alpha_{\P}:2^{\{0,\ldots,n\}}\to \mathbb{Z}_{\geqslant 0}$ defined by:
\[
\alpha_{\P}(S) \coloneqq 
\abs{\:\left\{s_1 < \cdots < s_m \in \P : \rk(s_i)=r_i \text{ for } 1\leqslant i \leqslant m \,\right\}\:}.
\]
The \textit{flag $h$-vector} of $\P$ is the map $\beta_{\P} \colon 2^{\{0,\dots,n\}}\to \mathbb{Z}$ defined by any of the two equivalent conditions
\[\beta_{\P}(S) \coloneqq \sum_{T\subseteq S} (-1)^{|S|-|T|}\,\alpha_{\P}(T), \qquad \text{or} \qquad
\alpha_{\P}(S) = \sum_{T \subseteq S} \beta_{\P}(T).
\]
If $\P$ has $\widehat{0}$ and $\widehat{1}$, one has:
\[\begin{aligned}
\alpha_{\P}(S) &= \alpha_{\P}(S \cap [n-1]) 
    && \text{for all $S$; and}\\
\beta_{\P}(S)  &= 0 
    && \text{whenever $S \not\subseteq [n-1]$ (i.e., $0\in S$ or $n\in S$).}
\end{aligned}\]
If $\Delta = \Delta(\P)$ is the order complex of a finite graded poset $\P$, then the vectors $\alpha_{\P}$ and $\beta_{\P}$ can be used to enumerate the faces of $\Delta$ according to dimension:
\[
    f_i(\Delta) = \sum_{\abs{S}=i} \alpha_{\P}(S)
    \qquad \text{and} \qquad
    h_i(\Delta) = \sum_{\abs{S} = i} \beta_{\P}(S).
\]
Therefore, the $f$-polynomial and the $h$-polynomial of $\P$ can be written as  
\begin{equation}\label{eq:h-from-flag-h}
f(\Delta; x) = \sum_{i} \left(\sum_{\abs{S} = i} \alpha_{\P}(S) \right) x^i \quad \text{ and } \quad h(\Delta; y) = \sum_{i} \left( \sum_{ \abs{S} = i} \beta_{\P}(S) \right) y^i. 
\end{equation}

If the poset $\P$ admits an EL-labeling $\lambda: \Edge(\P)\rightarrow \Lambda,$ by a result of Stanley appearing in \cite[Theorem~2.7]{bjorner-el-shellable} it is possible to interpret the entries of the flag $h$-vector of $\P$ in terms of statistics of $\lambda$, as we now explain. A sequence of labels 
\[
S = (\lambda_1, \ldots, \lambda_n) \in \Lambda
\] has a \emph{descent} at $i$ for some $i \in [n-1]$, if $\lambda_{i} \vartriangleright \lambda_{i+1}$. The sequence has descent set $\Des(S)$ and descent count $\des(S)$:  
\[
\Des(S) \coloneqq \{i \in [n-1] \mid  S \textrm{ has a descent at }i\}, \quad \text{ and } \quad \des(S) \coloneqq \abs{\Des(S)}. 
\]

\begin{lemma}
\label[lemma]{lem:beta-counts-des}
    Let $\P$ be a finite graded bounded poset of rank $n$ which admits an EL-labeling $\lambda$. For every subset $S \subseteq [n-1]$, the flag $h$-vector entry $\beta_{\P}(S)$ equals the number of maximal chains with descent set $S$ with respect to $\lambda$.  
\end{lemma}

The above result (proved in the more general context of R-labelings) can also be found in \cite[Theorem~3.14.2]{stanley-ec1}.

\subsection{Chow polynomials via flag enumeration}\label{sec:chow-poly-recap}

The Chow polynomial $\uH_{\M}(x)$ of a matroid $\M$ was defined in \cite{ferroni-matherne-stevens-vecchi} as the Hilbert--Poincar\'e series of an Artinian graded algebra $\uCH(\M)$, called the \emph{Chow ring} of $\M$. The Chow ring of a matroid $\M$ was introduced by Feichtner and Yuzvinksy \cite{feichtner2004chow} as the Chow ring of the smooth and generally noncomplete toric variety associated with the Bergman fan of $\M$. Similarly, the augmented Chow polynomial $\H_\M(x)$ was defined as the Hilbert--Poincar\'e series of another Artinian graded algebra $\CH(\M)$, called the \emph{augmented Chow ring} of $\M$. The augmented Chow ring of a matroid $\M$ is introduced by \cite{braden2020singular} as the Chow ring of the smooth and generally noncomplete toric variety associated with the augmented Bergman fan of $\M$. 

In \cite{ferroni-matherne-vecchi} a more general definition of Chow polynomial for finite bounded graded posets is introduced. Precisely, the Chow polynomial of a poset $\P$ is defined recursively by $\H_{\P}(x) = 1$ whenever $|\P| = 1$, and
    \[ \uH_{\P}(x) = \sum_{\substack{s\in \P\setminus\{\zero\}}} \overline{\chi}_{\widehat{0},s}(x)\, \uH_{s,\widehat{1}}(x)\]
otherwise. Here, $\overline{\chi}_{st}(x)$ stands for the \emph{reduced} characteristic polynomial of an interval $[s,t]\subseteq \P$. The augmented Chow polynomial $\H_{\P}(x)$ is defined, correspondingly, as:
    \[ \H_{\P}(x) = \sum_{s\in \P} x^{\rk(s)}\, \uH_{s,\widehat{1}}(x).\]

A non-obvious property of Chow polynomials and augmented Chow polynomials is that their coefficients always form a nonnegative unimodal palindromic sequence (see \cite[Theorem~1.3]{ferroni-matherne-vecchi}). In the special case of matroids, the nonnegativity is a consequence of the (augmented) Chow polynomial being a Hilbert--Poincar\'e series, whereas the symmetry and the unimodality follow from Poincar\'e duality and the hard Lefschetz theorem proved for (augmented) Chow rings, see \cite{adiprasito2018hodge,BRADEN2022108646}. 

Thanks to the following result by Ferroni, Matherne, and Vecchi, the Chow and the augmented Chow polynomial of a poset can be read off of the flag $h$-vector of stable sets. A set $S \subseteq \Z$ is \emph{stable} if there exists no integer $i$ such that $i \in S$ and $i + 1 \in S$. \footnote{The etymology of ``stable'' is as follows: in a graph, a stable set of vertices is a collection of vertices no two of which are adjacent. In the Hasse diagram of the poset $\mathbb{Z}$ with its natural order, the stable sets are precisely those subsets that contain no pair of consecutive integers.}

\begin{thm}[{\cite[Theorem~4.25]{ferroni-matherne-vecchi}}]
\label[theorem]{thm:gamma-flag-h-vector}
    Let $\P$ be a finite graded bounded poset of rank $n$. Then, 
    \begin{align*} \uH_{\P}(x) &= \sum_{\substack{S\subseteq \{2,\ldots,n-1\}\\S\text{ stable}}} \beta_{\P}(S)\, x^{|S|}\,(1+x)^{n-1-2|S|},\\
    \H_{\P}(x)&=\sum_{\substack{S\subseteq \{1,\ldots,n-1\}\\S\text{ stable}}} \beta_{\P}(S)\, x^{|S|}\,(1+x)^{n-2|S|}.
    \end{align*}
\end{thm}

The preceding result can be recast in terms of the so-called $\gamma$-polynomials. Recall that a polynomial $f(x) = \sum_{i=0}^d a_i\, x^i$ (where we allow $a_d=0$) is said to be \emph{symmetric} (or \emph{palindromic}) with center of symmetry $d/2$ if $a_i = a_{d-i}$ for each $i\in \mathbb{Z}$ (here we are implicitly defining $a_i=0$ if $i<0$ or $i>d$). This symmetry condition can be expressed succinctly by the equation $f(x) = x^{d}\, f(x^{-1})$. 

A linear basis for symmetric polynomials with center of symmetry $d/2$ is given by the collection $\left\{x^i(1+x)^{d-2i}\right\}_{i=0}^{\lfloor d/2\rfloor}$. Therefore, if $f(x)$ is symmetric with center of symmetry $d/2$, we can write
    \[ f(x) = \sum_{i=0}^{\lfloor d/2\rfloor} \gamma_i\, x^i (1+x)^{d-2i},\]
for some unique real numbers $\gamma_0,\ldots,\gamma_{\lfloor d/2\rfloor}$.  The \emph{$\gamma$-polynomial} associated to $f(x)$ is the univariate polynomial 
\[
\gamma(f;y) \coloneqq \sum_{i=0}^{\floor{d/2}} \gamma_{i} \: y^{i}. 
\]
If the $\gamma$-polynomial of $f(x)$ has nonnegative coefficients, we say that $f(x)$ is \emph{$\gamma$-positive}. A routine check shows that if $f(x)$ is $\gamma$-positive, then it is also unimodal. In particular, Theorem~\ref{thm:gamma-flag-h-vector} implies that if the flag $h$-vector of a poset is nonnegative, then the Chow and the augmented Chow polynomials are $\gamma$-positive. 

\smallskip

A very useful application of Theorem~\ref{thm:gamma-flag-h-vector} is the following formula for Chow and augmented Chow polynomials of poset duals.

\begin{lemma}[Duality Lemma]\label[lemma]{lemma:duality}
    For any finite graded bounded poset $\P$, the following formulas hold:
    \[ \uH_{\P^*}(x) = \H_{\trunc(\P)}(x) \qquad \text{ and } \qquad \H_{\P^*}(x) = \H_{\P}(x). \]
    Equivalently,
    \[ \gamma(\uH_{\P^*};y) = \gamma(\H_{\trunc(\P)};y) \qquad \text{ and } \qquad \gamma(\H_{\P^*};y) = \gamma(\H_{\P};y). \]
\end{lemma}

\begin{proof}
Let $n$ be the rank of $\P$. In the non-augmented case,  

\begin{align*}   
\gamma(\uH_{\Pdual};y) &= \sum_{\substack{S \textrm{ stable}\\ 1 \notin S}}\beta_{\Pdual}(S)y^{|S|} \\
    &= \sum_{\substack{S \textrm{ stable}\\ 1 \notin S}}\beta_{\P}(n-S)y^{|S|} 
    =\sum_{\substack{S \textrm{ stable}\\ n-1 \notin S}}\beta_{\P}(S)y^{|S|} 
    = \gamma(\H_{\Tr(\P)}; y).
\end{align*}
In the augmented case,  
\begin{align*}
    \gamma(\H_{\Pdual};y) 
    &= \sum_{S \textrm{ stable}}\beta_{\Pdual}(S)y^{|S|}\\  
    &= \sum_{S \textrm{ stable}}\beta_{\P}(n-S)y^{|S|} 
    =\sum_{S \textrm{ stable}}\beta_{\P}(S)y^{|S|} 
    = \gamma(\H_{\P}; y).\qedhere
\end{align*} 
\end{proof}

\subsection{Interlacing polynomials}
Let $f$ be a monic real-rooted polynomial of degree $n$ and $g$ a monic real-rooted polynomial of degree $n$ or $n+1$, with ordered roots
\[
\alpha_n \leqslant \cdots \leqslant \alpha_1, \quad (\beta_{n+1} \leqslant) \beta_n \leqslant \cdots \leqslant \beta_1
\]
respectively (ignoring $\beta_{n+1}$ if $g$ is of degree $n$). We say that $f$ \emph{interlaces} $g$, denoted by $f \preceq g$, if
\[
(\beta_{n+1} \leqslant \: ) \: \alpha_n \leqslant \beta_n \leqslant \alpha_{n-1} \leqslant \beta_{n-1} \leqslant \cdots \leqslant \alpha_1 \leqslant \beta_1.
\]
We say that the \emph{zeros of $f$ and $g$ interlace} if either $f \preceq g$, or $g \preceq f$. By convention, the zeros of any two polynomials of degree $0$ or $1$ interlace, $0 \preceq f, f \preceq 0$ and $0 \preceq 0$ for all real-rooted polynomials $f$ of arbitrary degree. 

A classical analytic criterion of interlacing polynomials is the following positivity condition: If $g \preceq f$, then the Wronskian of $f$ and $g$ is nonnegative over all real numbers
\[
\textrm{W}[f, g] \coloneqq \det \begin{pmatrix} 
f' & g' \\
f  & g 
\end{pmatrix} \geqslant 0. 
\] The converse is true if the multiplicity of every root of $f$ and $g$ is $1$. 

The main motivation behind the notion of interlacing polynomials is the following result, often called Obreschkoff's theorem, which allows us to produce new real-rooted polynomials using interlacing sequences of polynomials. 

\begin{thm}[{\cite[Satz 5.2]{obreshkov1963verteilung}}]\label{thm:sum-of-interlacing-is-RR}
    Let $f, g$ be nonzero real-rooted polynomials. The zeros of $f$ and $g$ interlace if and only if any polynomial in the linear space spanned by $f, g$
    \[
    \left\{\: \alpha \: f + \beta \: g \mid \alpha, \beta \in \R \: \right\}
    \] only has real roots. 
\end{thm} 

This result will be instrumental in the sequel. The following properties of interlacing polynomials are due to Wagner. 

\begin{proposition}[\cite{wagner1992total}]\label[proposition]{prop:interlacing-properties}
Let $f, g, h$ be real-rooted polynomials in $\R_{\geqslant 0}[x]$. The following properties hold: 
\begin{enumerate}[\normalfont (i)]
    \item (Convexity) If $f \preceq g, h$, then for any $\alpha, \beta \in \R_{\geqslant 0}$, $f \preceq \alpha g + \beta h$. 
    \item (Convexity) If $g, h \preceq f$, then for any $\alpha, \beta \in \R_{\geqslant 0}$, $\alpha g + \beta h \preceq f$. 
    \item (Adding zero) If $f \preceq g$, then $g \preceq x f$. 
    \item (Interpolation) If $f \preceq g$, then for all $\alpha, \beta \in \R_{\geqslant 0}$, $f \preceq \alpha f + \beta g \preceq g$. 
\end{enumerate}
\end{proposition}

\begin{definition}
A family of real-rooted polynomials $F_n = (f_1, \ldots, f_n)$ with nonnegative coefficients is an \emph{interlacing family} or \emph{interlacing sequence} if and only if $f_{i} \preceq f_j$ for all $1 \leqslant i < j \leqslant n$. 
\end{definition}

Our main tool to produce new interlacing sequences out of known interlacing sequences is the following more general theorem about linear operators preserving sequences of interlacing polynomials with nonnegative coefficients. 
\begin{thm}[{\cite[Theorem~7.8.5]{branden2015unimodality}}]
\label{thm:interlacing}
    Let $\mathbf{G}(x) = \begin{pmatrix}G_{ij}(x) \end{pmatrix}$ be an $m\times n$ matrix of real polynomials $G_{ij}(x)$. Then multiplication by $\mathbf{G}(x)$ preserves interlacing sequences of polynomials with nonnegative coefficients if and only if 
    \begin{enumerate}[\normalfont (i)]
        \item each polynomial $G_{ij}(x)$ has positive coefficients for all $i \in [m]$ and all $j \in [n]$, and 
        \item for all $1 \leqslant i_1 < i_2 \leqslant m$ and all $1 \leqslant j_1 < j_2 \leqslant n$ we have the interlacing relation 
        \begin{equation*}
            (p x + q)G_{i_1j_2}(x) + G_{i_2j_2}(x) \preceq (p x + q)G_{i_1j_1}(x) + G_{i_2 j_1}(x)
        \end{equation*}
        for all real numbers $p, q >0.$
    \end{enumerate}
\end{thm}

We provide a short proof of the following lemma, a slight generalization of  \cite[Corollary 7.8.7]{branden2015unimodality}, which shows that certain matrices of specific shapes satisfy the hypotheses of Theorem~\ref{thm:interlacing}. Those matrices will appear in the proofs of Theorem~\ref{thm:main2-intro} and Theorem~\ref{thm:gamma-RR-mono-rank-up-uniform}. 
\begin{lemma}\label[lemma]{lem:two-column-monotonicity-interlacing}
    Let $m$ and $n$ be positive integers, and let $d_1 \leqslant d_2 \leqslant \cdots \leqslant d_{m}$ be a weakly increasing sequence of integers. 
    Let $\mathbf{A}(x)$ be an $m \times n$ matrix with polynomial entries defined as 
    \[
    \mathbf{A}(x) \coloneqq \begin{pmatrix} a_{ij}(x) \end{pmatrix}, \text{ where } a_{ij}(x) = \begin{cases}
    \alpha_{i} \cdot x & \text{ if } j \leqslant d_i, \\
    \alpha_{i} & \text{ if } j > d_i, 
    \end{cases} \quad \alpha_i \in \R_{\geqslant 0}.
    \] 
    For any fixed column number $1 \leqslant j \leqslant n$, any two rows $1 \leqslant i < i' \leqslant m$, and nonnegative real numbers $p, q$, we write $f_j(x)$ for the polynomial 
    \begin{align*}
        f_{j}(\mathbf{A}; p, q, i, i'; x) \coloneqq (p \: x + q ) \:  a_{i \: j}(x) + a_{i' \: j}(x). 
    \end{align*}
    Then for any two column numbers $1 \leqslant j < j' \leqslant n$, and any nonnegative real numbers $p, q \in \R_{\geqslant 0}$,  we have that  
    \[
    f_{j'}\preceq f_{j}.
    \]
\end{lemma}

\begin{proof}
    By monotonicity of the sequence $(d_k)_{k=1}^{m}$, we have $d_i \leqslant d_{i'}$. 
    By definition of the matrix $A$, there are $5$ cases, depending on the relative positions of $j < j'$ with respect to $d_i \leqslant d_{i'}$. For each of the cases, we fix $\alpha_{i}, \alpha_{i'} \in \R_{\geqslant 0}$, and consider the $2 \times 2$ submatrix 
    \[
    \mathbf{A}_{i, \:i',\: j,\:j'} \coloneqq 
    \begin{pmatrix}
        a_{i\: j} & a_{i \: j'} \\
        a_{i' \: j} & a_{i' \: j'}
    \end{pmatrix}. 
    \]
    \begin{enumerate}[\normalfont(i), leftmargin=18pt]
        \item If $j < j' \leqslant d_i \leqslant d_{i'}$, then we have 
        \[
        \mathbf{A}_{i, \:i',\: j,\:j'}(x) = 
        \begin{pmatrix}
        \alpha_i \: x & \alpha_i \: x  \\
        \alpha_{i'} \: x & \alpha_{i'} \: x
        \end{pmatrix}.
        \]
        This implies that $f_{j} = f_{j'}$ for all $p, q \in \R_{> 0}$. 
        \item If $j \leqslant  d_i < j' \leqslant d_{i'}$, then we have 
        \[
        \mathbf{A}_{i, \:i',\: j,\:j'}(x) = 
        \begin{pmatrix}
        \alpha_i \: x & \alpha_i  \\
        \alpha_{i'} \: x & \alpha_{i'} \: x
        \end{pmatrix}.
        \]
        This implies that 
        \begin{align*}
            f_{j}(x) &= (px + q) \alpha_ix + \alpha_{i'} x, \text{ with roots } - \frac{\alpha_{i}q + \alpha_{i'}}{\alpha_{i}p}, \text{ and }0,\\
            f_{j'}(x) &= (px + q) \alpha_i + \alpha_{i'} x, \text{ with root } - \frac{\alpha_{i} q}{\alpha_{i}p + \alpha_{i'}}.\\
        \end{align*}
        This implies that $f_{j'} \preceq f_{j}$ for all $p, q \in \R_{> 0}$. 

        \item If $d_{i} < j < j' \leqslant d_{i'}$, then we have
        \[
        \mathbf{A}_{i, \:i',\: j,\:j'}(x) = 
        \begin{pmatrix}
        \alpha_i  & \alpha_i  \\
        \alpha_{i'} \: x & \alpha_{i'} \: x
        \end{pmatrix}.
        \]
        This implies that $f_{j'}(x) = f_{j}(x)$ for all $p, q \in \R_{> 0}$. 

        \item If $d_i < j \leqslant d_{i'} < j'$, then we have  
        \[
        \mathbf{A}_{i, \:i',\: j,\:j'}(x) = 
        \begin{pmatrix}
        \alpha_i  & \alpha_i  \\
        \alpha_{i'} \: x & \alpha_{i'}
        \end{pmatrix}.
        \]
        This implies that 
        \begin{align*}
            f_{j}(x) &= (px+q) \alpha_{i} + \alpha_{i'} x, \text{ with root} - \frac{\alpha_{i} q}{\alpha_{i} p + \alpha_{i'}},\\
            f_{j'}(x) &= (px+q) \alpha_{i} + \alpha_{i'}, \text{ with root} -\frac{\alpha_{i}q + \alpha_{i'}}{\alpha_{i}p}.  
        \end{align*}
        This implies that $f_{j'} \preceq f_{j}$ for all $p, q \in \R_{> 0}$.

        \item If $d_i \leqslant d_{i'} < j < j'$, then we have 
        \[
        \mathbf{A}_{i, \:i',\: j,\:j'} =
        \begin{pmatrix}
        \alpha_i  & \alpha_i  \\
        \alpha_{i'} & \alpha_{i'}
        \end{pmatrix}.
        \]
        This implies that $f_{j'} = f_{j}$ for all $p, q \in \R_{> 0}$. \qedhere
    \end{enumerate}
\end{proof}
 
\medskip
We also have the following useful corollary of Theorem~\ref{thm:interlacing}. 

\begin{corollary}[{\cite[Corollary 7.8.6]{branden2015unimodality}}]
\label[lemma]{lem:two-minors-monotonicity-interlacing}
    Let $A$ be an $m \times n$ matrix with nonnegative real entries. Then $A$ preserves interlacing families of polynomials with nonnegative coefficients if and only if all two-by-two minors of $A$ are nonnegative.
\end{corollary}

Finally, we have the following two results stating that the $\gamma$-transformation preserves real-rootedness and interlacing polynomials. 
\begin{lemma}[\cite{branden2004sign}, \cite{gal2005real}]\label[lemma]{lem:gamma-transfo-RR}
    If $f(x)$ is a symmetric polynomial then the following are equivalent:
    \begin{enumerate}[\normalfont(i)]
        \item All the roots of $f(x)$ are real and nonpositive.
        \item All the roots of $\gamma(f; y)$ are real and nonpositive.
    \end{enumerate}
\end{lemma}

\begin{lemma}[{\cite[Proposition 2.5]{hoster-stump}}]\label[lemma]{lem:gamma-transform-preserves-interlacing}
    If $f(x), g(x)$ are palindromic polynomials with degree $\deg(g) = \deg(f)+1$, then the following are equivalent: 
    \begin{enumerate}[\normalfont(i)]
        \item $f(x) \preceq g(x)$. 
        \item $\gamma(f; y) \preceq \gamma(g; y)$. 
    \end{enumerate}
\end{lemma}
Using these two lemmas above, we will prove our main theorem about real-rootedness of Chow polynomials by proving the corresponding statements for their $\gamma$-polynomials in \Cref{sec:chow-RR}. This will be done by expressing those polynomials as nonnegative sums of certain sequences of polynomials, and then proving that those sequences are interlacing by proving that they satisfy a linear recursion satisfying the hypotheses of Theorem~\ref{thm:interlacing}. 

\section{Rank-uniform labeled posets}\label{sec:rank-uniform-labeled-posets}

Let $\P$ be a bounded finite graded poset with rank function $\rk \colon \P \to \mathbb{Z}_{\geqslant 0}$. For every element $s$ and for every nonnegative integer $k$, we write $W_{k}(s)$ for the number of elements of rank $k + \rk s$ above $s$, called the \emph{$k$-th Whitney number above $s$}. The poset $\P$ is \textit{rank-uniform} if for all $s, t \in \P$ of the same rank and for every positive integer $k$, the number of elements of rank $k + \rk s$ above $s$ is equal to the number of elements of rank $k + \rk s$ above $t$: 
\begin{equation*}
    W_k(s) = W_{k}(t).
\end{equation*}

In this section we will introduce an edge-labeled version of the above definition, which applies to a poset equipped with an EL-labeling. We start with some terminology related to EL-labelings. 

\begin{definition}[Width, index, descent]Let $(\P, \lambda)$ be an EL-labeled poset and let $s \in \P$. Denote by $\lambda_1 \vartriangleleft \cdots \vartriangleleft \lambda_{\ell}$ the labels of the covering relations above $s$.
\begin{itemize}[leftmargin=23pt,itemsep=-1pt]
    \item The \emph{width above $s$} is the function 
    \[
    \omega_s \colon \mathbb{Z}_{>0} \to \Z_{\geqslant 0}, \quad \omega_s(i) \coloneqq \abs{ \: \{t \in \P \mid s \lessdot t\: , \: \lambda(s, t) = \lambda_i \: \} \: }. 
    \] 
    In other words, the width above $s$ at index $i$ records the number of covers of $s$ labeled by $\lambda_i$ (see Figure \ref{fig:width} below). By convention, for $i > \ell$, $\omega_{s}(i) = 0$.
     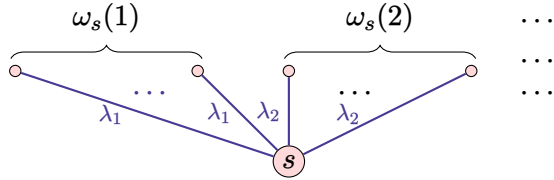
\begin{figure}[H]
         \centering
         \begin{tikzpicture}[scale=0.60, baseline=14px]
         \node[circle, draw=black, fill=red!15, inner sep=1.7pt] (A) at (0,0) {$s$};
         \node[circle, draw=black, fill=red!15, inner sep=1.5pt] (B) at (-6,2) {};
        \node[circle, draw=black, fill=red!15, inner sep=1.5pt] (C) at (-2,2) {};
         \node[circle, draw=black, fill=red!15, inner sep=1.5pt] (D) at (0,2) {};
         \node[circle, draw=black, fill=red!15, inner sep=1.5pt] (E) at (4,2) {};
        
         \draw[thick, BlueViolet] (A) -- (B) node[pos=0.5, left=6pt, scale = 0.8] {$\lambda_1$};
        \draw[thick, BlueViolet] (A) -- (C) node[midway,left, scale = 0.8] {$\lambda_1$};
         \draw[thick, BlueViolet] (A) -- (D) node[midway,left, scale = 0.8] {$\lambda_2$};
         \draw[thick, BlueViolet] (A) -- (E) node[pos=0.5, left=6pt, scale = 0.8] {$\lambda_2$};
         \node[thick, BlueViolet] (F) at (-3, 1.5) {$\cdots$};
         \node[] (G) at (1.5, 1.5) {$\cdots$};
         \node[] (H) at (5.5, 1.5) {$\cdots$};
          \node[] (I) at (5.5, 2.2) {$\cdots$}; 
          \node[] (J) at (5.5, 3.1) {$\cdots$};
         \draw [decorate, decoration={brace, amplitude=5pt, mirror=false}] ([yshift=5pt]B.north west) -- ([yshift=5pt]C.north east) node [midway, above=6pt] {$\omega_{s}(1)$};
         \draw [decorate, decoration={brace, amplitude=5pt, mirror=false}] ([yshift=5pt]D.north west) -- ([yshift=5pt]E.north east) node [midway, above=6pt] {$\omega_{s}(2)$};
     \end{tikzpicture}
         \caption{Widths above an element $s$.}
         \label{fig:width}
    \end{figure}
    
    \item Let $t$ cover $s$ in $\P$. We denote the \emph{index} of the label $\lambda(s, t)$ in $[\ell]$ by 
    \[
    \ind(s, t) \coloneqq \text{ the unique index $i \in [\ell]$ such that } \lambda_i =    \lambda(s, t). 
    \] 

    \item Let $\lambda_1' \vartriangleleft \cdots \vartriangleleft \lambda_{\ell'}'$ be the labels of the covering relations above $t$. We denote by $\Des_{s}(t)$ the set of indices $j$ in $[\ell']$ such that $\lambda'_j \vartriangleleft \lambda(s, t)$, that is, which gives rise to a descent, and we write $\des_{s}(t)$ for the cardinality of that set. 
    \begin{equation*}
        \Des_{s}(t) \coloneqq \{\: j \in [\ell']\mid \lambda_j' \vartriangleleft \lambda(s, t)\}, \quad \text{ and } \quad \des_{s}(t) \coloneqq \abs{\:\Des(s, t)\:}.  
    \end{equation*}
\end{itemize}
\end{definition}

The central object of study in this article is the following class of posets.

\begin{definition}[Rank-uniform labeled poset]
Let $(\P, \lambda)$ be an EL-labeled poset. We say that $(\P, \lambda)$ is \textit{rank-uniform} if the following holds.
\begin{itemize}
    \item (Uniform width) For all $s, s' \in \P$ of same rank and for all $i \geqslant 1$, the widths are equal: $\omega_s(i) = \omega_{s'}(i).$
    \item (Uniform descent) For all $s, s' \in \P$ of same rank and for every covering relations $s \lessdot t$ and $s' \lessdot t'$ such that $\ind(s,t) = \ind(s',t')$, the descent counts are equal: $\des_{s}(t) = \des_{s'}(t').$
\end{itemize}
\end{definition}
If $(\P, \lambda)$ is rank-uniform, for any $0 \leqslant k \leqslant \rk(\P)$ and any $i \geqslant 1$ we denote $\omega_k(i) \coloneqq \omega_s(i)$ for any $s \in \P$ of rank $k$. Similarly we denote $\des_k(i) \coloneqq \des_{s}(t)$ for any $s$ of rank $k$ and any $t \gtrdot s$ such that $\Ind(s, t) = i$. In other words, $\des_k(i)$ records the number of indices of descents immediately above rank $k$ in $\P$ which starts with the $i$-th label among all labels at rank $k$.

\begin{example}\label{ex:diamond-poset}
Consider the labeled poset $(\P, \lambda)$ whose labeled Hasse diagram is drawn in Figure \ref{fig:example-condition} below. 
    \begin{figure}[H]
    \centering
    \begin{tikzpicture}[xscale=2,every path/.style={line width=0.8pt}]
        \node[circle, draw=black, fill=red!15, inner sep=1.5pt] (A) at (0,0) {};
        \node[circle, draw=black, fill=red!15, inner sep=1.5pt] (B) at (-0.5, 1) {};
        \node[circle, draw=black, fill=red!15, inner sep=1.5pt] (C) at (0.5, 1) {};
        \node[circle, draw=black, fill=red!15, inner sep=1.5pt] (D) at (-1, 2) {};
        \node[circle, draw=black, fill=red!15, inner sep=1.5pt] (E) at (0, 2) {};
        \node[circle, draw=black, fill=red!15, inner sep=1.5pt] (F) at (1, 2) {};
        \node[circle, draw=black, fill=red!15, inner sep=1.5pt] (G) at (0, 3.5) {};
    
        \draw[thick, BlueViolet] (A) -- (B) node[pos=0.5, left, scale = 0.8] {$1$};
        \draw[thick, BlueViolet] (A) -- (C) node[pos=0.5, right, scale = 0.8] {$2$};

        \draw[thick, BlueViolet] (B) -- (D) node[pos=0.5, left, scale = 0.8] {$1$};
        \draw[thick, BlueViolet] (B) -- (E) node[pos=0.5, left, scale = 0.8] {$2$};
        \draw[thick, BlueViolet] (C) -- (E) node[pos=0.5, right, scale = 0.8] {$1$};
        \draw[thick, BlueViolet] (C) -- (F) node[pos=0.5, right, scale = 0.8] {$2$};
        \draw[thick, BlueViolet] (D) -- (G) node[pos=0.5, left=2pt, scale = 0.8] {$1$};
        \draw[thick, BlueViolet] (E) -- (G) node[pos=0.5, left, scale = 0.8] {$1$};
        \draw[thick, BlueViolet] (F) -- (G) node[pos=0.5, right = 2pt, scale = 0.8] {$1$};
    \end{tikzpicture}
    \caption{A rank-uniform labeled poset.}
    \label{fig:example-condition}
\end{figure}
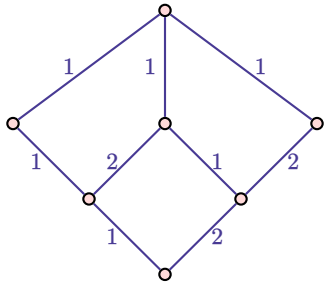
One can check that $\lambda$ is an EL-labeling. For all element $s \in \P$ of rank $0$ or $1$ we have $\omega_{s}(1) = 1$, $\omega_s(2) = 1$ and $\omega_{s}(i) = 0$ for all $i> 2.$ Moreover, for all element $s\in \P$ of rank $2$ we have $\omega_{s}(1) = 1$ and $\omega_{s}(i) = 0$ for all $i >1$ (which is more generally true for any $s$ of corank $1$ in a labeled poset). This proves that $\P$ has uniform widths. 

For all element $s \in \P$ of rank $0$ or $1$ we have $\des_s(1) = 0$, $\des_s(2) = 1$ and $\des_s(i) = 0$ for all $i \geqslant 2.$ Moreover, for all element $s\in \P$ of rank $2$ we have $\des_s(i) = 0$ for all $i \geqslant 1$ (which is more generally true for any $s$ of corank $1$ in a labeled poset). This proves that $\P$ has uniform descents as well, and so $(\P, \lambda)$ is a rank-uniform labeled poset. 
\end{example}
\begin{example}[Uniform matroids]\label{ex:uniform-matroid}
Let $n \geqslant k \geqslant 1$ be integers, and consider the poset given by all the subsets of $[n]$ of cardinality $\leqslant k-1$ or cardinality $n$, ordered by inclusion. This is the lattice of flats of the rank $k$ uniform matroid of size $n$, denoted $\L(\mathrm{U}_{k,n})$. This poset has an EL-labeling $\lambda: \Edge(\L(\mathrm{U}_{k,n})) \rightarrow [n]$ defined by 
$$ \lambda(S \lessdot T ) \coloneqq \min T\setminus S.$$ Let us check that $(\L(\mathrm{U}_{k,n}), \lambda)$ is a rank-uniform labeled poset. For all element $s \in \L(\mathrm{U}_{k,n})$ of rank $p \leqslant k-2$ and all $i\geqslant 1$, we have $\omega_{s}(i) = 1$ if $i \leqslant n-k$ and $\omega_s(i) = 0$ otherwise. This proves uniformity of widths.

For every element $s \in \L(\mathrm{U}_{k,n})$ of rank $p \leqslant k-3$ and all $i\geqslant 1,$ we have $\des_{s}(i) = i-1$ if $i \leqslant n-k$ and $\des_s(i) = 0$ otherwise. This proves uniformity of descents, since the uniformity of descents at rank $\geqslant k-2$ is trivially verified. 
\end{example}

In \cite{dowling1973class, dowling1973q}, Dowling introduced a generalization of partition lattices called Dowling geometries. We highlight this class and its EL-labeling in the example below. Dowling geometries are examples of rank-uniform supersolvable lattices, which we discuss in detail in \Cref{sec:supersolvable}.

\begin{example}[Dowling geometries]\label{ex:dowling-geo} Let $n  \geqslant 1 $ be a positive integer, and let $\G$ be a finite group. A \emph{$\G$-prepartition} of $[n]$ is a pair $(\mu,\varphi)$ consisting of a partition 
\[
0\sqcup [n] = \mu_0 \sqcup \mu_1 \sqcup \cdots \sqcup \mu_k,
\] with $0 \in \mu_0$, and a labeling function 
\[
\varphi \colon \mu_1 \sqcup \cdots \sqcup \mu_k \rightarrow \G.
\] The unique block $\mu_0$ containing $0$ is called the zero block and the others the non-zero blocks. $\G$-prepartitions are ordered by setting $(\mu, \varphi) \geqslant(\mu', \varphi')$ if every non-zero block of $\mu$ is a union of non-zero blocks of $\mu'$ and the labelings coincide. Two $\G$-prepartitions $(\mu, \varphi)$ and $(\mu', \varphi')$ are said to be equivalent if $\mu = \mu'$ and for each block $\mu_i = \mu'_i$ there exists an element $g\in \G$ such that we have the equality $\varphi_{|\mu_i} = g\cdot \varphi'_{|\mu'_i}$.  A $\G$-partition is an equivalence class of $\G$-partitions, and the Dowling geometry $\QnG$ is the set of $\G$-partitions of $[n]$ with the order induced by $\leqslant$. For instance, if $\G$ is the trivial group then a $\G$-partition is simply a partition of $0 \sqcup [n]$, and $\QnG$ is the $(n+1)$-th partition lattice $\Pi_{n+1}.$ More generally, if $\G$ is a finite cyclic group then $\QnG$ is the intersection lattice of the arrangement consisting of the hyperplanes defined by equations of the form 
\begin{align*}
\{ \: x_i = g \: x_j \mid 1 \leqslant i < j \leqslant n, g \in G \: \} \cup \{\: x_k = 0 \mid 1 \leqslant k \leqslant n \: \}. 
\end{align*}
For instance, if $\G = \Z/2\Z$, we get the $n$-th type B braid arrangement. 

Consider the poset $\Pi^{\textrm{B}}_3$ whose Hasse diagram is drawn in Figure \ref{tikz:pi-3-B}. This is the intersection lattice of the type B braid arrangement of rank $3$, defined by the hyperplanes \[
\bigg\{ H_{ij}^{\pm} \coloneqq V(x_i \pm x_j) \mid 1 \leqslant i < j \leqslant 3 \bigg\} \cup \bigg\{ H_k \coloneqq V(x_k) \mid 1 \leqslant k \leqslant 3 \bigg\}. 
\]  More concretely, the poset $\Pi^{\textrm{B}}_3$ consists of all the intersections of these hyperplanes, ordered by reverse inclusion, and ranked by codimension: 
\begin{enumerate}[(i)]
    \item In rank $0$, the only intersection space is the full space $\C^3$. 
    \item In rank $1$, the intersection spaces are the $9$ hyperplanes of the form $H_i$ and $H_{jk}^{\pm}$. 
    \item In rank $2$, the intersection spaces are $13$ lines: 
    \begin{itemize}[--]
    \item the $3$ coordinate lines: $H_{i} \cap H_j$.  
    \item the $6$ lines that are the intersections of one coordinate hyperplane with one of the diagonals: $H_k \cap H_{ij}^{\pm}$,
    \item the $4$ lines that are the intersections of pairs of diagonals $H_{ij}^{\pm} \cap H_{jk}^{\pm}$. 
    \end{itemize}
    \item In rank $3$, the only intersection space is the origin. 
\end{enumerate}

These intersections are encoded using signed partitions as follows: the zero block contains indices $i$ with $x_i = 0$ and the nonzero block contains indices whose coordinates are equal up to a common sign shift. For example, the block $03$ represents the hyperplane given by $x_3 = 0$, and the block $1 \overline{2}$ represents the hyperplane given by $x_1 = - x_2$. Therefore, $03\mid 1\overline{2}$ corresponds to the intersection space $H_{3} \cap H_{12}^{-}$. 
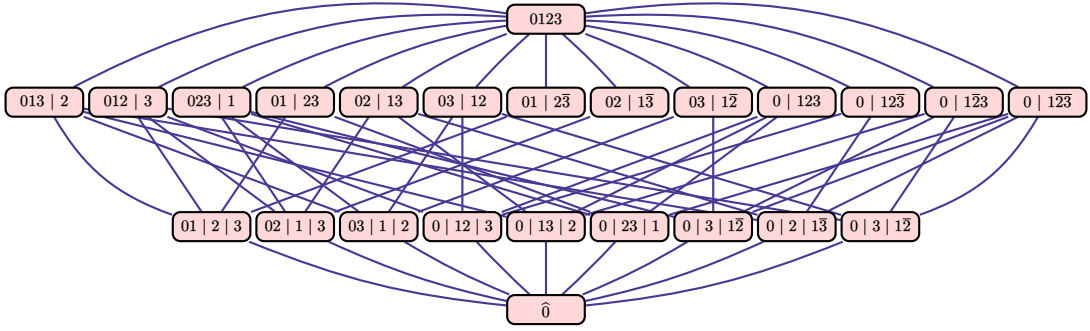
\begin{figure}[h!]
\begin{center}
\begin{tikzpicture}[scale=0.55, transform shape, node distance=16mm and 20mm, xshift=-2cm,
  bubble/.style={
    draw,
    rounded corners=1mm,
    text width=16mm,
    minimum height=7mm,
    align=center,
    fill=red!15,
    font = \normalfont,
  },every path/.style={line width=0.8pt}]
  \node[bubble] (bot) at (0, 0) {$\zero$};
  \node[bubble] (rk1-1) at (-8, 2) {$01 \mid 2 \mid 3$};
  \node[bubble] (rk1-2) at (-6, 2) {$02 \mid 1\mid 3$};
  \node[bubble] (rk1-3) at (-4, 2) {$03 \mid 1 \mid 2$};
  \node[bubble] (rk1-4) at (-2, 2) {$0 \mid 12 \mid 3$};
  \node[bubble] (rk1-5) at (0, 2) {$0 \mid 13 \mid 2$};
  \node[bubble] (rk1-6) at (2, 2) {$0 \mid 23 \mid 1$};
  \node[bubble] (rk1-7) at (4, 2) {$0 \mid 3 \mid 1\overline{2}$};
  \node[bubble] (rk1-8) at (6, 2) {$0 \mid 2 \mid 1\overline{3} $};
  \node[bubble] (rk1-9) at (8, 2) {$0 \mid 3 \mid 1\overline{2}$};
  \node[bubble] (rk2-1) at (-12, 5) {$013 \mid 2$};
  \node[bubble] (rk2-2) at (-10, 5) {$012 \mid 3$};
  \node[bubble] (rk2-3) at (-8, 5) {$023 \mid 1$};
  \node[bubble] (rk2-4) at (-6, 5) {$01 \mid 23$};
  \node[bubble] (rk2-5) at (-4, 5) {$02 \mid 13$};
  \node[bubble] (rk2-6) at (-2, 5) {$03 \mid 12$};
  \node[bubble] (rk2-7) at (0, 5) {$01 \mid 2\overline{3}$};
  \node[bubble] (rk2-8) at (2, 5) {$02 \mid 1\overline{3} $};
  \node[bubble] (rk2-9) at (4, 5) {$03 \mid 1\overline{2}$};
  \node[bubble] (rk2-10) at (6, 5) {$0 \mid 123$};
  \node[bubble] (rk2-11) at (8, 5) {$0 \mid 12\overline{3}$};
  \node[bubble] (rk2-12) at (10, 5) {$0 \mid 1\overline{2}3 $};
  \node[bubble] (rk2-13) at (12, 5) {$0 \mid 1\overline{23}$};

  \node[bubble] (top) at (0, 7) {$0123$};

  \begin{scope}[on background layer]
  \draw[thick, BlueViolet] (bot) edge[bend left=8]  (rk1-1);
  \draw[thick, BlueViolet] (bot) edge[bend left=6]  (rk1-2);
  \draw[thick, BlueViolet] (bot) edge[bend left=4]  (rk1-3);
  \draw[thick, BlueViolet] (bot) edge[bend left=2] (rk1-4);
  \draw[thick, BlueViolet] (bot) -- (rk1-5);
  \draw[thick, BlueViolet] (bot) edge[bend right=2] (rk1-6);
  \draw[thick, BlueViolet] (bot) edge[bend right=4] (rk1-7);
  \draw[thick, BlueViolet] (bot) edge[bend right=6] (rk1-8);
  \draw[thick, BlueViolet] (bot) edge[bend right=8] (rk1-9);

  \draw[thick, BlueViolet] (rk1-1) edge[bend left=20] (rk2-1);
  \draw[thick, BlueViolet] (rk1-1) --(rk2-2);
  \draw[thick, BlueViolet] (rk1-1) -- (rk2-4);
  \draw[thick, BlueViolet] (rk1-1) --(rk2-7);

  \draw[thick, BlueViolet] (rk1-2) -- (rk2-2);
  \draw[thick, BlueViolet] (rk1-2) -- (rk2-3);
  \draw[thick, BlueViolet] (rk1-2) -- (rk2-5);
  \draw[thick, BlueViolet] (rk1-2) -- (rk2-8);

  \draw[thick, BlueViolet] (rk1-3) -- (rk2-1);
  \draw[thick, BlueViolet] (rk1-3) -- (rk2-3);
  \draw[thick, BlueViolet] (rk1-3) -- (rk2-6);
  \draw[thick, BlueViolet] (rk1-3) -- (rk2-9);

  \draw[thick, BlueViolet] (rk1-4) -- (rk2-2);
  \draw[thick, BlueViolet] (rk1-4) -- (rk2-6);
  \draw[thick, BlueViolet] (rk1-4) -- (rk2-10);
  \draw[thick, BlueViolet] (rk1-4) -- (rk2-11);

  \draw[thick, BlueViolet] (rk1-5) -- (rk2-1);
  \draw[thick, BlueViolet] (rk1-5) -- (rk2-5);
  \draw[thick, BlueViolet] (rk1-5) -- (rk2-10);
  \draw[thick, BlueViolet] (rk1-5) -- (rk2-12);

  \draw[thick, BlueViolet] (rk1-6) -- (rk2-4);
  \draw[thick, BlueViolet] (rk1-6) -- (rk2-3);
  \draw[thick, BlueViolet] (rk1-6) -- (rk2-10);
  \draw[thick, BlueViolet] (rk1-6) -- (rk2-13);

  \draw[thick, BlueViolet] (rk1-7) -- (rk2-9);
  \draw[thick, BlueViolet] (rk1-7) -- (rk2-3);
  \draw[thick, BlueViolet] (rk1-7) -- (rk2-12);
  \draw[thick, BlueViolet] (rk1-7) -- (rk2-13);

  \draw[thick, BlueViolet] (rk1-8) -- (rk2-5);
  \draw[thick, BlueViolet] (rk1-8) -- (rk2-1);
  \draw[thick, BlueViolet] (rk1-8) -- (rk2-11);
  \draw[thick, BlueViolet] (rk1-8) -- (rk2-13);

  \draw[thick, BlueViolet] (rk1-9) -- (rk2-6);
  \draw[thick, BlueViolet] (rk1-9) -- (rk2-2);
  \draw[thick, BlueViolet] (rk1-9) -- (rk2-12);
  \draw[thick, BlueViolet] (rk1-9) edge[bend right=20] (rk2-13);

  \draw[thick, BlueViolet] (rk2-1) edge[bend left=18] (top);
  \draw[thick, BlueViolet] (rk2-2) edge[bend left=15] (top);
  \draw[thick, BlueViolet] (rk2-3) edge[bend left=12] (top);
  \draw[thick, BlueViolet] (rk2-4) edge[bend left=9] (top);
  \draw[thick, BlueViolet] (rk2-5) edge[bend left=6] (top);
  \draw[thick, BlueViolet] (rk2-6) edge[bend left=3] (top);
  \draw[thick, BlueViolet] (rk2-7) -- (top);
  \draw[thick, BlueViolet] (rk2-8) edge[bend right=3] (top);
  \draw[thick, BlueViolet] (rk2-9) edge[bend right=6] (top);
  \draw[thick, BlueViolet] (rk2-10) edge[bend right=9] (top);
  \draw[thick, BlueViolet] (rk2-11) edge[bend right=12] (top);
  \draw[thick, BlueViolet] (rk2-12) edge[bend right=15] (top);
  \draw[thick, BlueViolet] (rk2-13) edge[bend right=18] (top);
  \end{scope}
\end{tikzpicture}
\end{center}
\label{tikz:pi-3-B}
\caption{The geometric lattice $\mathrm{Q}_3(\Z_2) = \Pi^{\textrm{B}}_{3}$.}
\end{figure}

In \cite{simion1995q}, Rodica Simion introduced an EL-labeling on Dowling geometries, which endows Dowling lattices with the structure of rank-uniform labeled posets. Let $\overline{(\mu, \varphi_{\mu})} \lessdot \overline{(\mu', \varphi_{\mu'}})$ be a covering relation in $\QnG$, which is a merge of two $\G$-blocks in $(\mu, \varphi_{\mu})$. Then 
\[
\lambda\left(\overline{(\mu, \varphi_{\mu})}, \overline{(\mu', \varphi_{\mu'})}\right) \coloneqq \begin{cases}
    \min \mu_j & \text{if $\overline{(\mu_i,\varphi_i)}, \overline{(\mu_j,\varphi_j)}$ are merging, with $\min \mu_{i} < \min \mu_{j}$}\\
    \min \mu_{i} & \text{if  $\overline{(\mu_i,\varphi_i)}, \mu_0$ are merging.}
\end{cases}
\]
We compute the widths and descents using this labeling for $\Pi_{2}^{\textrm{B}}$; see \Cref{fig:pi-2-B}. For the unique element $\zero$ of rank $0$, $\omega_0(1) = \omega_{\zero}(1) = 1$, $\omega_0(2) = \omega_{\zero}(2) = 3$, and $\des_{0}(1) = 0$, $\des_{0}(2) = 1$. For any $s$ of rank $1$, $\omega_{1} = \omega_{s}(1) = 1$.  

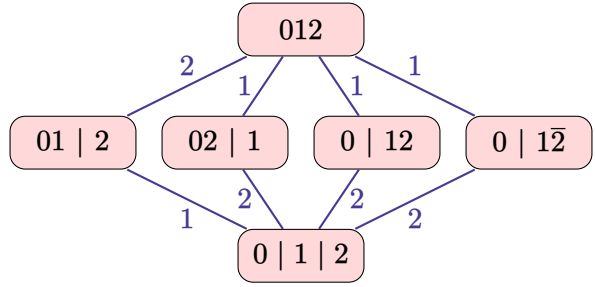
\begin{figure}
\begin{tikzpicture}[node distance=16mm and 20mm,
  poset/.style={
    draw,
    rounded corners=2mm,
    text width=14mm,
    minimum height=7mm,
    align=center,
    fill=red!15,
  }]
  \node[poset] (bot) at (0, 0) {$0 \mid 1 \mid 2$};
  \node[poset] (mid1) at (-3, 1.5) {$01 \mid 2$};
  \node[poset] (mid2) at (-1, 1.5) {$02 \mid 1$};
  \node[poset] (mid3) at (1, 1.5) {$0 \mid 12$};
  \node[poset] (mid4) at (3, 1.5) {$0 \mid 1\overline{2}$};
  \node[poset] (top) at (0, 3) {$012$};

  \draw[thick, BlueViolet] (bot) -- node[label, midway, below] {$1$} (mid1);
  \draw[thick, BlueViolet] (bot) -- node[label, midway, left] {$2$} (mid2);
  \draw[thick, BlueViolet] (bot) -- node[label, midway, right] {$2$} (mid3);
  \draw[thick, BlueViolet] (bot) -- node[label, midway, below] {$2$} (mid4);
  \draw[thick, BlueViolet] (mid1) -- node[label, midway, above] {$2$} (top);
  \draw[thick, BlueViolet] (mid2) -- node[label, midway, left] {$1$} (top);
  \draw[thick, BlueViolet] (mid3) --  node[label, midway, right] {$1$} (top);
  \draw[thick, BlueViolet] (mid4) --  node[label, midway, above] {$1$} (top);
\end{tikzpicture}
\caption{Simion's EL-labeling on $\Pi_2^{\textrm{B}}$. }
\label{fig:pi-2-B}
\end{figure}
\end{example}

We have the following lemma relating this labeled notion of uniformity to the classical, non-labeled notion of rank-uniformity for posets. 

\begin{lemma}\label{lem:el-rank-upper-implies-rank-upper}
If an EL-labeled poset $(\P, \lambda)$ is rank-uniform, then $\P$ is rank-uniform. 
\end{lemma}

\begin{proof}
We prove that $W_k(s)$ only depends on the rank of $s$ and $k$. Since $\lambda$ is an EL-labeling of $\P$, every element $t$ of rank $k$ above $s$ corresponds to a unique maximal chain from $s$ to $t$ of length $k$ 
\[
s = s_0 \lessdot s_1 \lessdot \cdots \lessdot s_k = t
\] with weakly increasing labels $\lambda(s_0, s_1) \trianglelefteqslant \cdots \trianglelefteqslant \lambda(s_{k-1}, s_k)$. 
Therefore, 
\[
W_k(s) = \abs{ \{ s = s_0 \lessdot s_1 \lessdot \cdots \lessdot s_k \, | \, \lambda(s_{\medbullet}, s_{\medbullet+1}) \textrm{ weakly increasing}\} \: }. 
\]
One can refine the quantity on the right by the index of the last label $\lambda(s_{k-1}, s_{k})$ in the total order $\vartriangleleft$ as follows: For every $i$, we write 
\[
W_{k}^i(s) \coloneqq \abs{\: \{ s = s_0 \lessdot s_1 \lessdot \cdots \lessdot s_k \, | \, \lambda(s_{\medbullet}, s_{\medbullet+1}) \textrm{ weakly increasing}, \Ind(s_{k-1}, s_{k}) = i\}\:},  
\]
and prove by induction on $k$ that $W_k^i(s) = W_k^i(s')$ as soon as $\rk s = \rk s'$. 

The base case is when $k = \rk(s) + 1 $. The number of covers of $s$ whose covering relation is labeled by the $i$-th label is the width number $w_s(i)$ which only depends on $\rk s$ and $i$ by definition of a rank-uniform labeled poset. Let $k \geqslant \rk(s) + 1$ and assume that $W_j^i(s)$ on depends on $\rk s$ for all $j \leqslant k$ and all $i$. Then, we can refine the set 
\[
\{ s = s_0 \lessdot s_1 \lessdot \cdots \lessdot s_{k+1} \, | \, \lambda(s_{\medbullet}, s_{\medbullet+1}) \textrm{ weakly increasing}, \Ind(s_{k}, s_{k+1}) = i\}
\] by the index of the penultimate label $\Ind(s_{k-1}, s_k)$. This gives the following recursive formula 
\[
W_{k+1}^{i}(s) = \sum_{j \geqslant 1} \: W_{k}^j(s) \: \cdot \left(\sum_{i \geqslant \des_k(j)}\omega_k(i) \right). 
\]
Since $(\P, \lambda)$ has uniform widths and uniform descent counts, this shows that $W_{k+1}^{i}(s) = W_{k+1}^{i}(s')$ if $\rk s = \rk s'$, which concludes the inductive step. 
\end{proof}

An additional condition on a labeled poset $(\P, \lambda)$ that will play a crucial role in this article is the \emph{monotonicity} of the descent counts at a given rank, as a function of the label indices. 

\begin{definition}[Monotonic labeled poset]
Let $(\P, \lambda)$ be a rank-uniform labeled poset, and for all $0 \leqslant k \leqslant \rk(\P)-1$ denote by $\ell_k$ the number of distinct labels above any element of rank $k$. The labeled poset $(\P, \lambda)$ is \textit{monotonic} if for all $0 \leqslant k \leqslant \rk(\P)-1$ the function $\des_k(\medbullet)$ is weakly increasing on $[\ell_k]$. A poset $\P$ admitting an EL-labeling $\lambda$ such that $(\P, \lambda)$ is a monotonic rank-uniform labeled poset is called \textit{uniformly monotonically edge-lexicographically shellable}, or \textit{UMEL-shellable} for short.
\end{definition}
In other words, a rank-uniform labeled poset $(\P, \lambda)$ is monotonic if at any element $s$ in $\P$, the larger the label above $s$, the more descents we observe from this label. 
\begin{example}
The rank-uniform labeled posets considered in Example~\ref{ex:diamond-poset}, Example~\ref{ex:uniform-matroid} and Example~\ref{ex:dowling-geo} can be verified to be monotonic via routine checks. 
\end{example}
In the next two sections, we will show that the $h$-polynomial and the $\gamma$-polynomial of a UMEL-shellable poset are real-rooted. We will do so by refining the $h$-polynomial and $\gamma$-polynomials as a positive sum of polynomials using face and flag enumeration. These refinement polynomials will satisfy a linear recursion with transition matrix as in Lemma~\ref{lem:two-column-monotonicity-interlacing}. In particular, the monotonicity condition on the EL-labeling ensures the Wronskian nonnegativity condition, which is exactly the condition guaranteeing that the transition matrix preserves interlacing. Let $\mathbf{A}(x)$ be the transition matrix between two families of polynomials $\mathbf{g}_m(x)$ and $\mathbf{h}_n(x)$
\[
\mathbf{h}_n(x) = \mathbf{A}(x) \cdot \mathbf{g}_{m}(x) = \begin{pmatrix} a_{ij}(x) \end{pmatrix} \cdot \mathbf{g}_m(x)
\] such that $\mathbf{A}$ is an $n$-by-$m$ matrix with nonnegative polynomial entries. By \Cref{thm:interlacing}, the requirement that two columns $j < j'$ satisfy, for all $p, q \in \R_{> 0}$, 
\[
f_{j'} \preceq f_{j}
\] is equivalent to the nonnegativity of the corresponding Wronskian over the real numbers
\[
\textrm{W}[f_j, f_{j'}] = \det \begin{pmatrix} 
f_{j}' & f_{j'}' \\
f_{j}  & f_{j'} 
\end{pmatrix} \geqslant 0. 
\]
The monotonicity condition on the $\des( \medbullet)$ ensures the nonnegativity of values of Wronskians at all real numbers for all column pairs. Consequently, $\mathbf{A}(x)$ preserves interlacing sequences, and the refinement polynomials form an interlacing sequence by induction, Lemma~\ref{lem:two-column-monotonicity-interlacing}, and Theorem~\ref{thm:interlacing}.

\section{Chain polynomials and \texorpdfstring{$h$}{h}-polynomials are real-rooted
}
\label{sec:h-RR}
In this section, we prove that the chain polynomial and the $h$-polynomial of the order complex of a UMEL-shellable poset has only real and nonpositive zeros. This result unifies and extends several previously known instances of real-rootedness, including those arising from Dowling geometries, projective geometries, and uniform matroids. The definition of the interlacing family introduced here also prepare the ground for the recursive framework that we will use in \Cref{sec:chow-RR}.

\newtheorem*{thm:intro2}{Theorem~\ref{thm:main2-intro}}
\begin{thm:intro2}
    {\itshape
Let $\P$ be a UMEL-shellable poset. The $h$-polynomial of the order complex $\Delta(\P)$ of $\P$ only has real and nonpositive roots.} 
\end{thm:intro2}

\begin{proof}
Let $\lambda$ be an EL-labeling of $\P$ such that $(\P, \lambda)$ is a monotonic rank-uniform labeled poset. Let $s$ be a nonmaximal element in $\P$ of corank $k$ and let $\lambda_1  \vartriangleleft \cdots \vartriangleleft \lambda_{\ell}$ denote the labels immediately above $s$. We write $\calC_s$ for the maximal chains of the interval $[s, \un]$: 
\[
\calC_{s} \coloneqq \{ s = s_0 \lessdot s_1 \lessdot \cdots \lessdot s_k = \un \}
\] 
We refine this set of maximal chains by the index of the first covering relation. For any $1 \leqslant i \leqslant \ell$, we write $\calC_{s}(i)$ for the maximal chains of the interval $[s, \un]$ such that the first covering relation is labeled by $\lambda_i$: 
\[
\calC_{s}(i) \coloneqq \{\: s = s_0 \lessdot s_1 \lessdot \cdots \lessdot s_k = \un_{s} \mid \lambda(s_0, s_1) = \lambda_i\}. 
\] 

Let us denote 
\[
h_s(y) \coloneqq \sum_{C \in \calC_s} y^{\des(\lambda(C))} = h(\Delta([s, \un])),
\] 
where the second equality comes from Lemma \ref{lem:beta-counts-des} and Equation \eqref{eq:h-from-flag-h}. We use the above refinement of $\calC_{s}$ to refine this polynomial. For any $1 \leqslant i \leqslant \ell$, we write $h_{s}^{i}(y)$ for the generating polynomial of maximal chains with first label index $i$ in $[s, \un]$ with the descent statistics: 
\[
h_{s}^i(y) \coloneqq \sum_{C \in \calC_s(i)} y^{\des(\lambda(C))}. 
\] 

By rank-uniformity of $(\P, \lambda)$, the sequence of polynomials
\[
\mathbf{h}_{s}(y) \coloneqq (h_{s}^{1}(y)\: , \: \ldots \:, \: h_{s}^{\ell}(y))
\] only depends on the corank $k$ of $s$, and so we instead write
\[
\mathbf{h}_{k}(y) = (h^1_{k}(y)\: , \: \ldots \:, \: h^{\ell}_{k}(y)).
\]
If $n$ is the rank of $\P$, we have 
\[
h(\Delta(\P), y) = \sum_{i\geqslant 1}^{\ell} h^i_{n}(y).
\] 
We will prove that $\mathbf{h}_k(y)$ is interlacing for all $1 \leqslant k \leqslant n$, by induction on the corank $k$. The result will then follow from Theorem \ref{thm:sum-of-interlacing-is-RR}.

For the base case, we have that $\mathbf{h}_1(y) = (1)$, which is an interlacing family. For the inductive hypothesis, we assume that $\mathbf{h}_k(y)$ is an interlacing family for some $k < n$, and let $s$ be any element of $\P$ of corank $k+1$. For every $1 \leqslant i \leqslant \ell$, the number of covers of $s$ with labels $\lambda_i$ is the width of $s$ at index $i$
\[
\omega_s(i) = \text{the number of $t$ in $[s,\un]$ with $s \lessdot t$ such that $\lambda(s, t) = \lambda_i$}. 
\] 
For each $i$, recall that $\des_s(i)$ records the number of indices in the labels above covers of $s$ labeled by $\lambda_i$ that would give rise to a descent. Therefore, a cover of $s$ labeled by $\lambda_i$ followed by $\lambda_j'$ with $j \leqslant \des_s(i)$ gives rise to a descent $\lambda_i \vartriangleright \lambda_j'$, and a cover of $s$ labeled by $\lambda_i$ followed by $\lambda_j'$ with $j > \des_s(i)$ gives rise to an ascent $\lambda_i \trianglelefteqslant \lambda_j'$. Therefore, we have the following recursive formula
\begin{equation}\label{eq:h-recursion}
h_{s}^{i}(y) = \omega_{s}(i) \left(\sum_{j \leqslant \des_s(i)} \: y \: \cdot \: h_k^{j}(y) + \sum_{j > \des_s(i)} h_k^j(y) \right).
\end{equation}
By the uniform-width condition on $(\P, \lambda)$, the statistics $\omega_s(i)$ only depends on the corank of $s$. Therefore, 
\[
h_{k+1}^{i}(y) = \omega_{k+1}(i) \left(\sum_{j \leqslant \des_{k+1}(i)} \: y \: \cdot \: h_k^{j}(y) + \sum_{j > \des_{k+1}(i)} h_k^j(y) \right).
\]
We can write this recursion in terms of a matrix-vector multiplication as follows:  
\[
\mathbf{h}_{k+1}(y) = \mathbf{A}_{k+1}(y) \cdot \mathbf{h}_k(y), 
\] where the matrix $\mathbf{A}_{k+1} = \begin{pmatrix}a_{ij} \end{pmatrix}$ is a matrix with polynomial entries defined by 
\[
 a_{ij}(y) = \begin{cases}
\omega_{k+1}(i) \cdot y & \text{ if } j \leqslant \des_{k+1}(i), \\
\omega_{k+1}(i) & \text{ if } j > \des_{k+1}(i). 
\end{cases}
\]
By monotonicity of $(\P, \lambda)$ the numbers $\des_{k+1}(i)$ are weakly increasing in $i$, and so by \Cref{lem:two-column-monotonicity-interlacing}, and Theorem~\ref{thm:interlacing}, the matrix $\mathbf{A}_{k+1}(y)$ preserves the interlacing property. This shows that $\mathbf{h}_{k+1}(y)$ is an interlacing sequence, which concludes the proof. 
\end{proof}

Our proof implies the following interlacing statement. 
\begin{proposition}
    For an atom $a$ of $\P$, the roots of $h(\Delta([a, \un]); y)$ and $h(\Delta(\P); y)$ interlace. 
\end{proposition}

\begin{proof}
    Fix an atom $a$ of $\P$. 
    By \Cref{eq:h-recursion} and the fact that $\des_0(1) = 0$, 
    the first polynomial in the refinement of $h_{0}(y)$ is equal to
    \[
    h_{n}^{1}(y) = \omega_{0}(1) \sum_{j \geqslant 1} h_{n-1}^{j}(y) = \omega_0(1) \sum_{j \leqslant 1} h_{[a, \un]}^j(y) = \omega_0(1) \; h(\Delta([a, \un]); y). 
    \]
    Since $\mathbf{h}_{n}(y)$ is an interlacing sequence, \Cref{prop:interlacing-properties} (Convexity) implies that 
    \[
    h_{n}^{1}(y) \preceq \sum_{i \geqslant 1} h_{n}^{i}(y) = h(\Delta(\P); y). 
    \]
    The statement follows from combining this interlacing with the previous equality. 
\end{proof}

\section{Chow polynomials and \texorpdfstring{$\gamma$}{gamma}-polynomials are real-rooted}
\label{sec:chow-RR}
In this section, we prove that the Chow and augmented Chow polynomial of a UMEL-shellable poset only have real and nonpositive roots.

\newtheorem*{thm:intro1}{Theorem~\ref{thm:main1-intro}}
\begin{thm:intro1}
    {\itshape
    Let $\P$ be a UMEL-shellable poset.
    Then the following statements hold:
    \begin{enumerate}[\normalfont (i)]
        \item The Chow polynomial $\uH_{\P}(x)$ only has real and nonpositive roots. 
        \item The augmented Chow polynomial $\H_{\P}(x)$ only has real and nonpositive roots.
    \end{enumerate}}
\end{thm:intro1}

We prove the theorem by proving the following equivalent theorem for the respective $\gamma$-polynomials. 

\begin{maintheorem}
\label{thm:gamma-RR-mono-rank-up-uniform}
    Let $\P$ be a UMEL-shellable poset. 
    The following statements hold.
    \begin{enumerate}[\normalfont(i)]
        \item The $\gamma$-polynomial $\gamma(\uH_{\P}; y)$ only has real and nonpositive roots. 
        \item The $\gamma$-polynomial $\gamma(\H_{\P}; y)$ only has real and nonpositive roots. 
    \end{enumerate}
\end{maintheorem}
The upshot of passing to the $\gamma$-polynomials is that those polynomials count certain maximal chains of $\P$ with the descent statistics, as explained in Sections~\ref{subsec:flag-enum}~and~\ref{sec:chow-poly-recap}. We will prove \Cref{thm:gamma-RR-mono-rank-up-uniform} by providing a decomposition of the $\gamma$-polynomials into interlacing families of polynomials, given by refinements coming from the face and flag enumerations of the EL-labeled poset $(\P, \lambda)$, similarly to what we did in the previous section. These refinement polynomials will satisfy explicit linear recursive formulas, and we will use Theorem \ref{thm:interlacing} to prove that those polynomials form interlacing sequences throughout the recursion. 
\subsection{Refinements} 
Let $(\P,\lambda)$ be a monotonic rank-uniform labeled poset of rank $n$. We introduce several sequences of interlacing polynomials associated to $(\P, \lambda)$. They are defined by counting the maximal chains of $\P$ whose EL-label sequence have no double descents as follows. 

Let $\calC \coloneq \calC(\P)$ denote the collection of maximal chains of $\P$. If $S$ is a subset of $[n-1]$, let $\calC(S)$ denote the collection of maximal chains of $\P$ labeled by $\lambda$ with descent set $S$: 
\[
\calC(S) \coloneqq \{ C \in \calC \mid \Des(C) = S\}. 
\]
Let $\calC^{\textrm{ndd}} \coloneqq \calC_{\textrm{ndd}}(\P, \lambda)$ denote the collection of maximal chains in $\P$ labeled by $\lambda$ with no double descent: 
\[
\calC^{\textrm{ndd}} \coloneq \{\: C \in \calC \mid \Des(C) \text{ is stable }\}.
\] 
By construction, 
\[
\calC^{\textrm{ndd}} =  \bigsqcup_{\substack{ S \subseteq [n-1] \\ S \text{ stable}}} \: \calC(S). 
\]  

\subsubsection{Refinement by the starting label}
The set $\calC^{\textrm{ndd}}$ admits a refinement depending on whether the first label and the second label of a maximal chain labeled by $\lambda$ form a descent or an ascent. We define $\calC^{\textrm{ndd}}_{\asce}$ as the set of maximal chains of $\P$ with no double descents, which start with an ascent, and we define $\calC^{\textrm{ndd}}_{\desc}$ as the set of maximal chains of $\P$ with no double descents, which start with a descent: 
\begin{align*}
\calC^{\textrm{ndd}}_{\asce}
   &\coloneqq \{C \in \calC_{\textrm{ndd}} \mid 1 \notin \Des(C)\}, \\
\calC^{\textrm{ndd}}_{\desc}
   &\coloneqq \{C \in \calC_{\textrm{ndd}} \mid 1 \in \Des(C)\}. 
\end{align*}

The set $\calC^{\textrm{ndd}}$ admits another refinement by the index of its first edge label. For every $i \geqslant 1$, we define $\calC^{\textrm{ndd}}_i$ as the set of maximal chains of $(\P, \lambda)$ of the form 
\[
C \colon \zero = s_0 \lessdot s_1 \lessdot \cdots \lessdot s_n = \un
\] 
such that the index of the first label is $i$. 
\begin{align*}
\calC^{\textrm{ndd}}_{i}
   \coloneqq \{ C \in \calC^{\textrm{ndd}} \mid \Ind(s_0, s_1) = i\}. 
\end{align*}
We have the following intersection of the two refinements. 
\begin{align*}
\calC^{\textrm{ndd}}_{i, \asce}
   &\coloneqq \{C \in \calC^{\textrm{ndd}} \mid \Ind(s_0, s_1) = i, \lambda(s_0, s_1) \trianglelefteqslant \lambda(s_1, s_2)\}, \\
\calC^{\textrm{ndd}}_{i, \desc}
   &\coloneqq \{C \in \calC^{\textrm{ndd}} \mid \Ind(s_0, s_1) = i, \lambda(s_0, s_1) \vartriangleright \lambda(s_1, s_2)\}. 
\end{align*}
Finally, we consider the generating polynomials of descent counts of these refinements
\begin{align*}
    \gamma^n_{i}(y) \coloneqq \sum_{C \in \calC^{\textrm{ndd}}_i} y^{\des(C)}, \quad 
    \gamma^n_{i, \asce}(y) \coloneqq \sum_{C \in \calC^{\textrm{ndd}}_{i, \asce}} \: y^{\des(C)}, \quad 
    \gamma^n_{i, \desc}(y) \coloneqq \sum_{C \in \calC^{\textrm{ndd}}_{i, \desc}} \: y^{\des(C)}.
\end{align*}

\begin{proposition}
\label[proposition]{prop:gamma-refinement}
Let $(\P, \lambda)$ be an EL-labeled poset of rank $n$. Then the $\gamma$-polynomial of $\uH_{\P}(x)$ is the sum of all $\gamma^{n}_{i, \asce}(y)$ 
\[
\gamma(\uH_{\P}; y) = \sum_{i = 1}^{\ell} \gamma^{n}_{i, \asce}(y). 
\]
The $\gamma$-polynomial of $\H_{\P}(x)$ is the sum of all $\gamma^{n}_{i}(y)$ 
\[
\gamma(\H_{\P}; y) = \sum_{i = 1}^{\ell} \gamma^{n}_{i}(y). 
\]
\end{proposition}

\begin{proof}
    By \Cref{thm:gamma-flag-h-vector}, the $\gamma$-polynomial of the Chow polynomial of $\P$ is given by
    \[
    \gamma(\uH_{\P}; \: y) \;=\; \sum_{\substack{S \text{ stable} \\ 1 \notin S \subseteq [n-1]}} \beta_{\P}(S) \: y^{|S|}.
    \]
    By \Cref{lem:beta-counts-des},
    \[
    \sum_{\substack{S \text{ stable} \\ 1 \notin S \subseteq [n-1]}} \beta_{\P}(S) \: y^{|S|} \;=\; \sum_{\substack{S \text{ stable} \\ 1 \notin S \subseteq [n-1]}} \abs{\calC(S)} \cdot y^{\abs{S}}.
    \]
    Grouping the maximal chains according to the index of their first label yields
    \[
    \sum_{\substack{S \text{ stable} \\ 1 \notin S \subseteq [n-1]}} \beta_{\P}(S) \: y^{|S|} \;=\; \sum_{i = 1}^{\ell} \sum_{\substack{S \text{ stable} \\ 1 \notin S \subseteq [n-1]}}\abs{\calC_i(S)} \cdot y^{\abs{S}}.
    \]
    By the construction of the refining polynomials, we have
    \[
    \gamma(\uH_{\P}; \: y) \;=\; \sum_{\substack{S \text{ stable} \\ 1 \notin S \subseteq [n-1]}} \beta_{\P}(S) \: y^{|S|} \;=\; \sum_{i = 1}^{\ell} \gamma^{n}_{i, \asce}(y).
    \]
    Note that the condition that $S$ does not contain $1$ implies that the edge labeling of the maximal chains begins with an ascent.

    By \Cref{thm:gamma-flag-h-vector}, the $\gamma$-polynomial of the augmented Chow polynomial is given by
    \[
    \gamma(\H_{\P}; y) \;=\; \sum_{\substack{S \text{ stable} \\ S \subseteq [n-1]}} \beta_{\P}(S) \: y^{|S|}. 
    \]
    By \Cref{lem:beta-counts-des},
    \[
    \sum_{\substack{S \textrm{ stable} \\ S \subseteq [n-1]}} \beta_{\P}(S) \: y^{|S|} \;=\; \sum_{\substack{S \text{ stable} \\ S \subseteq [n-1]}} \abs{\calC(S)} \cdot y^{\abs{S}}. 
    \]
    Grouping the maximal chains according to the index of their first edge label gives
    \[
    \sum_{\substack{S \text{ stable} \\ S \subseteq [n-1]}} \beta_{\P}(S) \: y^{|S|} \;=\; \sum_{i = 1}^{\ell} \sum_{\substack{S \text{ stable} \\ S \subseteq [n-1]}} \abs{\calC_i(S)} \cdot y^{\abs{S}}. 
    \]  
    By the construction of the refining polynomials, we have
    \[
    \gamma(\H_{\P};y) \;=\; \sum_{\substack{S \text{ stable} \\ S \subseteq [n-1]}} \beta_{\P}(S) \: y^{|S|} \;=\; \sum_{1 \leqslant i  \leqslant\ell} \gamma^{n}_{i}(y).
    \]
    Note that in this case, there is no restriction on whether the edge labeling of a maximal chain begins with an ascent or a descent.
\end{proof}

More generally for all $1 \leqslant k \leqslant n$ and all $i \geqslant 1$, we denote by $\gamma^k_{i}, \gamma^k_{i,\asce}, \gamma^k_{i, \desc}$ the corresponding $\gamma$-polynomials associated to $([s, \un], \lambda_{s, \un})$ for any element $s$ of corank $k.$ By rank-uniformity of $(\P, \lambda)$ those polynomials only depend on the corank of $s$. Let $\ell$ and $\ell'$ denote the numbers of distinct labels immediately above $\zero$ and above any atom, respectively. To prove \Cref{thm:gamma-RR-mono-rank-up-uniform} (1) and (2), it suffices by \Cref{prop:gamma-refinement} and Theorem \ref{thm:sum-of-interlacing-is-RR} to show that the sequences of polynomials
\[
\pmb{\bm{\gamma}}^n_{\asce}(y) \coloneqq (\gamma^n_{i, \asce}(y))_{i=1}^{\ell}, 
\quad \text{ and } \quad 
\pmb{\bm{\gamma}}^{n}(y) \coloneqq \left(\gamma^{n}_{i}(y)\right)_{i=1}^{\ell}
\]
are interlacing. 

This is implied by the following lemma. 
\begin{lemma}
\label[lemma]{lem:induction}
The following diagram of polynomials forms interlacing families along directed paths. 
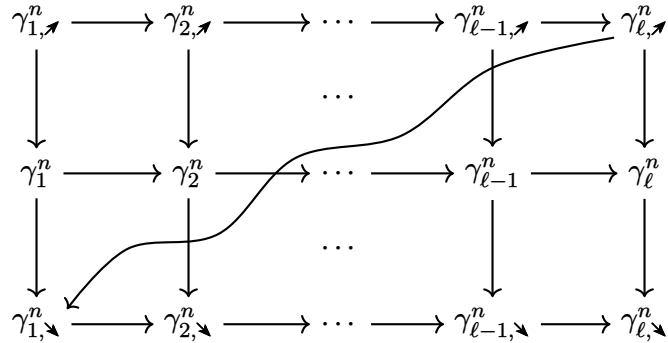
\begin{center}
\begin{equation}
\raisebox{-2.3cm}{
\begin{tikzpicture}[scale=2,
    baseline=(current bounding box.center),
    every path/.style={line width=0.8pt}  
]
    \node[] (A) at (0,0) {$\gamma^{n}_{1, \asce}$};
    \node[] (B) at (1,0) {$\gamma^{n}_{2, \asce}$};
    \node[] (C) at (2,0) {$\cdots$};
    \node[] (D) at (3,0) {$\gamma^{n}_{\ell - 1, \asce}$};
    \node[] (E) at (4,0) {$\gamma^{n}_{\ell, \asce}$};

    \node[] (F) at (0,-1) {$\gamma^{n}_{1}$};
    \node[] (G) at (1,-1) {$\gamma^{n}_{2}$};
    \node[] (H) at (2,-1) {$\cdots$};
    \node[] (I) at (3,-1) {$\gamma^{n}_{\ell - 1}$};
    \node[] (J) at (4,-1) {$\gamma^{n}_{\ell}$};

    \node[] (K) at (0,-2) {$\gamma^{n}_{1, \desc}$};
    \node[] (L) at (1,-2) {$\gamma^{n}_{2, \desc}$};
    \node[] (M) at (2,-2) {$\cdots$};
    \node[] (N) at (3,-2) {$\gamma^{n}_{\ell - 1, \desc}$};
    \node[] (0) at (4,-2) {$\gamma^{n}_{\ell, \desc}$};

    \node[] () at (2, -0.5) {$\cdots$};
    \node[] () at (2, -1.5) {$\cdots$};

    \draw[->] (A) -- (B);
    \draw[->] (B) -- (C);
    \draw[->] (C) -- (D);
    \draw[->] (D) -- (E);

    \draw[->] (F) -- (G);
    \draw[->] (G) -- (H);
    \draw[->] (H) -- (I);
    \draw[->] (I) -- (J);

    \draw[->] (K) -- (L);
    \draw[->] (L) -- (M);
    \draw[->] (M) -- (N);
    \draw[->] (N) -- (0);

    \draw[->] (A) -- (F);
    \draw[->] (B) -- (G);
    \draw[->] (D) -- (I);
    \draw[->] (E) -- (J);

    \draw[->] (F) -- (K);
    \draw[->] (G) -- (L);
    \draw[->] (I) -- (N);
    \draw[->] (J) -- (0);

    \draw[->,smooth, tension=0.6] plot coordinates {
    (3.8, -0.1)   
    (3.0, -0.3)   
    (2.4, -0.75)  
    (1.7, -0.9)   
    (1.2, -1.4)   
    (0.6, -1.5)   
    (0.2, -1.9)   
    };
\end{tikzpicture}
}
\tag{D}
\label{eq:inductive-step-P}
\end{equation}
\end{center}
\end{lemma}

This diagram is a generalization of the diagram in \cite[Theorem~3.3]{hoster-stump}. 

\begin{proof}[Proof~of~\Cref{lem:induction}]
The proof goes by induction on $n$. In rank $1$ we have $\ell = 1$ and we set by convention $\gamma^1_{1, \asce}=1,\gamma^1_{1} = 1, \gamma^1_{1, \desc} = 0,$ for which the statement is true. For the induction step we split the result into several statements covering all directed paths. 
\begin{enumerate}[\normalfont(i)]
    \item \label{item:main-top-row} The top row of \eqref{eq:inductive-step-P} is an interlacing sequence. 
    \item \label{item:main-mid-row} The middle row of \eqref{eq:inductive-step-P} is an interlacing sequence.
    \item \label{item:main-bot-row} The bottom row of \eqref{eq:inductive-step-P} is an interlacing sequence.
    \item \label{item:main-top-to-bot} For all $1 \leqslant i , j \leqslant \ell$ we have $\gamma^{n}_{i, \asce} \preceq \gamma^{n}_{j, \desc}.$
    \item \label{item:main-top-to-mid} For all $1 \leqslant i \leqslant j \leqslant \ell$ we have $\gamma^{n}_{i, \asce} \preceq \gamma^{n}_{j}.$
    \item \label{item:main-mid-to-bot} For all $1 \leqslant i \leqslant j \leqslant \ell$ we have $\gamma^{n}_{i} \preceq \gamma^{n}_{j, \desc}.$
\end{enumerate}

For statement \ref{item:main-top-row}, we derive a recursive formula for the family $\pmb{\bm{\gamma}}^n_{\asce}$ in terms of the family $\pmb{\bm{\gamma}}^{n-1}$ as follows. We consider any element $t$ of rank $2$, and any atom $s$ such that $\zero \lessdot s \lessdot t$ and $\Ind(\zero, s) = i$, and analyze whether the labels of this chain of length $2$ form a descent or an ascent. Recall that $\des_{s}(i)$ counts the number of indices in the labels above any element covering $s$ with label $i$, that gives rise to a descent, and let $j = \Ind(s, t)$. 
\begin{itemize}
    \item If $j \leqslant \des(i)$, then $\lambda(\zero, s) \vartriangleright \lambda(s, t)$, giving rise to a descent.
    \item If $j > \des(i)$, then $\lambda(\zero, s) \trianglelefteqslant \lambda(s, t)$, giving rise to an ascent.
\end{itemize} 
To simplify notations we denote $\omega = \omega_{0}$. By rank-uniformity of $(\P, \lambda)$, we have the following recursion 
\[
\gamma^{n}_{i, \asce} = \omega(i) \cdot \sum_{j > \des(i)} \gamma^{\,n-1}_{j}
= \omega(i) \cdot \sum_{j > \des(i)} \gamma^{\,n-1}_{j}.
\]
This recursion can be rewritten as 
\[
\pmb{\bm{\gamma}} _{\asce}^{n} = \mathbf{A} \cdot \pmb{\bm{\gamma}}^{n-1}, 
\] where the matrix $\mathbf{A} = \begin{pmatrix} a_{ij}\end{pmatrix}$ is defined by  
\[
a_{ij} = \begin{cases}
    0 & \text{if $j \leqslant \des(i)$}, \\
    \omega(i) & \text{if $j > \des(i)$}.
\end{cases}
\]
By \Cref{lem:two-minors-monotonicity-interlacing} and \Cref{thm:interlacing}, $\pmb{\bm{\gamma}} _{\asce}^{n}$ is an interlacing sequence of polynomials. 
\smallskip

For statement \ref{item:main-mid-row}, let $1\leqslant i \leqslant j \leqslant \ell$ be two integers. We prove the interlacing relation $\gamma^n_i \preceq  \gamma^n_{j}$ by expressing  $(\gamma^n_i, \gamma^n_{j})$ in terms of the interlacing families $\pmb{\bm{\gamma}}_{\asce}^{n-1}$ and $\pmb{\bm{\gamma}}^{n-1}$ as follows. By definition we have 
\[
\gamma^{n}_{i} = \gamma^{n}_{i, \asce} + \gamma^{n}_{i, \desc}, 
\quad \text{ and } \quad 
\gamma^{n}_{j} = \gamma^{n}_{j, \asce} + \gamma^{n}_{j, \desc}.
\]
As explained in the proof of statement \ref{item:main-top-row}, 
\begin{align*}
\gamma^{n}_{i, \asce} &= \omega(i) \: \:\sum_{k > \des(i)} \gamma^{n-1}_{k}, \\
\gamma^{n}_{j, \asce} &= \omega(j) \: \:  \sum_{k > \des(j)} \gamma^{n-1}_{k}, 
\end{align*}
and by similar arguments, 
\begin{align*}
\gamma^{n}_{i, \desc} &= \omega(i) \: y \: \: \sum_{k \leqslant \des(i)} \gamma^{n-1}_{k, \asce}, \\
\gamma^{n}_{j, \desc} &= \omega(j) \: y \: \: \sum_{k \leqslant \des(j)} \gamma^{n-1}_{k, \asce}.
\end{align*}
Summing the relations above, we obtain 
\begin{align}\label{eq:gamma-i-j-all-1}
\gamma_{i}^{n} &= \omega(i) \left(y \cdot \sum_{k \leqslant \des(i)} \gamma_{k, \asce}^{n-1} + \sum_{k > \des(i)} \gamma_k^{n-1} \right),
\end{align}
\begin{align}\label{eq:gamma-i-j-all-2} 
\gamma_{j}^{n} &= \omega(j) \left( y \cdot \sum_{k \leqslant \des(j)} \gamma_{k, \asce}^{n-1} + \sum_{k > \des(j)} \gamma_{k}^{n-1}\right). 
\end{align}
We divide the final argument into two cases, depending on the difference $\des(j) - \des(i)$. By monotonicity of $(\P, \lambda)$, we have $\des(j) - \des(i) \geqslant 0$.
\begin{itemize}[leftmargin=18pt]
\item If $\des(i) = \des(j)$, then 
\[
\gamma_{j}^{n} = \frac{\omega(j)}{\omega(i)} \gamma_{i}^{n}. 
\] 
By the fact that the sequence $(\gamma^{n-1}_{1, \asce},\ldots, \gamma^{n-1}_{\des(i), \asce}, \gamma^{n-1}_{\des(i) + 1}, \ldots, \gamma^{n-1}_{\ell'})$ is interlacing, Lemma~\ref{lem:two-column-monotonicity-interlacing}, Theorem~\ref{thm:interlacing},  Equation~\eqref{eq:gamma-i-j-all-1}, and Equation~\eqref{eq:gamma-i-j-all-2}, $\pmb{\bm{\gamma}}^n$ is interlacing and so we have $\gamma_{i}^{n} \preceq \gamma_{j}^{n}$.

\item If $\des(i) < \des(j)$, then let us denote $J \coloneqq \{\des(i) + 1, \ldots, \des(j)\}$, and set 
\[
\gamma^{n-1}_{J, \asce} = \sum_{k \in J } \gamma^{n-1}_{k, \asce}, \quad \gamma^{n-1}_{J} = \sum_{k \in J} \gamma^{n-1}_{k}.
\]  
We start by proving that the sequence of polynomials 
\begin{equation*}\label{eq:sequence-intermediate-sum}
(\gamma^{n-1}_{1, \asce}, \ldots, \gamma^{n-1}_{\des(i), \asce}, \gamma^{n-1}_{J, \asce}, \gamma^{n-1}_{J}, \gamma^{n-1}_{\des(j) +1}, \ldots, \gamma^{n-1}_{\ell'}) 
\end{equation*}
is interlacing. Since the sequence $(\gamma^{n-1}_{1, \asce}, \ldots, \gamma^{n-1}_{\des(i), \asce}, \gamma^{n-1}_{\des(j) +1}, \ldots, \gamma^{n-1}_{\ell'})$ is interlacing we only have to prove that both $\gamma^{n-1}_{J, \asce}$ and $\gamma^{n-1}_{J}$ are compatible with every other polynomial of sequence \eqref{eq:sequence-intermediate-sum}. First we have 
\begin{equation*}\label{eq:gamma-J}
\begin{pmatrix}
\gamma^{n-1}_{J, \asce} &
\gamma^{n-1}_{J}
\end{pmatrix}^{\top}
=      
\mathbf{A} \cdot 
\begin{pmatrix}
\gamma^{n-1}_{\des(i) +1, \asce} &
\cdots &
\gamma^{n-1}_{\des(j), \asce} &
\gamma^{n-1}_{\des(i)+1, \desc } &
\cdots &
\gamma^{n-1}_{\des(j), \desc}
\end{pmatrix}^{\top},  
\end{equation*}
where  
\begin{equation*}
    \mathbf{A} =  
\begin{pmatrix} 
    1 & \cdots & 1 & 0 & \cdots & 0 \\
    1 & \cdots & 1 & 1 & \cdots & 1 
\end{pmatrix}.
\end{equation*}
Since the sequence of polynomials on the right hand side of \eqref{eq:gamma-J} is interlacing, by Theorem~\ref{thm:interlacing} we obtain $\gamma^{n-1}_{J, \asce} \preceq \gamma^{n-1}_{J}$. Similarly, by the fact that the sequence of polynomials $(\gamma^{n-1}_{1, \asce}, \ldots, \gamma^{n-1}_{\des(j), \asce})$ is interlacing we get $\gamma^{n-1}_{k ,\asce} \preceq \gamma^{n-1}_{J, \asce}$ for any $k \leqslant \des(i),$ and by the fact that the sequence $(\gamma^{n-1}_{\des(i) +1}, \ldots, \gamma^{n-1}_{\ell'})$ is interlacing we get $\gamma^{n-1}_{J} \preceq \gamma^{n-1}_{k}$ for all $k > \des(j)$. This proves that the sequence \eqref{eq:sequence-intermediate-sum} is interlacing. To conclude, we have 
\begin{equation*}
\begin{pmatrix}
\gamma^n_i &
\gamma^{n}_{j}
\end{pmatrix}^{\top}
=      
\mathbf{A}(y) \cdot 
\begin{pmatrix}
\gamma^{n-1}_{1, \asce}& \ldots& \gamma^{n-1}_{\des(i), \asce}& \gamma^{n-1}_{J, \asce}& \gamma^{n-1}_{J}& \gamma^{n-1}_{\des(j) +1}& \ldots& \gamma^{n-1}_{\ell'}
\end{pmatrix}^{\top},  
\end{equation*}
with 
\begin{equation*}
    \mathbf{A}(y) =  
\begin{pmatrix} 
    y & \cdots & y & 0 & 1 & 1 &\cdots & 1 \\
    y & \cdots & y & y & 0 & 1 & \cdots & 1 
\end{pmatrix}.
\end{equation*}
One can check that the above matrix satisfies the hypotheses of Theorem~\ref{thm:interlacing}, which implies that $\gamma^n_i$ interlaces $\gamma^n_{j}$.
\end{itemize}

For statement \ref{item:main-bot-row}, using similar arguments as in the proof of \ref{item:main-top-row} we obtain 
\[
\gamma^{n}_{i, \desc} = \omega(i) \: y \cdot  \sum_{j \leqslant \des(i)} \gamma^{n-1}_{j},
\]
for all $1 \leqslant i \leqslant \ell.$ This can be reformulated as 
\begin{equation*}
\pmb{\bm{\gamma}}^n_{\desc} = 
\mathbf{A}(y) \cdot \pmb{\bm{\gamma}}^{n-1}_{\asce}
\end{equation*}
where $\mathbf{A}(y) = \begin{pmatrix} a_{ij}(y) \end{pmatrix}$ is the matrix defined by 
\[ a_{ij} = 
\begin{cases}
    \omega(i) y & \text{ if } j \leqslant \des(i), \\
    0 & \text{ if } j > \des(i).
\end{cases}
\]
One can verify that the matrix $\mathbf{A}$ satisfies the conditions of \Cref{thm:interlacing} and so by Theorem~\ref{thm:interlacing}, the polynomials in the bottom row of diagram \eqref{eq:inductive-step-P} form an interlacing sequence. 
\smallskip

For statement \ref{item:main-top-to-bot}, let $1 \leqslant i,j \leqslant \ell$ be two integers. We distinguish between two cases.
\begin{itemize}[leftmargin=18pt]
\item If $\des(j) \leqslant \des(i)$, then we have 
\begin{equation*}\label{eq:gamma-i-j-topbot}
\begin{pmatrix}
\gamma^n_{i, \asce} &
\gamma^{n}_{j, \desc}
\end{pmatrix}^{\top}
=      
\mathbf{A}(y) \cdot 
\begin{pmatrix}
\gamma^{n-1}_{1, \asce}& \ldots& \gamma^{n-1}_{\des(j), \asce}&  \gamma^{n-1}_{\des(i) +1}& \ldots& \gamma^{n-1}_{\ell'}
\end{pmatrix}^{\top},  
\end{equation*}
where 
\begin{equation*}
    \mathbf{A}(y) = \begin{pmatrix}
0& \cdots & 0 & \omega(i) & \cdots & \omega(i) \\
\omega(j)y & \cdots & \omega(j)y & 0 & \cdots & 0
\end{pmatrix}.
\end{equation*}
Since the sequence of polynomials in the right hand side of \eqref{eq:gamma-i-j-topbot} is interlacing, by Theorem~\ref{thm:interlacing} we get $\gamma^{n}_{i, \asce} \preceq \gamma^{n}_{j, \desc}.$
\item If $\des(j) > \des(i)$, we have 
\begin{equation*}\label{eq:gamma-i-j-topbot-incrdesc}
\begin{pmatrix}
\gamma^n_{i, \asce} &
\gamma^{n}_{j, \desc}
\end{pmatrix}^{\top}
=      
\mathbf{A}(y) \cdot 
\begin{pmatrix}
\gamma^{n-1}_{1, \asce}& \ldots& \gamma^{n-1}_{\des(i), \asce}&  \gamma^{n-1}_{J, \asce}& \gamma^{n-1}_{J}, \gamma^{n-1}_{\des(j) +1}& \ldots& \gamma^{n-1}_{\ell'}
\end{pmatrix}^{\top},  
\end{equation*}
where $\gamma^{n-1}_{J, \asce}$ and $\gamma^{n-1}_{J}$ have been defined in the proof of \ref{item:main-mid-row}, and 
\begin{equation*}
    \mathbf{A}(y) = 
\begin{pmatrix}
0& \cdots & 0 & 0& \omega(i) & \omega(i) &\cdots & \omega(i) \\
\omega(j)y & \cdots & \omega(j)y & \omega(j)y& 0 & 0& \cdots & 0
\end{pmatrix}.
\end{equation*}
In the proof of \ref{item:main-mid-row} we have shown that the sequence of polynomials in the right hand side of \eqref{eq:gamma-i-j-topbot-incrdesc} is interlacing and so we can conclude by Theorem~\ref{thm:interlacing}. 
\end{itemize}

For statements \ref{item:main-top-to-mid} and \ref{item:main-mid-to-bot}, since we have $\gamma^{n}_{i} = \gamma^{n}_{i, \asce} + \gamma^{n}_{i, \desc}$ for all $1 \leqslant i \leqslant \ell$ the result follows from statements \ref{item:main-top-row}, \ref{item:main-bot-row}, \ref{item:main-top-to-bot} and convexity of interlacing polynomials (Proposition~\ref{prop:interlacing-properties}).
\end{proof}

\begin{proof}[Proof of \Cref{thm:gamma-RR-mono-rank-up-uniform}]
    The theorem follows from \Cref{lem:induction} and \Cref{prop:gamma-refinement}. 
\end{proof}

Our proof of \Cref{thm:gamma-RR-mono-rank-up-uniform} implies several interlacing relations among different Chow polynomials. In general, both the ($\gamma$-polynomials of the) Chow polynomials of $\P$ and $\P^{\ast}$ are conjectured to interlace the ($\gamma$-polynomials of the) the augmented Chow polynomials for matroids \cite[Conjecture 5.7]{ferroni-matherne-stevens-vecchi}. Li also conjectured that the ($\gamma$-polynomial of the) Chow polynomial of a matroid contracted at an atom interlaces the ($\gamma$-polynomial of the) Chow polynomial of the matroid. See \Cref{sec:open-questions}. The theorem below settles those conjectures for UMEL-shellable posets. In the case of Boolean lattices, the statement (iv) below recovers the statement that the roots of $n$th Eulerian polynomial and the $(n+1)$th Eulerian polynomial interlace \cite[Theorem 2.7]{savage2015}. 

\newtheorem*{thm:intro-interlacing}{Theorem~\ref{thm:main-interlacing-intro}}
\begin{thm:intro-interlacing}
    {\itshape 
    Let $\P$ be a UMEL-shellable poset. The following statements hold: 
    \begin{enumerate}[\normalfont(i)]
        \item The roots of $\uH_{\P}(x)$ and $\H_{\P}(x)$ interlace.
        \item For an atom $a$ of $\P$, the roots of $\H_{[a, \un]}(x)$ and $\uH_{\P}(x)$ 
        interlace. 
        \item For an atom $a$ of $\P$, the roots of $\H_{[a, \un]}(x)$ and $\H_{\P}(x)$ 
        interlace.
        \item For an atom $a$ of $\P$ under a mild condition (see \Cref{thm:gamma-interlacing-aug-contract}), the roots of $\uH_{[a, \un]}(x)$ and $\uH_{\P}(x)$ 
        interlace. 
    \end{enumerate}}
\end{thm:intro-interlacing}
The validity of the statement above is implied by the following theorem, by the fact that the $\gamma$-transformation preserves interlacing polynomials in \Cref{lem:gamma-transform-preserves-interlacing}. Note that we also add a condition under which the Chow polynomial of a contraction at an atom interlaces the Chow polynomial of the poset. 

\begin{maintheorem}\label[theorem]{thm:gamma-interlacing-aug-contract}
    Let $\P$ be a UMEL-shellable poset, and let $a$ be an atom of $\P$. We have the following interlacing relations. 
    \begin{enumerate}[\normalfont(i)]
        \item The $\gamma$-polynomial $\gamma(\uH_{\P}; y)$ interlaces the $\gamma$-polynomial $\gamma(\H_{\P}; y)$
    \[
        \gamma(\uH_{\P}; y) \preceq \gamma(\H_{\P}; y).  
    \]
        \item The $\gamma$-polynomial $\gamma(\H_{[a, \un]}; y)$ interlaces the $\gamma$-polynomial $\gamma(\uH_{\P}; y)$ 
        \[
        \gamma(\H_{[a, \un]}; y) \preceq \gamma(\uH_{\P}; y). 
        \]
        \item The $\gamma$-polynomial $\gamma(\H_{[a, \un]}; y)$ interlaces the $\gamma$-polynomial $\gamma(\H_{\P}; y)$
        \[
        \gamma(\H_{[a, \un]}; y) \preceq \gamma(\H_{\P}; y). 
        \]
    \end{enumerate}
    Furthermore, if $\des_0(\ell) = \ell'$, where $\ell$ (resp. $\ell'$) denotes the number of distinct labels above $\zero$ (resp. any atom), then we have 
    \begin{enumerate}
    \item[\normalfont(iv)] 
     The $\gamma$-polynomial $\gamma(\uH_{[a, \un]}; y)$ interlaces the $\gamma$-polynomial $\gamma(\uH_{\P}; y)$
        \[
        \gamma(\uH_{[a, \un]}; y) \preceq \gamma(\uH_{\P}; y). 
        \]
    \end{enumerate}
\end{maintheorem}

\begin{proof}Let $n$ be the rank of $\P$. \leavevmode
\begin{enumerate}[\normalfont(i),leftmargin=18pt]
    \item By Lemma~\ref{lem:induction}, the sequence $\begin{pmatrix}
    \gamma^n_{1, \asce} & \cdots & \gamma^n_{\ell, \asce} & \gamma^n_{1, \desc} & \cdots & \gamma^n_{\ell, \desc}
    \end{pmatrix}$ is interlacing. By \Cref{prop:interlacing-properties} (Convexity),
    \[
    \left(\sum_{i} \gamma^n_{i, \asce}\right) \: \preceq \: \left(\sum_{j} \gamma^n_{j, \desc}\right). 
    \]
    By \Cref{prop:interlacing-properties} (Convexity) again, 
    \[
    \sum_{i} \gamma^n_{i, \asce} \: \preceq \: \left(\sum_{i} \gamma^n_{i, \asce} + \sum_{j} \gamma^n_{j, \desc} \right), 
    \]
    Therefore, \Cref{prop:gamma-refinement} implies that $\gamma(\uH_{\P}; y) \preceq \gamma(\H_{\P}; y).$

    \item By Lemma~\ref{lem:induction}, the sequence $(\gamma^{n}_{1,\asce},\ldots, \gamma^{n}_{\ell, \asce})$ is interlacing. By \Cref{prop:interlacing-properties} (Convexity),
    \[
    \gamma^{n}_{1, \asce} \preceq \sum_{1 \leqslant i \leqslant \ell}\gamma^n_{i, \asce} = \gamma(\uH_{\P}; y),
    \]
    where the last equality comes from Proposition~\ref{prop:gamma-refinement}. However, there cannot be a descent starting from the smallest label above $\zero$, as this would contradict the lexicographic minimality of the unique maximal chain with increasing labels. This implies the equality $\gamma^n_{1, \asce} = \gamma(\H_{[a, \un]};y)$ which gives the second statement. 
    
    \item By Lemma~\ref{lem:induction}, the sequence $(\gamma^{n}_{1},\ldots, \gamma^{n}_{\ell})$ is interlacing. By convexity of interlacing we obtain 
    $$\gamma^{n}_{1} \preceq \sum_{1 \leqslant i \leqslant \ell}\gamma^n_{i} = \gamma(\H_{\P}; y).$$
    However, by the same argument as in the preceding paragraph we have $\gamma^{n}_{1} = \gamma(\H_{[a, \un]}; y)$ which gives the result. 

    \item Let us assume $\des_0(\ell) = \ell'.$ This implies that we have $\gamma^n_{\ell} = y\,\gamma(\H_{[a, \un]};y)$. By Lemma~\ref{lem:induction} the sequence $(\gamma^n_{1, \asce}, \ldots, \gamma^{n}_{\ell, \asce}, \gamma^{n}_{\ell})$ is interlacing, and so by convexity this implies
    $$\sum_{1 \leqslant i \leqslant \ell}\gamma^n_{i, \asce} = \gamma(\uH_{\P}; y) \preceq \gamma^n_{\ell} = y\,\gamma(\uH_{[a, \un]};y),$$
    which is equivalent to the desired result. \qedhere
\end{enumerate}
\end{proof}

Note that the condition $\des_0(\ell) = \ell'$, or, in other words, the condition that there is no ascent starting from the biggest label above $\zero$, does not have to be true in general, as Example~\ref{ex:diamond-poset} shows. However, that condition is always verified when $\P$ is atomistic. Indeed in that case, if $a$ is an atom with maximal label then for all $a \lessdot t$ in $\P$ by atomisticity of $\P$ the element $t$ is above another atom $a'.$ By definition of $a$ we have $\lambda(\zero, a') \trianglelefteqslant \lambda(\zero, a)$. If $\lambda(\zero, a') \vartriangleleft \lambda(\zero, a)$ then by minimality of the unique maximal chain with increasing labels between $\zero$ and $t$ we must have $\lambda(\zero, a) \vartriangleright \lambda(a, t).$ If $\lambda(\zero, a') = \lambda(\zero, a)$, then we must also have $\lambda(\zero, a) \vartriangleright \lambda(a, t).$ Indeed, if the contrary was true, by the uniqueness of the maximal chain with increasing labels between $\zero$ and $t$ we would have $\lambda(\zero, a') \vartriangleright \lambda(a, t),$ which would contradict the lexicographic minimality of that maximal chain. 

\section{Rank-uniform semimodular supersolvable lattices}
\label{sec:supersolvable}
In this section, we explore the interactions between the notion of rank-uniform labeled posets introduced in this article and the theory of supersolvable lattices (see Definition~\ref{def:supersolvable}) in the special case of semimodular lattices (see Definition \ref{def:semimodular}). The main result of this section shows that all rank-uniform semimodular supersolvable lattices are UMEL-shellable (Theorem~\ref{thm:main4-intro}), and so we may specialize Theorem~\ref{thm:main2-intro} and Theorem~\ref{thm:gamma-RR-mono-rank-up-uniform} to them. This parallels the classical result that characteristic polynomials of supersolvable geometric lattices are real-rooted (see \cite{Stanley1971ModularElements}). We start with some preliminaries.  
\begin{definition}
A finite poset $\Latt$ is called a \emph{lattice} if every pair of elements $s, t \in \Latt$ has a unique supremum in $\Latt$ (denoted $s \vee t$ and called the \emph{join} of $s$ and $t$), and a unique infimum in $\Latt$ (denoted $s \wedge t$ and called the \emph{meet} of $s$ and $t$).
\end{definition} 

\begin{definition}
A pair $(s, t)$ in a lattice $\Latt$ is called \textit{modular} if for all $u \leqslant t \in \Latt$ we have the equality 
\begin{equation*}
    u \vee (s \wedge t) = (u \vee s) \wedge t.
\end{equation*}
An element $s$ in a lattice $\Latt$ is called \textit{left modular} (resp. \textit{right modular}) if for all $t \in \Latt$, the pair $(s,t)$ (resp. $(t,s)$) is modular. An element is called \textit{modular} if it is both left and right modular. A finite lattice is \emph{modular} if all its elements are modular.   
\end{definition}

Modular elements satisfy the following lemma, which we will need later.   
\begin{lemma}[\cite{birkhoff1940lattice}, Section IV.2]\label[lemma]{lemmaisodiamond}
For any elements $s, t$ in some lattice $\Latt$, if $s$ is modular then taking the join with $t$ defines an isomorphism of posets from $[s \wedge t, s]$ to $[t, s\vee t]$, with inverse given by taking the meet with $s$. 
\end{lemma}

The following notion was introduced by Stanley in \cite{stanley1972supersolvable}. 
\begin{definition}\label{def:supersolvable}
A finite lattice $\Latt$ is \textit{supersolvable} if it is graded and if $\Latt$ admits a maximal chain of left modular elements. 
\end{definition}
This is not Stanley's original definition but is equivalent to it by \cite[Theorem 1.1]{FoldesWoodroofe2022}. Supersolvable lattices appear in several areas of mathematics. The motivation behind the terminology comes from the fact that this class of posets contains lattices of subgroups of finite supersolvable groups \cite{stanley1972supersolvable}. Furthermore, from the definition it is clear that modular lattices are supersolvable. For instance, the poset  $\P(\F_q^n)$ consisting of all subspaces of the vector space $(\F_q)^{n}$ over the finite field $\F_q$ ordered by inclusion is modular and so it is supersolvable---the poset $\P(\F_q^n)$ is often called a \emph{projective geometry} or a \emph{vector space lattice}. Further examples of supersolvable lattices in other areas of combinatorics include the non-crossing partition lattices of type A and the class of shuffle posets (see \cite[Chapter~4]{hersh-thesis}). 

In hyperplane arrangement theory, a flat $F$ in the intersection lattice $\L_{\Harr}$ of a complex hyperplane arrangment $\Harr$ is modular if and only if the natural map $\Arr_{\Harr} \rightarrow \Arr_{\Harr^F}$ between the arrangement complement $\Arr_{\Harr}$ of $\Harr$ and the arrangement complement $\Arr_{\Harr^F}$ of the localization $\Harr^F$ is a fibration (see \cite{Paris_2000} for the proof of that result, and see \cite{OT_1992} for a general reference on hyperplane arrangements). Thus, an intersection lattice $\L_{\Harr}$ is supersolvable if and only if there is a tower of codimension $1$ fibrations between $\Arr_{\Harr}$ and $\C^{\ast}$, induced by successive localizations. For instance, the intersection lattice of the $n$-th (type A) braid arrangement is supersolvable. Indeed the arrangement complement of the $n$-th braid arrangement is the $n$-th configuration space $\Conf_n(\C)$ of $\C$, and we have a tower of codimension $1$ fibrations between $\Conf_n(\C)$ and $\C^{\ast}$ given by forgetting points one after the other. The intersection lattice of the $n$-th braid arrangement is isomorphic to the $n$-th partition lattice, which is simply the set of all set partitions of $[n]$ ordered by refinement. \\

In the rest of this paper we will restrict our attention to supersolvable lattices which satisfy in addition the following property. 
\begin{definition}\label{def:semimodular}
A graded finite lattice $\Latt$ with rank function $\rk$ is called semimodular if we have the inequality 
$$\rk(s\vee t) + \rk(s\wedge t) \leq \rk(s) + \rk(t)$$
for all $s,t \in \Latt.$
\end{definition}
In a semimodular lattice, left and right modularity are equivalent (see \cite[p. 83]{birkhoff1940lattice}) and so if a supersolvable lattice is semimodular, then its maximal chain of left modular elements is in fact composed of modular elements. \\

It has been shown by Stanley that supersolvability is a hereditary property. 
\begin{lemma}[{\cite[Proposition 3.2]{stanley1972supersolvable}}]\label[lemma]{lem:hereditySS}
Let $\Latt$ be a semimodular supersolvable lattice with maximal chain of modular elements $(m_i)_{0\leqslant i \leqslant n}$. For any $s \in \Latt$, the set of elements 
\[
s \vee m \coloneqq \{s \vee m_0, \ldots, s \vee m_n\}
\]
forms a maximal chain of modular elements of the interval $[s, \un],$ and the set of elements 
\[
s \wedge m \coloneqq \{s \wedge m_0, \ldots, s \wedge m_n\}
\] forms a maximal chain of modular elements of the interval $[\zero, s]$.  
\end{lemma}
Let $\Latt$ be a semimodular supersolvable lattice with maximal chain of modular elements 
\[
\zero = m_0 \lessdot \cdots  \lessdot m_n = \un.
\]
As explained in \cite[Section 3]{bjorner-el-shellable} one can define an EL-labeling $\lambda_m$ on $\Latt$ by setting 
\begin{equation*}
    \lambda_{m}:\Edge(\Latt) \rightarrow [n], \quad \lambda_{m}(s, t) \coloneqq \min \{ i \, |\, s \vee m_i \geqslant t \}.
\end{equation*}
We will call this labeling the \emph{supersolvable labeling} associated to $\Latt$ and the maximal chain $m$. We will need the two lemmas below about supersolvable labelings. 

\begin{lemma}\label[lemma]{lem:descentSS}
Let $\Latt$ be a semimodular supersolvable lattice with maximal chain of modular elements $(m_i)_{0\leqslant i \leqslant n}$ and associated supersolvable labeling $\lambda_m$. For every atom $s$ of $\Latt$, 
\[
\des(\zero, s) = \Ind(\zero, s) - 1.
\]
\end{lemma}

\begin{proof}
Denote by $\lambda_1 \vartriangleleft \lambda_2 \vartriangleleft \cdots $ the labels of the atoms of $\Latt.$ For all $i < \Ind(\zero, s)$ there exists an atom $s_i$ of $\Latt$ such that $\lambda_m(\zero, s_i) = \lambda_i$. By modularity of $m_{\lambda_m(\zero,s)-1}$ and Lemma~\ref{lemmaisodiamond}, taking the join with $s$ defines an isomorphism from $[\zero, m_{\lambda_m(\zero,s)-1}]$ to $[s, m_{\lambda_m(\zero,s)-1}]$. This implies that for all $i < \Ind(\zero, s)$ the element $s\vee s_i$ must cover $s$, and $s_i$ is the unique element below $m_{\lambda_m(\zero,s)-1}$ such that its join with $s$ is $s\vee s_i$. This implies that we have $\lambda_m(s, s \vee s_i) = \lambda_i$. This proves that every label $\lambda_i$ with $i < \Ind(\zero, s)$ appears as a label above $s$, and so we have the inequality $\des(\zero, s) \geqslant \Ind(\zero, s) - 1.$ 

For the converse inequality, if $s\lessdot t$ is a covering relation such that $\lambda_m(s,t) \vartriangleleft \lambda_m(\zero, s)$ by the definition of the supersolvable labeling $\lambda_m$ this means that $m_{\lambda_m(\zero,s)-1}\vee s \geqslant t$ and so by Lemma~\ref{lemmaisodiamond} we get the existence of an atom of $\Latt$ having label $\lambda_m(s,t)$. This proves that every label appearing as a descent above $s$ also appears as the label of an atom of $\Latt$ and so we get $\des(\zero, s) \leqslant \Ind(\zero, s) - 1.$
\end{proof}
\begin{lemma}\label[lemma]{lem:samelabelSS}
Let $\Latt$ be a semimodular supersolvable lattice with maximal chain of modular elements $(m_i)_{0\leqslant i \leqslant n},$ and let $s$ be some element of $\Latt.$ There exists a monotonic injection 
\[
\phi_s: [n - \rk s] \hookrightarrow [n]
\] such that for every covering relation $t \lessdot u \in [s, \un]$ we have the equality 
\[
\phi_s(\lambda_{s\vee m}(t,u)) = \lambda_m(t,u).
\] In particular $([s, \un], (\lambda_m)_{s,\un})$ and $([s,\un], \lambda_{s\vee m})$ have the same width and descent numbers. 
\end{lemma}
\begin{proof}
Denote by $m^s_0 = s \lessdot \cdots \lessdot \un = m^s_{n- \rk s}$ the elements of the set $s\vee m$ (see Lemma~\ref{lem:hereditySS}). For all $i\in [n - \rk s]$ put $\phi_s(i) \coloneqq \min \{ j \in [n] \, | \, m_j \vee s = m^s_i\}$. One can see that $\phi_s$ is a monotonic injection from $[n - \rk(s)]$ to $[n]$. Let us prove the equality $\lambda_m(t,u) =  \phi_s(\lambda_{s\vee m}(t,u))$ for all covering relation $t \lessdot u$ in $[s, \un]$, by proving the inequality in both direction. 

Denote by $k$ the integer such that $s\vee m_{\lambda_m(t,u)} = m^s_k$. Since $m^x_k\vee t = m_{\lambda(t,u)}\vee t \geqslant u,$ we have $k \geqslant \lambda_{s\vee m}(t,u).$ Furthermore we have $\phi_s(k) = \lambda_m(t,u),$ because if we had $m_{\lambda(t,u) -1} \vee s = m^s_k,$ then  we would have $m_{\lambda(t,u)-1} \vee t \geqslant u,$ contradicting the definition of $\lambda_m(t,u).$ By monotonicity of $\phi^s$ this implies that we have $\lambda_m(t,u) = \phi^s(k) \trianglerighteqslant \phi^s( \lambda_{s\vee m}(t,u)).$ 

On the other hand by definition we have $m^s_{\lambda_{s\vee m}(s,t)}\vee t \geqslant u,$ which implies that we have $m_{\phi^s(\lambda_{s\vee m}(t,u))}\vee t \geqslant u$ which implies that we have $\lambda_m(t,u) \trianglelefteqslant \phi^s(\lambda_{s\vee m}(t,u)).$
\end{proof}

If $\Latt$ is a semimodular supersolvable lattice with maximal chain of modular elements $(m_i)_{0\leqslant i \leqslant n}$, notice that Lemma~\ref{lemmaisodiamond} implies that for any atom $s\gtrdot \zero$ which is not below $m_{n-1}$ we have an isomorphism of posets $[\zero, m_{n-1}] \simeq [s, \un]$ given by taking the join with $s$, and this isomorphism preserves the induced labelings on both sides. This means that semimodular supersolvable lattices with their supersolvable labelings are at least partially rank-uniform, because for instance $\omega_{\medbullet}(i)$ is constant on rank $1$ elements not below $m_{n-1},$ for all $i$. However, they do not have to be rank-uniform in general. The following theorem shows that for semimodular supersolvable lattices, in order to check lexicographic rank-uniformity one only needs to check classical rank-uniformity, and in that case the labeling is automatically monotonic. A special subclass of rank-uniform supersolvable geometric lattices was studied in detail in~\cite{Damiani1994}.

\newtheorem*{thm:intro4}{Theorem~\ref{thm:main4-intro}}
\begin{thm:intro4}
{\itshape Every rank-uniform semimodular supersolvable lattice is UMEL-shellable.}
\end{thm:intro4}
\begin{proof}

We first check the uniformity of widths. For all $0 \leqslant k \leqslant n$ let us denote by $\omega^{k}$ the width in the labeled poset $([\zero, m_{k}], \lambda_m)$. Let us prove by induction on the rank $n$ of $\Latt$ that for any atom $s$ of $\Latt$ we have $\omega_s= \omega^{n-1}_{\zero}$. For $s \not\leqslant m_{n-1}$ as explained in the previous paragraph this is a consequence of the isomorphism $[s, \un] \simeq [\zero, m_{n-1}]$ given by Lemma~\ref{lemmaisodiamond}. That isomorphism also implies that $[\zero, m_{n-1}]$ is a rank-uniform lattice, which is also supersolvable by Lemma~\ref{lem:hereditySS}. This means that by induction we can assume that we have the equality $\omega^{n-1}_s = \omega^{n-2}_{\zero}$ for every atom $s$ of $[\zero, m_{n-1}].$ Consequently, for every atom $s$ of $\Latt$ below $m_{n-1}$ and every $i \leqslant n-2$ we have 
\begin{equation*}
    \omega_{s}(i) = \omega^{n-1}_s(i) = \omega^{n-2}_{\zero}(i) = \omega^{n-1}_{\zero}(i).
\end{equation*}
For every $i \geqslant n$ we have $\omega_s(i) = 0 = \omega^{n-1}_{\zero}(i)$ by Lemma~\ref{lem:samelabelSS}. By rank-uniformity every atom of $s$ is covered by the same number of elements which is $\sum_{1 \leqslant i \leqslant n-1} \omega_{\zero}(i).$ This implies that we must have $\omega_{s}(n-1) = \omega^{n-1}_{\zero}(n-1),$ which concludes the proof. By a direct induction one gets that for all $s \in \Latt$ one has the equality $\omega_{s} = \omega^{n-\rk(s)}_{\zero}$ which gives the desired uniformity. 

By Lemmas~\ref{lem:descentSS}~and~\ref{lem:samelabelSS}, for all $s \lessdot t \in \Latt$ one has $\des(s,t) = \Ind(s,t) - 1$. This proves both the uniformity of the descent numbers of $(\Latt, \lambda_m)$ and the monotonicity of $(\Latt, \lambda_m).$
\end{proof}

\begin{corollary}
Let $\Latt$ be a rank-uniform semimodular supersolvable lattice. The $h$-polynomial of the order complex of  $\Latt$, the Chow polynomials of $\Latt$ and $\Latt^{\ast}$, and the augmented Chow polynomials of $\Latt$ and $\Latt^{\ast}$ are real-rooted.
\end{corollary}

\begin{corollary}
For every integer $n \geqslant 2$ and every finite group $\G$, the $h$-polynomial of the order complex of $\QnG$, the Chow polynomial of $\QnG$ and the augmented Chow polynomial of $\QnG$ are real-rooted. 
\end{corollary}
For the $h$-polynomial the statement already appeared in \cite[Section 4]{branden-saud1} (see also \cite{athanasiadis-kalampogia} for the two special cases $G=\{0\},\Z/2\Z$). For the Chow polynomial the result appeared concurrently and independently in \cite[Section 6]{branden-vecchi2} using different methods. 
\begin{proof}
It was mentioned in Example~\ref{ex:dowling-geo} and proved by \cite[Theorem 4]{dowling1973class} that Dowling geometries are semimodular and supersolvable. By \cite[Theorem 2(a)]{dowling1973class}, if $F\in \QnG$ has corank $r$ then we have an isomorphism of posets $[F, \un] \simeq \Qd_r(\G).$ This shows that $\QnG$ is rank-uniform and we can conclude by Theorem~\ref{thm:main4-intro}, Theorem~\ref{thm:main2-intro}, and Theorem~\ref{thm:gamma-RR-mono-rank-up-uniform}. 
\end{proof}

\begin{example}[Chow polynomials for braid matroids of type \(B_n\)]
As explained in \Cref{ex:dowling-geo}, when $\G = \Z_2$, the Chow polynomial $\uH^{\textrm{B}}_{n}(x)$ is equal to the Chow polynomial $\uH_{\QnG}(x)$. We list the first few explicit $\uH^{\textrm{B}}_{n}(x)$, which only have real and nonpositive roots by our corollary above. 

\[
\uH_{n}^{\textrm{B}}(x) =
\begin{cases}
x^2 + 14x + 1, & n = 3,\\[4pt]
x^3 + 99x^2 + 99x + 1, & n = 4,\\[4pt]
x^4 + 622x^3 + 3162x^2 + 622x + 1, & n = 5,\\[4pt]
x^5 + 4051x^4 + 65812x^3 + 65812x^2 + 4051x + 1, & n = 6,\\[4pt]
x^6 + 28590x^5 + 1205199x^4 + 3724100x^3 + 1205199x^2 + 28590x + 1, & n = 7.\\[4pt]
\end{cases}
\]
\end{example}

\section{Rank-selection of rank-uniform labeled posets}\label{sec:rank-selection}
Let $\P$ be a finite graded bounded poset of rank $n$. For any subset $S \subseteq [n-1]$ one can consider the rank-selected subposet $\P_S$ of $\P$ at $S$, which is
\begin{equation*}
    \P_S \coloneqq \{ s \in \P \, |\, \rk(s) \in S\sqcup \{0, n \}\}.
\end{equation*}
In particular, when $S = [k]$ for some $k \in [n-1]$, the poset $\P_{S}$ is called the \emph{$(n-k-1)$-th truncation of $\P$} and will be denoted $\Tr^{n-k-1}(\P)$, and when $S = [n-2]$, the poset $\P_S = \Tr^{1}(\P)$ is called the \emph{truncation of $\P$}, often denoted as $\Tr(\P)$. In this section we explore the interactions between the notion of labeled rank-uniform poset introduced in this article, and rank-selection. Our first result is the following.
\begin{proposition}\label{prop:truncation}
Let $(\P, \lambda)$ be a rank-uniform labeled poset of rank $n$. For all $1 \leqslant k \leqslant n-1$, the truncation $\Tr^k(\P)$ admits an EL-labeling $\lambda^k$ such that $(\Tr^k(\P), \lambda^k)$ is a rank-uniform labeled poset. Furthermore, if $(\P, \lambda)$ is monotonic, then so is $(\Tr^k(\P), \lambda^k).$
\end{proposition}
\begin{proof}
By induction it is enough to prove the result for $k = 1$. In that case $\Tr^1(\P)$ is simply the poset $\P$ with all the elements of rank $n-1$ removed. We set $\lambda^1(s,t) = \lambda(s,t)$ for all $s\lessdot t \in P$ with $\rk(t) \leqslant n-2.$ If $s\in \P$ has rank $n-2$, then there exists a unique element $s'\in \P$ of rank $n-1$ such that we have $\lambda(s,s') \trianglelefteqslant \lambda(s', \un).$ We then set $\lambda^1(s, \un) = \lambda(s,s').$ One can check that $\lambda^1$ is an EL-labeling of $\Tr^1(\P)$. Let us check that $(\Tr^1(\P), \lambda^1)$ is rank-uniform. For the widths we only have to check uniformity at rank $n-2,$ which is immediate since in $\Tr^1(\P)$ there is a unique covering relation above an element of rank $n-2$. For the descent numbers we only need to check uniformity at rank $n-3,n-2.$ Let us denote by $\des^1$ the descent function in $\Tr^1(\P).$ For all $s\lessdot t \in \P$ with $s$ of rank $n-3$ we have $\des^1(s,t) = 1$ when $\des_{n-3}(\Ind(s,t)) > 0$ and $\des^1(s,t) = 0$ otherwise, which proves uniformity. This computation also shows that if $(\P, \lambda)$ is monotonic, then $\des_{n-3}(\medbullet)$ is weakly-increasing, and so $\des^1_{n-3}(\bullet)$ is too, which implies that $(\Tr^1(\P), \lambda^1)$ is monotonic as well. 
\end{proof}
This has the following consequence. 
\begin{corollary}
    Let $\P$ be a UMEL-shellable poset. The $h$-polynomial, the Chow polynomial and the augmented Chow polynomial of $\Pdual$ are real-rooted.
\end{corollary}
\begin{proof}
For the $h$-polynomial it is immediate from the definitions that $h_{\P}= h_{\Pdual}$ and so the result follows from Theorem~\ref{thm:main2-intro}. For the Chow polynomial, by Lemma~\ref{lemma:duality} the result follows from Proposition~\ref{prop:truncation} and Theorem~\ref{thm:gamma-RR-mono-rank-up-uniform}. For the augmented Chow polynomial it suffices to apply Lemma~\ref{lemma:duality} and Theorem~\ref{thm:gamma-RR-mono-rank-up-uniform}. 
\end{proof}

We were not able to prove or disprove that monotonic rank-uniform labeled posets are closed under general rank-selection, so we leave it as a question. 

\begin{question}\label{question:rank-selected}
Let $(\P, \lambda)$ be a monotonic rank-uniform labeled poset of rank $n$, and let $S$ be a subset of $[n-1]$. Does the rank selected subposet $\P_S$ admit an EL-labeling $\lambda^S$ such that $(\P_S, \lambda^S)$ is a monotonic rank-uniform labeled poset? 
\end{question}

By  \cite[Theorem 4.1.1]{Li2020LexicographicShellablePosets}, if a ranked poset $\P$ has an EL-labeling $\lambda$, then every rank-selected subposet $\P_S$ admits an explicit EL-labeling $\lambda^S$. One can show that if $(\P, \lambda)$ is rank-uniform then so is $(\P_S, \lambda^S)$. However it is not true that the monotonicity of $(\P, \lambda)$ implies that of $(\P_S, \lambda^S)$, as can be seen already for boolean lattices for instance. Nevertheless, it is still possible that one might be able to find another EL-labeling of $\P_S$ which preserves monotonicity. 

Despite this lack of result, we still have the following theorem, which generalizes Theorem~\ref{thm:gamma-RR-mono-rank-up-uniform}.

\newtheorem*{thm:intro3}{Theorem~\ref{thm:main3-intro}}
\begin{thm:intro3}
    {\itshape 
    Let $\P$ be UMEL-shellable poset and let $\P_S$ be a rank-selected subposet of $\P$. Then the following holds:
    \begin{enumerate}[\normalfont (i)]
        \item The Chow polynomials $\uH_{\P_S}(x)$ and $\uH_{\P_S^*}(x)$ have nonpositive real roots.
        \item The augmented Chow polynomial $\H_{\P_S}(x)$ has nonpositive real roots.
        \item The roots of the polynomials $\uH_{\P_S}(x)$ and $\H_{\P_S}(x)$ interlace.
        \item The $h$-polynomial of the order complex $\Delta(\P_S)$ has nonpositive real roots.
    \end{enumerate}
    }
\end{thm:intro3}
The proof consists of a modification of the proof of \cite[Theorem 3.1]{athanasiadis-douvropoulos-kalampogia}, adding the same ideas that we have used for the proofs of Theorems~\ref{thm:gamma-RR-mono-rank-up-uniform}~and~\ref{thm:main2-intro}. 

\begin{proof}
We start with the statement on the $h$-polynomial. Let us denote by $h^S(y)$ the $h$-polynomial of the order complex of $\P_S$. For any $s\in \P$ and any integer $i \geqslant1$ we denote 
\begin{align*}
    \calC^{S,s}_i \coloneqq \{ C \colon s = s_0 \lessdot s_1 \lessdot \cdots \lessdot s_{n - \rk(s)} = \un \,  | \, \Des(C) \subseteq S, \Ind(s_0, s_1) = i \}
\end{align*}
and for every $s \in \P$ and for all $S \subseteq [n - \rk(s) - 1]$, we denote the sequence of polynomials 
\[
\mathbf{h}^{S,s}(y) \coloneqq \left( \sum_{C \in \calC^{S,s}_i}y^{\des(C)}\right)_{i \geqslant 1}. 
\]
Since
\[
h^{S}(y) = \sum_{T \subseteq S} \beta_{\P}(T)y^{|T|} = \sum_{i \geqslant1} \sum_{C \in \calC^{S, \zero}_i}y^{\des(C)} = \sum_{i \geq 1}h^{S,s}_i(y),
\]
it is sufficient to prove that the sequence of polynomials $\mathbf{h}^{S, \zero}$ is an interlacing sequence, which we prove by induction on the rank and the set $S$. By uniformity of $(\P, \lambda)$ the polynomials $h_{i}^{S,s}(y)$ only depend on the corank $k$ of $s$, and so we write $h_{i}^{S,k}(y)$ instead. For any $s$ of corank $k$ and any $S \subseteq [k-1]$, if $1 \in S$ then by partitioning the elements of $\calC_{s, S}(i)$ according to their second label we get the recursive formula 
\begin{equation*}
    \mathbf{h}^{S, k}(y) = \mathbf{A} (y) \cdot \mathbf{h}^{(S\setminus 1)-1,k-1}(y)
\end{equation*}
where $\mathbf{A}(y) = \begin{pmatrix} a_{ij}(y)\end{pmatrix}$ is the matrix of polynomials defined by 
\[
a_{ij}(y) = \begin{cases}
\omega_{n-k}(i) \: y & \text{ if } j \leqslant \des_{n-k}(i), \\
\omega_{n-k}(i) & \text{ if } j > \des_{n-k}(i).
\end{cases}
\] By Lemma~\ref{lem:two-column-monotonicity-interlacing} the matrix $\mathbf{A}$ satisfies the hypotheses of Theorem~\ref{thm:interlacing} and so that theorem allows us to conclude by induction. On the other hand if $1 \notin S$, we have the recursion formula 
\begin{equation*}
    \mathbf{h}^{S,k}(y) = \mathbf{A} \cdot \mathbf{h}^{S-1, k-1}(y)
\end{equation*}
where $\mathbf{A} = (a_{ij})$ is the matrix of real numbers defined by $a_{ij}(y) = 0$ if $j \leqslant\des_{n-k}(i)$ and $a_{ij}(y) =\omega_{n-k}(i)$ otherwise. One can check that the matrix $\mathbf{A}$ satisfies the hypotheses of Theorem~\ref{thm:interlacing} and so that theorem allows us to conclude by induction.

We prove together the statements on the augmented and non-augmented Chow polynomials of $\P_{S}$. By Lemma~\ref{lem:gamma-transfo-RR} it is enough to prove that $\gamma(\uH_{\P_{S}};y)$ and $\gamma(\H_{\P_S}; y)$ are real-rooted. For all $T \subseteq S$ we say that $T$ is $S$-stable if $T$ does not contain two elements which are consecutive in $S.$ \footnote{Two elements $i, j$ with $i < j$ are consecutive in an ordered set $S$ if no $k \in S$ satisfies $i < k < j$.} For all $s \in \P$ and all $i\geqslant1$ we introduce the following notations. 
{\small
\begin{align*}
    \calC^{\ndd, S ,s} & \coloneqq \{ C = \{s_0 = s \lessdot \cdots \lessdot s_{n - \rk s} = \un\} \, | \, \Des(C) \subseteq S, \, \Des(C)  \,S\textrm{-stable}\}, \\
    \calC^{\ndd, S ,s}_{i} & \coloneqq \{ C = \{s_0 = s \lessdot \cdots \lessdot s_{n - \rk s} = \un\} \, | \, C \in \calC^{\ndd, S ,s}, \Ind(s_0, s_1) = i\}, \\
    \calC^{\ndd, S ,s}_{i, \asce} & \coloneqq \{ C = \{s_0 = s \lessdot \cdots \lessdot s_{n - \rk s} = \un \}\, | \, C \in \calC^{\ndd, S ,s}, \Ind(s_0, s_1) = i, \min S \notin \Des(C) \}, \\
    \calC^{\ndd, S ,s}_{i, \desc} & \coloneqq \{ C = \{s_0 = s \lessdot \cdots \lessdot s_{n - \rk s} = \un \}\, | \, C \in \calC^{\ndd, S ,s}, \Ind(s_0, s_1) = i, \min S \in \Des(C) \}.
\end{align*}
}
By Theorem~\ref{thm:gamma-flag-h-vector} we have the equalities of polynomials
\begin{align*}
    \gamma(\uH_{\P_S}; y) & = \sum_{\substack{S\textrm{-stable } T \\ \min(S) \notin T}} \beta(T)y^{|T|} = \sum_{C \in \calC^{\ndd, S ,s}_{i, \asce}} y^{\des(C)},
\end{align*} 
and 
\begin{align*}
    \gamma(\H_{\P_S}; y) &= \sum_{S \text{-stable } T} \beta(T)y^{|T|} = \sum_{C \in \calC^{\ndd, S ,s}_{i}} y^{\des(C)}. 
\end{align*} 
Let us prove by induction that for all $s \in \P$ and for all $S \subseteq [n - \rk(s) -1]$, the sequences of polynomials 
\[
\gamma^{S, s}_{i,\asce}(y) \coloneqq \left(\sum_{C \in \calC^{\ndd, S ,s}_{i, \asce}} y^{\des(C)} \right)_{i \geqslant 1}
\quad \text{ and } \quad 
\gamma^{S, s}_i(y) \coloneqq \left(\sum_{C \in \calC^{\ndd, S ,s}_{i}} y^{\des(C)} \right)_{i \geqslant 1}
\] are interlacing. By uniformity of $\P$ the polynomials $\gamma^{S, s}_{i,\asce}$ and $\gamma^{S, s}_{i}$ only depend on the corank $k$ of $s$ and so we denote them by $\gamma^{S, k}_{i,\asce}$ and $\gamma^{S, k}_{i}$ respectively. We will prove both those interlacing statements by proving that the following diagram 
\begin{center}
\begin{equation}
\raisebox{-2.3cm}{
\begin{tikzpicture}[scale=2,every path/.style={line width=0.8pt}]
    \node[] (A) at (0,0) {$\gamma^{S,k}_{1, \asce}$};
    \node[] (B) at (1,0) {$\gamma^{S,k}_{2, \asce}$};
    \node[] (C) at (2,0) {$\cdots$};
    \node[] (D) at (3,0) {$\gamma^{S,k}_{\ell-1, \asce}$};
    \node[] (E) at (4,0) {$\gamma^{S,k}_{\ell, \asce}$};
    \node[] (F) at (0,-1) {$\gamma^{S,k}_{1}$};
    \node[] (G) at (1,-1) {$\gamma^{S,k}_{2}$};
    \node[] (H) at (2,-1) {$\cdots$};
    \node[] (I) at (3,-1) {$\gamma^{S,k}_{\ell-1}$};
    \node[] (J) at (4,-1) {$\gamma^{S,k}_{\ell}$};
    \node[] (K) at (0,-2) {$\gamma^{S,k}_{1, \desc}$};
    \node[] (L) at (1,-2) {$\gamma^{S,k}_{2, \desc}$};
    \node[] (M) at (2,-2) {$\cdots$};
    \node[] (N) at (3,-2) {$\gamma^{S,k}_{\ell-1, \desc}$};
    \node[] (O) at (4,-2) {$\gamma^{S,k}_{\ell, \desc}$};
    \node[] () at (2, -0.5) {$\cdots$};
    \node[] () at (2, -1.5) {$\cdots$};
    \draw[->] (A) -- (B);
    \draw[->] (B) -- (C);
    \draw[->] (D) -- (E);
    \draw[->] (C) -- (D);
    \draw[->] (F) -- (G);
    \draw[->] (G) -- (H);
    \draw[->] (H) -- (I);
    \draw[->] (I) -- (J);
    \draw[->] (A) -- (F);
    \draw[->] (B) -- (G);
    \draw[->] (D) -- (I);
    \draw[->] (E) -- (J);
    \draw[->] (K) -- (L);
    \draw[->] (L) -- (M);
    \draw[->] (M) -- (N);
    \draw[->] (N) -- (O);
    \draw[->] (F) -- (K);
    \draw[->] (G) -- (L);
    \draw[->] (I) -- (N);
    \draw[->] (J) -- (O);
    \draw[->,smooth, tension=0.6] plot coordinates {
    (3.8, -0.1)   
    (3.0, -0.3)   
    (2.4, -0.75)  
    (1.7, -0.9)   
    (1.2, -1.4)   
    (0.6, -1.5)   
    (0.2, -1.9)   
    };
\end{tikzpicture} 
}
\tag{D}
\label{eq:inductive-step-rkselec}
\end{equation}
\end{center}
is interlacing, where $\ell$ denotes the number of distinct labels above any element of corank $k$. The base cases are routine checks and similar to the base cases of \Cref{thm:gamma-RR-mono-rank-up-uniform}. For the induction step, if $1 \in \min S$, the result follows by analogous arguments as in the proof of Lemma~\ref{lem:induction}. If $1 \notin \min S$, then to show that the diagram \eqref{eq:inductive-step-rkselec} has  interlacing families along every directed path, we prove the following statements.
\begin{enumerate}[(i)]
    \item \label{item:top-row-interlacing} The top row of \eqref{eq:inductive-step-rkselec} is interlacing. 
    \item \label{item:mid-row-interlacing} The middle row of \eqref{eq:inductive-step-rkselec} is interlacing.
    \item \label{item:bot-row-interlacing} The bottom row of \eqref{eq:inductive-step-rkselec} is interlacing.
    \item \label{item:top-to-mid} For all $1\leqslant i \leqslant j \leqslant \ell$ we have $\gamma^{S,k}_{i, \asce} \preceq \gamma^{S, k}_{j}.$
    \item \label{item:top-to-bot} For all $1 \leqslant i , j \leqslant \ell$ we have $\gamma^{S,k}_{i, \asce} \preceq \gamma^{S, k}_{j, \desc}.$
    \item \label{item:mid-to-bot} For all $1 \leqslant i \leqslant j \leqslant \ell$ we have $\gamma^{S,k}_{i} \preceq \gamma^{S, k}_{j, \desc}.$
\end{enumerate}

For \ref{item:top-row-interlacing}, \ref{item:mid-row-interlacing}, \ref{item:bot-row-interlacing} we have the inductive formulas 
\[
\pmb{\bm{\gamma}}^{S, k}_{\asce} = \mathbf{A} \cdot  \pmb{\bm{\gamma}}^{S-1, k-1}_{\asce}, \quad \quad \pmb{\bm{\gamma}}^{S, k} = \mathbf{A} \cdot \pmb{\bm{\gamma}}^{S-1, k-1}, \quad \text{ and } \quad \pmb{\bm{\gamma}}^{S, k}_{\desc} = \mathbf{A} \cdot \pmb{\bm{\gamma}}^{S-1, k-1}_{\desc},
\] where the matrix $\mathbf{A} = \begin{pmatrix} a_{ij}\end{pmatrix}$ is defined by 
\[
a_{ij} = \begin{cases}
\omega_{n-k}(i) & \text{ if } j > \des_{n-k}(i) \\
0 & \text{ if } j \leqslant \des_{n-k}(i). 
\end{cases}
\]
This gives the result by Theorem~\ref{thm:interlacing}. 
\smallskip

For~\ref{item:top-to-mid}, if $1 \leqslant i \leqslant j \leqslant \ell$ we have the equality
{\small
\begin{multline*}
\begin{pmatrix}
    \gamma^{S,k}_{i, \asce} \\
    \gamma^{S,k}_{j}
\end{pmatrix}
=
\begin{pmatrix}
\omega_{n-k}(i) & \cdots & \omega_{n-k}(i) & \omega_{n-k}(i) & \cdots & \omega_{n-k}(i) & 0 & \cdots & 0  \\
0 & \cdots & 0& \omega_{n-k}(j) & \cdots & \omega_{n-k}(j) & \omega_{n-k}(j) & \cdots & \omega_{n-k}(j) 
\end{pmatrix} \\
\setlength\arraycolsep{2pt}
\begin{pmatrix}
\gamma^{S-1, k-1}_{\des_{n-k}(i)+1, \asce} &\cdots &
\gamma^{S-1, k-1}_{\des_{n-k}(j), \asce} &
\gamma^{S-1, k-1}_{\des_{n-k}(j) + 1, \asce} &
\cdots &
\gamma^{S-1, k-1}_{\ell', \asce} & 
\gamma^{S-1, k-1}_{\des_{n-k}(j)+1, \desc} &
\cdots &
\gamma^{S-1, k-1}_{\ell', \desc} 
\end{pmatrix}^{\top}.
\end{multline*}
}

By our inductive hypothesis the sequence of polynomials in the right hand side is interlacing and we can conclude by Theorem~\ref{thm:interlacing}. 
\smallskip

For \ref{item:top-to-bot}, if $1 \leqslant i, j \leqslant \ell$ we have the equality
{\smaller
\begin{multline*}
\begin{pmatrix}
    \gamma^{S,k}_{i, \asce} \\
    \gamma^{S,k}_{j, \desc}
\end{pmatrix}
=
\begin{pmatrix}
\omega_{n-k}(i) & \cdots & \omega_{n-k}(i) & 0 & \cdots & 0  \\
0 & \cdots & 0& \omega_{n-k}(j) & \cdots & \omega_{n-k}(j)  
\end{pmatrix} \\
\begin{pmatrix}
\gamma^{S-1, k-1}_{\des_{n-k}(i)+1, \asce} &\cdots &
\gamma^{S-1, k-1}_{\des_{n-k}(\ell'), \asce} &
\gamma^{S-1, k-1}_{\des_{n-k}(j)+1, \desc} &
\cdots &
\gamma^{S-1, k-1}_{\ell', \desc} 
\end{pmatrix}^{\top}.    
\end{multline*}
}
By our inductive hypothesis the vector of polynomials in the right hand side is interlacing and we can conclude by Theorem~\ref{thm:interlacing}. 
\smallskip

For \ref{item:mid-to-bot}, if $1 \leqslant i \leqslant n$ notice that $\gamma^{S, k}_{i} = \gamma^{S,k}_{i, \desc} + \gamma^{S, k}_{i, \asce}$ and so the result follows from \ref{item:bot-row-interlacing}, \ref{item:top-to-mid},  and convexity of interlacing (see Proposition~\ref{prop:interlacing-properties}). 
\end{proof}

As mentioned in the introduction, one special case of rank selection is the truncation (selecting all ranks except for the corank $1$) $\Tr(\P)$ of $\P$. Both the (augmented) Chow polynomials of $\Tr(\P)$ have tight connections with poset duality and the original poset.
\begin{thm}\label{thm:gamma-truncation-interlacing}
    Let $\P$ be a UMEL-shellable poset and let $\Tr(\P)$ be the truncation of $\P$. 
    The following interlacing diagrams of Chow polynomials and their $\gamma$-polynomials hold: 
    \[
    \begin{tikzcd}
    \uH_{\Tr(\P)} \arrow[r, "\preceq"] \arrow[d, "\preceq"] & \uH_{\P} \arrow[d, "\preceq"] \\
     \H_{\Tr(\P)} \arrow[r, "\preceq"] & \H_{\P}
    \end{tikzcd}, \quad \text{ and } \quad \begin{tikzcd}
    \gamma(\uH_{\Tr(\P)}; y) \arrow[r, "\preceq"] \arrow[d, "\preceq"] & \gamma(\uH_{\P}; y) \arrow[d, "\preceq"] \\
     \gamma(\H_{\Tr(\P)}; y) \arrow[r, "\preceq"] & \gamma(\H_{\P}; y)
    \end{tikzcd}.
    \]
    Equivalently, the following diagrams hold: 
    \[
    \begin{tikzcd}
    \uH_{\Tr(\P)} \arrow[r, "\preceq"] \arrow[d, "\preceq"] & \uH_{\P} \arrow[d, "\preceq"] \\
    \uH_{\P^{\ast}} \arrow[r, "\preceq"] & \H_{\P^{\ast}}
    \end{tikzcd}, \quad \text{ and } \quad 
    \begin{tikzcd}
    \gamma(\uH_{\Tr(\P)}; y) \arrow[r, "\preceq"] \arrow[d, "\preceq"] & \gamma(\uH_{\P};y) \arrow[d, "\preceq"] \\
    \gamma(\uH_{\P^{\ast}};y) \arrow[r, "\preceq"] & \gamma(\H_{\P^{\ast}};y)
    \end{tikzcd}
    \]
\end{thm}

\begin{proof}
    We prove the statement by setting up a cubical interlacing diagram of interlacing polynomials of all the polynomials involved. See \Cref{fig:gamma-cube}. By \Cref{lem:gamma-transfo-RR}, it is enough to prove the statement for the corresponding $\gamma$-polynomials. 
    By \Cref{thm:gamma-flag-h-vector} and the fact that, for stable $S \subseteq \{2, \ldots, n-2\}$, $\beta_{\Tr(\P)}(S) = \beta_{\P}(S)$, we have the following equalities of polynomials 
    \[
    \gamma(\uH_{\Tr(\P)}; y) = \sum_{\substack{S \subseteq \{2, \ldots, n-2\}\\ S \text{ stable}}} \beta_{\Tr(\P)}(S) \: y^{\abs{S}} = \sum_{\substack{S \subseteq \{2, \ldots, n-2\}\\ S \text{ stable}}} \beta_{\P}(S) \: y^{\abs{S}}.
    \]
    The result then follows from \Cref{lem:gamma-cube-interlacing}, the above expression of $\gamma(\uH_{\Tr(\P)})$ and the analogous expression for $\gamma(\H_{\Tr(\P)})$, and \Cref{lemma:duality}. 
\end{proof}

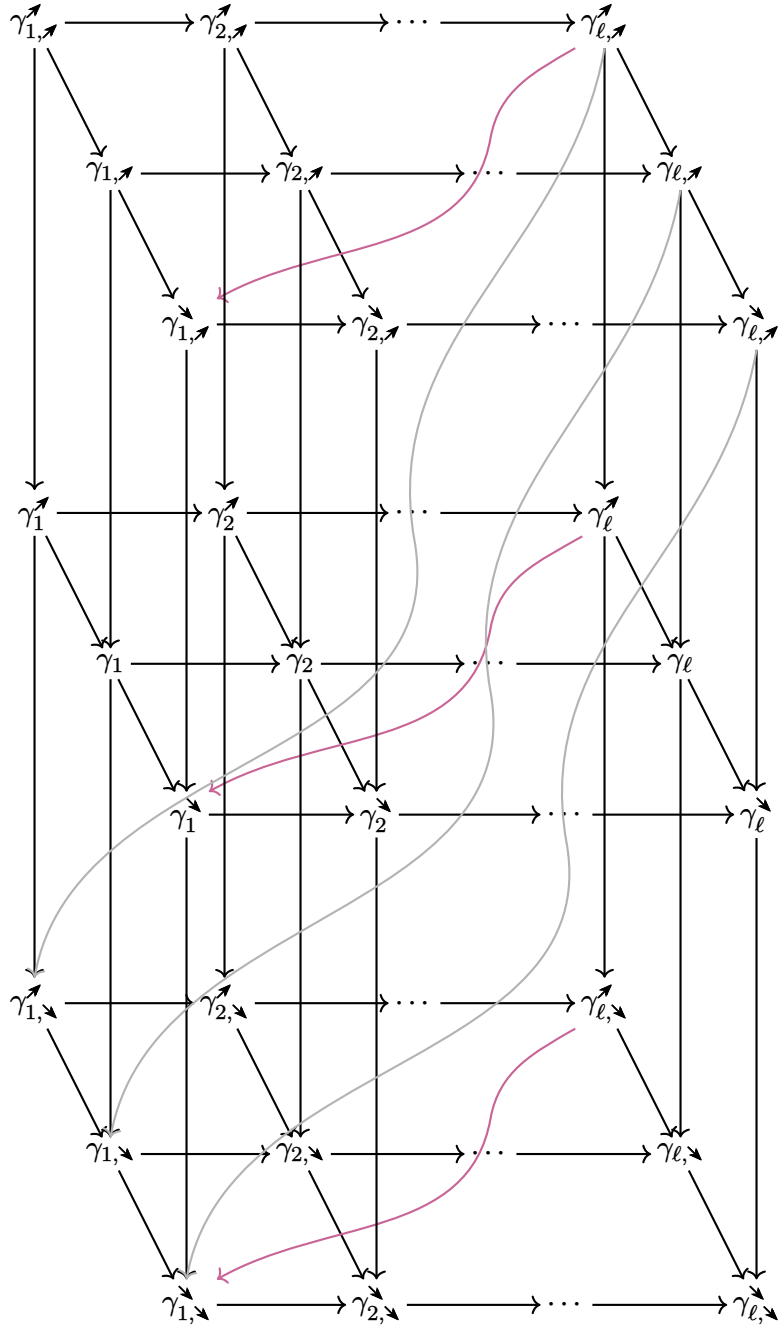
\begin{figure}[htbp]
    \centering
    \begin{tikzpicture}[
        baseline=(current bounding box.center),
        every node/.style= {inner sep=2pt},
        every path/.style={line width=0.8pt}
    ]

    \def\xsep{2.5}
    \def\ysep{2}
    \def\slant{1}
    \def\layersep{6.5}  


    \node (L1g11) at (0, 0) {$\gamma^{\asce}_{1,\asce}$};
    \node (L1g12) at (\xsep, 0) {$\gamma^{\asce}_{2,\asce}$};
    \node (L1d1) at (2*\xsep, 0) {$\cdots$};
    \node (L1g1l) at (3*\xsep, 0) {$\gamma^{\asce}_{\ell,\asce}$};

    \node (L1g21) at (\slant, -\ysep) {$\gamma_{1,\asce}$};
    \node (L1g22) at (\xsep + \slant, -\ysep) {$\gamma_{2,\asce}$};
    \node (L1d2) at (2*\xsep + \slant, -\ysep) {$\cdots$};
    \node (L1g2l) at (3*\xsep + \slant, -\ysep) {$\gamma_{\ell,\asce}$};

    \node (L1g31) at (2*\slant, -2*\ysep) {$\gamma^{\desc}_{1,\asce}$};
    \node (L1g32) at (\xsep + 2*\slant, -2*\ysep) {$\gamma^{\desc}_{2,\asce}$};
    \node (L1d3) at (2*\xsep + 2*\slant, -2*\ysep) {$\cdots$};
    \node (L1g3l) at (3*\xsep + 2*\slant, -2*\ysep) {$\gamma^{\desc}_{\ell,\asce}$};

    \draw[->] (L1g11) -- (L1g12);
    \draw[->] (L1g12) -- (L1d1);
    \draw[->] (L1d1) -- (L1g1l);
    \draw[->] (L1g21) -- (L1g22);
    \draw[->] (L1g22) -- (L1d2);
    \draw[->] (L1d2) -- (L1g2l);
    \draw[->] (L1g31) -- (L1g32);
    \draw[->] (L1g32) -- (L1d3);
    \draw[->] (L1d3) -- (L1g3l);

    \draw[->] (L1g11) -- (L1g21);
    \draw[->] (L1g12) -- (L1g22);
    \draw[->] (L1g1l) -- (L1g2l);
    \draw[->] (L1g21) -- (L1g31);
    \draw[->] (L1g22) -- (L1g32);
    \draw[->] (L1g2l) -- (L1g3l);

    \draw[->, curvyarrow, thick] (L1g1l.south west) to[out=210, in=80] 
        ($(L1d2.north)+(0,0.3)$) to[out=260, in=30] 
        (L1g31.north east);


    \node (L2g11) at (0, -\layersep) {$\gamma^{\asce}_{1}$};
    \node (L2g12) at (\xsep, -\layersep) {$\gamma^{\asce}_{2}$};
    \node (L2d1) at (2*\xsep, -\layersep) {$\cdots$};
    \node (L2g1l) at (3*\xsep, -\layersep) {$\gamma^{\asce}_{\ell}$};

    \node (L2g21) at (\slant, -\layersep - \ysep) {$\gamma_{1}$};
    \node (L2g22) at (\xsep + \slant, -\layersep - \ysep) {$\gamma_{2}$};
    \node (L2d2) at (2*\xsep + \slant, -\layersep - \ysep) {$\cdots$};
    \node (L2g2l) at (3*\xsep + \slant, -\layersep - \ysep) {$\gamma_{\ell}$};

    \node (L2g31) at (2*\slant, -\layersep - 2*\ysep) {$\gamma^{\desc}_{1}$};
    \node (L2g32) at (\xsep + 2*\slant, -\layersep - 2*\ysep) {$\gamma^{\desc}_{2}$};
    \node (L2d3) at (2*\xsep + 2*\slant, -\layersep - 2*\ysep) {$\cdots$};
    \node (L2g3l) at (3*\xsep + 2*\slant, -\layersep - 2*\ysep) {$\gamma^{\desc}_{\ell}$};

    \draw[->] (L2g11) -- (L2g12);
    \draw[->] (L2g12) -- (L2d1);
    \draw[->] (L2d1) -- (L2g1l);
    \draw[->] (L2g21) -- (L2g22);
    \draw[->] (L2g22) -- (L2d2);
    \draw[->] (L2d2) -- (L2g2l);
    \draw[->] (L2g31) -- (L2g32);
    \draw[->] (L2g32) -- (L2d3);
    \draw[->] (L2d3) -- (L2g3l);

    \draw[->] (L2g11) -- (L2g21);
    \draw[->] (L2g12) -- (L2g22);
    \draw[->] (L2g1l) -- (L2g2l);
    \draw[->] (L2g21) -- (L2g31);
    \draw[->] (L2g22) -- (L2g32);
    \draw[->] (L2g2l) -- (L2g3l);

    \draw[->, curvyarrow, thick] (L2g1l.south west) to[out=210, in=80] 
        ($(L2d2.north)+(0,0.3)$) to[out=260, in=30] 
        (L2g31.north east);


    \node (L3g11) at (0, -2*\layersep) {$\gamma^{\asce}_{1,\desc}$};
    \node (L3g12) at (\xsep, -2*\layersep) {$\gamma^{\asce}_{2,\desc}$};
    \node (L3d1) at (2*\xsep, -2*\layersep) {$\cdots$};
    \node (L3g1l) at (3*\xsep, -2*\layersep) {$\gamma^{\asce}_{\ell,\desc}$};

    \node (L3g21) at (\slant, -2*\layersep - \ysep) {$\gamma_{1,\desc}$};
    \node (L3g22) at (\xsep + \slant, -2*\layersep - \ysep) {$\gamma_{2,\desc}$};
    \node (L3d2) at (2*\xsep + \slant, -2*\layersep - \ysep) {$\cdots$};
    \node (L3g2l) at (3*\xsep + \slant, -2*\layersep - \ysep) {$\gamma_{\ell,\desc}$};

    \node (L3g31) at (2*\slant, -2*\layersep - 2*\ysep) {$\gamma^{\desc}_{1,\desc}$};
    \node (L3g32) at (\xsep + 2*\slant, -2*\layersep - 2*\ysep) {$\gamma^{\desc}_{2,\desc}$};
    \node (L3d3) at (2*\xsep + 2*\slant, -2*\layersep - 2*\ysep) {$\cdots$};
    \node (L3g3l) at (3*\xsep + 2*\slant, -2*\layersep - 2*\ysep) {$\gamma^{\desc}_{\ell,\desc}$};

    \draw[->] (L3g11) -- (L3g12);
    \draw[->] (L3g12) -- (L3d1);
    \draw[->] (L3d1) -- (L3g1l);
    \draw[->] (L3g21) -- (L3g22);
    \draw[->] (L3g22) -- (L3d2);
    \draw[->] (L3d2) -- (L3g2l);
    \draw[->] (L3g31) -- (L3g32);
    \draw[->] (L3g32) -- (L3d3);
    \draw[->] (L3d3) -- (L3g3l);

    \draw[->] (L3g11) -- (L3g21);
    \draw[->] (L3g12) -- (L3g22);
    \draw[->] (L3g1l) -- (L3g2l);
    \draw[->] (L3g21) -- (L3g31);
    \draw[->] (L3g22) -- (L3g32);
    \draw[->] (L3g2l) -- (L3g3l);

    \draw[->, curvyarrow, thick] (L3g1l.south west) to[out=210, in=80] 
        ($(L3d2.north)+(0,0.3)$) to[out=260, in=30] 
        (L3g31.north east);


    \draw[->] (L1g11) -- (L2g11);
    \draw[->] (L1g12) -- (L2g12);
    \draw[->] (L1g1l) -- (L2g1l);
    \draw[->] (L1g21) -- (L2g21);
    \draw[->] (L1g22) -- (L2g22);
    \draw[->] (L1g2l) -- (L2g2l);
    \draw[->] (L1g31) -- (L2g31);
    \draw[->] (L1g32) -- (L2g32);
    \draw[->] (L1g3l) -- (L2g3l);

    \draw[->] (L2g11) -- (L3g11);
    \draw[->] (L2g12) -- (L3g12);
    \draw[->] (L2g1l) -- (L3g1l);
    \draw[->] (L2g21) -- (L3g21);
    \draw[->] (L2g22) -- (L3g22);
    \draw[->] (L2g2l) -- (L3g2l);
    \draw[->] (L2g31) -- (L3g31);
    \draw[->] (L2g32) -- (L3g32);
    \draw[->] (L2g3l) -- (L3g3l);


    \draw[->, black!30] (L1g1l.south) to[out=260, in=100] 
        ($(L2d1.south)+(0,-0.2)$) to[out=280, in=80] 
        (L3g11.north);

    \draw[->, black!30] (L1g2l.south) to[out=260, in=100] 
        ($(L2d2.south)+(0,-0.2)$) to[out=280, in=80] 
        (L3g21.north);

    \draw[->, black!30] (L1g3l.south) to[out=260, in=100] 
        ($(L2d3.south)+(0,-0.2)$) to[out=280, in=80] 
        (L3g31.north);

    \end{tikzpicture}
    \caption{Interlacing diagram of refinement polynomials of $\gamma$-polynomials of the truncation, dual, and the original poset.}
    \label{fig:gamma-cube}
\end{figure}

\begin{lemma}\label[lemma]{lem:gamma-cube-interlacing}
    Using the same notations as in \Cref{thm:gamma-truncation-interlacing}, the polynomials in \Cref{fig:gamma-cube} form interlacing sequences along directed paths. 
\end{lemma}

\begin{proof}
    Along each directed path in the cube, polynomials interlace as they flow
    \begin{enumerate}[label=(\Alph*)]
        \item \label{item:cubical-horizontal} (horizontally) rightward, 
        \item \label{item:cubical-vertical} (vertically) downward, and 
        \item \label{item:cubical-depth} (depth) outward towards the reader. 
    \end{enumerate}
    
    For \ref{item:cubical-horizontal}, we have the following $9$ cases, $3$ cases for each horizontal square diagram: 
    \begin{enumerate}[(1)]
        \item \label{item:cubical-horizontal-top} All $3$ horizontal paths in the top square form interlacing sequences.
        \item \label{item:cubical-horizontal-mid} All $3$ horizontal paths in the middle square form interlacing sequences. 
        \item \label{item:cubical-horizontal-bot} All $3$ horizontal paths in the bottom square form interlacing sequences. 
    \end{enumerate}
    For \ref{item:cubical-horizontal} \ref{item:cubical-horizontal-top}, we have the following recursive formula
    \[
        \pmb{\bm{\gamma}}_{\asce}^{\asce} = \mathbf{A}_{\asce}^{\asce} \cdot \pmb{\bm{\gamma}}^{\asce}, \quad
        \pmb{\bm{\gamma}}_{\asce} = \mathbf{A}_{\asce} \cdot \pmb{\bm{\gamma}}^{\asce}, \quad 
        \pmb{\bm{\gamma}}_{\asce}^{\desc} = \mathbf{A}_{\asce}^{\desc} \cdot \pmb{\bm{\gamma}}^{\desc}, \\
    \]
    where all the transition matrices are equal
    \[
        \mathbf{A}_{\asce}^{\asce} = \mathbf{A}_{\asce} = \mathbf{A}_{\asce}^{\desc}  = \begin{pmatrix} a_{ij} \end{pmatrix}, \text{ with } a_{ij} = 
        \begin{cases}
            0 & \text{ if } j \leqslant \des(i), \\
            \omega(i) & \text{ if } j > \des(i). 
        \end{cases} \\
    \]
    For \ref{item:cubical-horizontal} \ref{item:cubical-horizontal-mid}, the argument is analogous to the Proof of \Cref{lem:induction} \ref{item:main-mid-row}. \\
    For \ref{item:cubical-horizontal} \ref{item:cubical-horizontal-bot}, we have the following recursive formula 
    \[
    \pmb{\bm{\gamma}}_{\desc}^{\asce} = \mathbf{A}_{\desc}^{\asce} \cdot \pmb{\bm{\gamma}}_{\asce}^{\asce}, \quad \pmb{\bm{\gamma}}_{\desc} = \mathbf{A}_{\desc} \cdot \mathbf{\gamma}_{\asce}, \quad \pmb{\bm{\gamma}}_{\desc}^{\desc} = \mathbf{A}_{\desc}^{\desc} \cdot \pmb{\bm{\gamma}}_{\asce}^{\desc}, 
    \]
    where all the transition matrices are equal 
    \[
    \mathbf{A}_{\desc}^{\asce} = \mathbf{A}_{\desc} = \mathbf{A}_{\desc}^{\desc} = \begin{pmatrix} a_{ij} \end{pmatrix}, \text{ with } a_{ij} = 
        \begin{cases}
            \omega(i) y & \text{ if } j \leqslant \des(i), \\
            0 & \text{ if } j > \des(i). 
        \end{cases} \\
    \]

    For \ref{item:cubical-vertical} and \ref{item:cubical-depth}, we have the following cases, depending on which row we consider.  
    \begin{enumerate}[(1)]
        \item \label{item:cubical-back-top} The back row on the top square forms interlacing sequences with all the other $8$ rows in different ways:  
            \begin{enumerate}
                \item For all $1 \leqslant i \leqslant j \leqslant \ell$, $\gamma_{i, \asce}^{\asce}$ interlaces with $\gamma_{j, b}^{t}$ if neither $b$ nor $t$ is $\desc$. 
                \item For all $1 \leqslant i, j \leqslant \ell$, $\gamma_{i, \asce}^{\asce}$ interlaces with $\gamma_{j, b}^{t}$ if either $t$ or $b$ is $\desc$. 
            \end{enumerate}
        \item \label{item:cubical-mid-top} The middle row on the top square forms interlacing sequences with $5$ other rows in different ways: 
            \begin{enumerate}
                \item For all $1 \leqslant i \leqslant j \leqslant \ell$, $\gamma_{i, \asce}$ interlaces with $\gamma_{j, b}^{t}$ if $t$ is not $\asce$ and $b$ is not $\desc$. 
                \item For all $1 \leqslant i, j \leqslant \ell$, $\gamma_{i, \desc}$ interlaces with $\gamma_{j, b}^{t}$ if $t$ is not $\asce$ and $b$ is $\desc$. 
            \end{enumerate}
        \item \label{item:cubical-front-top} The front row on the top square forms interlacing sequences with $2$ other rows in different ways: 
            \begin{enumerate}
                \item For all $1 \leqslant i \leqslant j \leqslant \ell$, $\gamma_{i, \asce}^{\desc}$ interlaces with $\gamma_{j, b}^{t}$ if $t$ is $\desc$ and $b$ is not $\desc$. 
                \item For all $1 \leqslant i, j \leqslant \ell$, $\gamma_{i, \asce}^{\desc}$ interlaces with $\gamma_{j, b}^{t}$ if $t$ and $b$ are both $\desc$. 
            \end{enumerate}

        \item \label{item:cubical-back-mid} The back row on the middle square forms interlacing sequences with $5$ other more downward or outward rows in different ways: 
            \begin{enumerate}
                \item For all $1 \leqslant i \leqslant j \leqslant \ell$, $\gamma_{i}^{\asce}$ interlaces with $\gamma_{j}$, $\gamma_{j, \desc}$, and $\gamma_{j, \desc}^{\asce}$. 
                \item For all $1 \leqslant i, j \leqslant \ell$, $\gamma_{i}^{\asce}$ interlaces with $\gamma_{j}^{\desc}$ and $\gamma_{j, \desc}^{\desc}$.
            \end{enumerate}
        \item \label{item:cubical-mid-mid} The middle row on the middle square forms interlacing sequences with $3$ other more downward or outward rows in the same way: 
            For all $1 \leqslant i \leqslant j \leqslant \ell$, $\gamma_{i}$ interlaces with $\gamma_{j}^{\desc}$, $\gamma_{j, \desc}^{\desc}$ and $\gamma_{j, \desc}$. 
        \item \label{item:cubical-front-mid} The front row on the middle square forms interlacing sequences with $1$ more downward row: For all $1 \leqslant i \leqslant j \leqslant \ell$, $\gamma_{i}^{\desc}$ interlaces with $\gamma_{j, \desc}^{\desc}$. 
        \item \label{item:cubical-back-bot} The back row on the bottom square forms interlacing sequences with $2$ more outward rows in different ways: 
            \begin{enumerate}
                \item For all $1 \leqslant i \leqslant j \leqslant \ell$, $\gamma_{i, \desc}^{\asce}$ interlaces with $\gamma_{j, \desc}$. 
                \item For all $1 \leqslant i, j \leqslant \ell$, $\gamma_{i, \desc}^{\asce}$ interlaces with $\gamma_{j, \desc}^{\desc}$. 
            \end{enumerate}
        \item \label{item:cubical-mid-bot} The middle row in the bottom square forms interlacing sequences with $1$ more outward row: 
        For all $1 \leqslant i \leqslant j \leqslant \ell$, $\gamma_{i, \desc}$ interlaces with $\gamma_{j, \desc}^{\desc}$. 
    \end{enumerate}
    In the above cases, if the qualifier is $1 \leqslant i \leqslant j \leqslant \ell$, arguments analogous to Proof of \Cref{lem:induction} \ref{item:top-to-mid} and \ref{item:mid-to-bot} apply; if the qualifier is $1 \leqslant i, j \leqslant \ell$, arguments analogous to Proof of \Cref{lem:induction} \ref{item:top-to-mid} together with convexity of interlacing in \Cref{prop:interlacing-properties} apply. The cubical diagram in \Cref{fig:gamma-cube} forms interlacing sequences by induction on $n$, together with repeated applications of interlacing preserving of the transition matrices. 
\end{proof}

\section{Open questions}
\label{sec:open-questions}
In this section, we formulate some additional open questions that we regard as relevant for further research. 
\subsection{Geometric lattices}
Recall that in Section \ref{sec:rank-uniform-labeled-posets} we proved that if an EL-labeled poset $(\P, \lambda)$ is rank-uniform then $\P$ is rank-uniform (Lemma~\ref{lem:el-rank-upper-implies-rank-upper}). We feel it is important to ask ourselves to what extent the converse is true. More explicitly, if we are given a finite rank-uniform bounded poset $\P$, what conditions on $\P$ would ensure that $\P$ necessarily admits an EL-labeling $\lambda$ such that $(\P, \lambda)$ is rank-uniform, or even monotonically rank-uniform? For instance Theorem~\ref{thm:main4-intro} shows that supersolvability is one such condition. Obviously a necessary condition is that $\P$ is EL-shellable. This leads us to the following question. 
\begin{question}\label{question:rank-upper-labeled}
Let $\P$ be an EL-shellable poset. If $\P$ is rank-uniform, does $\P$ admit an EL-labeling $\lambda$ such that $(\P, \lambda)$ is a (monotonic) EL-labeled poset?  
\end{question}

Another property we feel may be interesting to look at is geometricity. A geometric lattice is the lattice of flats of a matroid (see \cite{oxley2006matroid} for a reference on that topic). Geometric lattices are relevant to our problem because they are known to have multiple EL-labelings, as the following classical result shows. 
\begin{proposition}[\cite{bjorner-el-shellable}]
Let $\L$ be a geometric lattice and let $\vartriangleleft$ be a total order on the set $\At(\L)$ of atoms of $\L$. The edge labeling  $\lambda_{\vartriangleleft} : \Edge(\L) \rightarrow (\At(\L), \vartriangleleft)$ defined by
\begin{equation*}
    \lambda_{\vartriangleleft}(s\lessdot t) = \min \{ u \in \At(\L) \, | \, s\vee u = t \}
\end{equation*}
is an EL-labeling. 
\end{proposition}

We also propose the following variant of Question~\ref{question:rank-upper-labeled}. 

\begin{question}
Let $\L$ be a geometric lattice. If $\L$ is rank-uniform, does there exist a total order $\vartriangleleft$ on $\At(\L)$ such that $(\L, \lambda_{\vartriangleleft})$ is a monotonic rank-uniform labeled poset? 
\end{question}

This question seems already interesting for geometric lattices of rank $3$. If $\L$ is a rank-uniform geometric lattice of rank $3$, then for any total order $\vartriangleleft$ on $\At(\L)$, the labeled poset $(\L, \lambda_{\vartriangleleft})$ is rank-uniform. Indeed, the uniformity at rank $2$ is always satisfied. For any $s\in \L$ of rank $1$ and for all $i \geqslant1$ we have $\omega_s(i) = 1$ if $i \leqslant W_1(s)$ and $\omega_s(i) = 0$ otherwise, which proves uniformity of width. For all $t\gtrdot s \in \L$ we have $\des(s,t) = 0$ if $\Ind(s,t) = 1$ and $\des(s,t) = 1$ otherwise, which proves uniformity of descent. Only the question of monotonicity at rank 0 remains, which we reformulate below. 

\begin{question}
Let $\L$ be a geometric lattice of rank $3$. If $\L$ is rank-uniform, does there exist a total order $\vartriangleleft$ on $\At(\L)$ such that the quantity 
\begin{equation*}
    \des_0(s) = |\{ t \gtrdot s \, |\, \exists s' \vartriangleleft s \textrm{ s.t. } s' \leqslant t \}|
\end{equation*}
defined for all atoms $s$, is weakly increasing with respect to $\vartriangleleft$?  
\end{question}

On the topic of interlacing, recall that Theorem~\ref{thm:gamma-interlacing-aug-contract} and the paragraph following it show that if $\L$ is a UMEL-shellable geometric lattice, then for every atom $a \in \L$ we have the interlacing relations 
\[\uH_{[a, \un]} \preceq \uH_{\L}, \quad \H_{[a, \un]} \preceq \H_{\L}, \quad \H_{[a, \un]} \preceq \uH_{\L}.\]
This motivates the following strengthenings of Conjecture~\ref{conj:real-rootedness-chow}.

\begin{conjecture}
\label[conjecture]{conj:contraction-interlace}
    Let $\L$ be a geometric lattice, and let $a$ be an atom of $\L$. The (augmented) Chow polynomial of the upper-interval $[a, \un]$ interlaces the (augmented) Chow polynomial of $\L$: 
    \[
    \uH_{[a, \un]} \preceq \uH_{\L}, \quad 
    \H_{[a, \un]} \preceq \H_{\L}, \quad \H_{[a, \un]} \preceq \H_{\L}.\] 
\end{conjecture}
In matroid terminology, the above conjecture is asserting that for any (simple) matroid $\M$ on the ground set $E$, the contraction by a single element $i\in E$ satisfies $\uH_{\M/i}\preceq \uH_{\M}$, $\H_{\M/i} \preceq \H_{\M}$, and $\H_{\M/i} \preceq \uH_{\M}$.

\begin{conjecture}
\label[conjecture]{conj:truncation-interlace}
    Let $\L$ be a geometric lattice, and let $a$ be an atom of $\L$. The (augmented) Chow polynomial of the principle truncation of $\calL$ interlaces the (augmented) Chow polynomial of $\L$: 
    \[
    \uH_{\Tr(\calL)} \preceq \uH_{\L}, \quad \text{ and } \quad 
    \H_{\Tr(\calL)} \preceq \H_{\L}.\] 
\end{conjecture}

\subsection{TN-posets}

Br\"and\'en and Saud-Maia-Leite \cite{branden-saud2} introduced the concept of a totally nonnegative poset (TN-poset for short). By definition, a TN-poset is a (lower) rank-uniform poset whose lower triangular matrix of (lower) Whitney numbers of the second kind $(W_{n-i}(j))_{i,j}$ is totally nonnegative (meaning that all its minors are nonnegative). 

In \cite{branden-saud1} and \cite{branden-vecchi2} the authors show that the chain polynomial and the Chow polynomial respectively of a TN-poset are real-rooted. Due to the similarity of the examples considered in those articles and the present work (Dowling lattices, projective geometries, uniform matroids, etc.) it is natural to inquire to what extent the notions of TN-posets and UMEL-shellable posets relate to each other. On that topic we propose the following conjecture. 

\begin{conjecture}\label{conj:tn-geom-lattices-are-umel}
    Let $\M$ be a matroid such that $\mathcal{L}(\M)$ is a TN-poset. Then $\mathcal{L}(\M)$ is UMEL-shellable.
\end{conjecture}

The converse of the above conjecture is not true, because partition lattices are UMEL-shellable and they fail to be TN-posets. The conjecture is supported by the fact that TN-geometric lattices are in particular lower rank-uniform, and it is known that lower rank-uniform geometric lattices are upper rank-uniform (see \cite[Theorem 4.5]{brylawski}). 

A dual version of Conjecture~\ref{conj:tn-geom-lattices-are-umel}, which we propose as a question, would be the following.

\begin{question}
    Let $\M$ be a matroid such that $\mathcal{L}(\M)$ is a UMEL-shellable poset. Is $\mathcal{L}(\M)^*$ a TN-poset?
\end{question}

The above question is strongly related to a conjecture of Br\"and\'en and Saud-Maia-Leite \cite[Conjecture~4.3]{branden-saud1}\footnote{In the current arXiv posting of \cite{branden-saud1}, the conjecture contains a misprint. Br\"and\'en (personal communication) informed us that they conjecture that $\mathcal{L}(\M)^*$ is a TN-poset whenever $\mathcal{L}(\M)$ is rank-uniform.}. 
\smallskip

We can ask analogous questions in the broader generality of graded bounded posets. However, since UMEL-shellable posets are EL-shellable, we need to restrict our attention to EL-shellable TN-posets (there exist TN-posets that are not Cohen--Macaulay). We ask the following questions. 

\begin{question}
    If $\P$ is a UMEL-shellable poset, is $\Pdual$ a TN-poset? Conversely if $\P$ is an EL-shellable TN-poset, is $\Pdual$ UMEL-shellable?  
\end{question}

An affirmative answer to the first of the above two questions would be very interesting, as this would provide a combinatorial criterion for proving that a poset is TN.

\bibliographystyle{amsalpha}
\bibliography{biblio}

@ARTICLE{eur2022stellahedral,
    AUTHOR = {Eur, Christopher and Huh, June and Larson, Matt},
     TITLE = {Stellahedral geometry of matroids},
   JOURNAL = {Forum Math. Pi},
  FJOURNAL = {Forum of Mathematics. Pi},
    VOLUME = {11},
      YEAR = {2023},
     PAGES = {Paper No. e24, 48},
      ISSN = {2050-5086},
   MRCLASS = {14M25 (05B35 14N20 52B40 52C35)},
  MRNUMBER = {4653766},
MRREVIEWER = {Paolo\ Aluffi},
       DOI = {10.1017/fmp.2023.24},
       URL = {https://doi.org/10.1017/fmp.2023.24},
}

@book {stanley-ec1,
    AUTHOR = {Stanley, Richard P.},
     TITLE = {Enumerative combinatorics. {V}olume 1},
    SERIES = {Cambridge Studies in Advanced Mathematics},
    VOLUME = {49},
   EDITION = {Second},
 PUBLISHER = {Cambridge University Press, Cambridge},
      YEAR = {2012},
     PAGES = {xiv+626},
      ISBN = {978-1-107-60262-5},
   MRCLASS = {05-02 (05A15 06-02)},
  MRNUMBER = {2868112},
}

@article{de1995wonderful,
    AUTHOR = {De Concini, C. and Procesi, C.},
     TITLE = {Wonderful models of subspace arrangements},
   JOURNAL = {Selecta Math. (N.S.)},
  FJOURNAL = {Selecta Mathematica. New Series},
    VOLUME = {1},
      YEAR = {1995},
    NUMBER = {3},
     PAGES = {459--494},
      ISSN = {1022-1824,1420-9020},
   MRCLASS = {14D99 (32G13 52B30)},
  MRNUMBER = {1366622},
MRREVIEWER = {V.\ Leksin},
       DOI = {10.1007/BF01589496},
       URL = {https://doi.org/10.1007/BF01589496},
}

@article{ardila2016closure,
    AUTHOR = {Ardila, Federico and Boocher, Adam},
     TITLE = {The closure of a linear space in a product of lines},
   JOURNAL = {J. Algebraic Combin.},
  FJOURNAL = {Journal of Algebraic Combinatorics. An International Journal},
    VOLUME = {43},
      YEAR = {2016},
    NUMBER = {1},
     PAGES = {199--235},
      ISSN = {0925-9899,1572-9192},
   MRCLASS = {05B35 (13D02 13F20 13F55 13P10 52B40)},
  MRNUMBER = {3439307},
MRREVIEWER = {Benjamin\ P.\ Richert},
       DOI = {10.1007/s10801-015-0634-x},
       URL = {https://doi.org/10.1007/s10801-015-0634-x},
}

@article{wagner1992total,
    AUTHOR = {Wagner, David G.},
     TITLE = {Total positivity of {H}adamard products},
   JOURNAL = {J. Math. Anal. Appl.},
  FJOURNAL = {Journal of Mathematical Analysis and Applications},
    VOLUME = {163},
      YEAR = {1992},
    NUMBER = {2},
     PAGES = {459--483},
      ISSN = {0022-247X,1096-0813},
   MRCLASS = {15A48 (15A57)},
  MRNUMBER = {1145841},
MRREVIEWER = {M.\ Tismenetsky},
       DOI = {10.1016/0022-247X(92)90261-B},
       URL = {https://doi.org/10.1016/0022-247X(92)90261-B},
}

@book{obreshkov1963verteilung,
    AUTHOR = {Obreschkoff, Nikola},
     TITLE = {Verteilung und {B}erechnung der {N}ullstellen reeller
              {P}olynome},
 PUBLISHER = {VEB Deutscher Verlag der Wissenschaften, Berlin},
      YEAR = {1963},
     PAGES = {viii+298},
   MRCLASS = {30.11},
  MRNUMBER = {164003},
MRREVIEWER = {M.\ Marden},
}

@article {bjorner-el-shellable,
    AUTHOR = {Bj{\"o}rner, Anders},
     TITLE = {Shellable and {C}ohen-{M}acaulay partially ordered sets},
   JOURNAL = {Trans. Amer. Math. Soc.},
  FJOURNAL = {Transactions of the American Mathematical Society},
    VOLUME = {260},
      YEAR = {1980},
    NUMBER = {1},
     PAGES = {159--183},
      ISSN = {0002-9947,1088-6850},
   MRCLASS = {06A10 (13H10 52A25)},
  MRNUMBER = {570784},
MRREVIEWER = {P.\ McMullen},
       DOI = {10.2307/1999881},
       URL = {https://doi.org/10.2307/1999881},
}

@article {reiner-welker,
    AUTHOR = {Reiner, Victor and Welker, Volkmar},
     TITLE = {On the {C}harney-{D}avis and {N}eggers-{S}tanley conjectures},
   JOURNAL = {J. Combin. Theory Ser. A},
  FJOURNAL = {Journal of Combinatorial Theory. Series A},
    VOLUME = {109},
      YEAR = {2005},
    NUMBER = {2},
     PAGES = {247--280},
      ISSN = {0097-3165,1096-0899},
   MRCLASS = {06A07 (13F55 16S37)},
  MRNUMBER = {2121026},
MRREVIEWER = {Hugh\ Ross\ Thomas},
       DOI = {10.1016/j.jcta.2004.09.003},
       URL = {https://doi.org/10.1016/j.jcta.2004.09.003},
}

@book{OT_1992,
    AUTHOR = {Orlik, Peter and Terao, Hiroaki},
     TITLE = {Arrangements of hyperplanes},
    SERIES = {Grundlehren der mathematischen Wissenschaften [Fundamental
              Principles of Mathematical Sciences]},
    VOLUME = {300},
 PUBLISHER = {Springer-Verlag, Berlin},
      YEAR = {1992},
     PAGES = {xviii+325},
      ISBN = {3-540-55259-6},
   MRCLASS = {52B30 (14F35 20F36 20F55 32S25 57N65)},
  MRNUMBER = {1217488},
MRREVIEWER = {Michel\ Yves\ Jambu},
       DOI = {10.1007/978-3-662-02772-1},
       URL = {https://doi.org/10.1007/978-3-662-02772-1},
}

@book{Li2020LexicographicShellablePosets,
    AUTHOR = {Li, Tiansi},
     TITLE = {A {S}tudy on {L}exicographic {S}hellable {P}osets},
      NOTE = {Thesis (Ph.D.)--Washington University in St. Louis},
 PUBLISHER = {ProQuest LLC, Ann Arbor, MI},
      YEAR = {2020},
     PAGES = {43},
      ISBN = {979-8607-32731-6},
   MRCLASS = {99-05},
  MRNUMBER = {4094313},
       URL =
              {http://gateway.proquest.com/openurl?url_ver=Z39.88-2004&rft_val_fmt=info:ofi/fmt:kev:mtx:dissertation&res_dat=xri:pqm&rft_dat=xri:pqdiss:27836208},
}

@article{FoldesWoodroofe2022,
  author  = {Stephan Foldes and Russ Woodroofe},
  title   = {A modular characterization of supersolvable lattices},
  journal = {Proceedings of the American Mathematical Society},
  volume  = {150},
  number  = {1},
  pages   = {31--39},
  year    = {2022},
  doi     = {10.1090/proc/15645}
}

@book {hersh-thesis,
    AUTHOR = {Hersh, Patricia},
     TITLE = {Decomposition and enumeration in partially ordered sets},
      NOTE = {Thesis (Ph.D.)--Massachusetts Institute of Technology},
 PUBLISHER = {ProQuest LLC, Ann Arbor, MI},
      YEAR = {1999},
     PAGES = {(no paging)},
   MRCLASS = {99-05},
  MRNUMBER = {2716880},
       URL =
              {http://gateway.proquest.com/openurl?url_ver=Z39.88-2004&rft_val_fmt=info:ofi/fmt:kev:mtx:dissertation&res_dat=xri:pqdiss&rft_dat=xri:pqdiss:0801031},
}

@article{Paris_2000,
    AUTHOR = {Paris, Luis},
     TITLE = {Intersection subgroups of complex hyperplane arrangements},
   JOURNAL = {Topology Appl.},
  FJOURNAL = {Topology and its Applications},
    VOLUME = {105},
      YEAR = {2000},
    NUMBER = {3},
     PAGES = {319--343},
      ISSN = {0166-8641,1879-3207},
   MRCLASS = {52C35},
  MRNUMBER = {1769026},
       DOI = {10.1016/S0166-8641(99)00068-1},
       URL = {https://doi.org/10.1016/S0166-8641(99)00068-1},
}

@book{birkhoff1940lattice,
    AUTHOR = {Birkhoff, Garrett},
     TITLE = {Lattice theory},
    SERIES = {American Mathematical Society Colloquium Publications},
    VOLUME = {Vol. XXV},
   EDITION = {Third},
 PUBLISHER = {American Mathematical Society, Providence, RI},
      YEAR = {1967},
     PAGES = {vi+418},
   MRCLASS = {06.30},
  MRNUMBER = {227053},
MRREVIEWER = {P.\ A.\ Fillmore},
}

@article {athanasiadis-kalampogia,
    AUTHOR = {Athanasiadis, Christos A. and Kalampogia-Evangelinou,
              Katerina},
     TITLE = {Chain enumeration, partition lattices and polynomials with
              only real roots},
   JOURNAL = {Comb. Theory},
  FJOURNAL = {Combinatorial Theory},
    VOLUME = {3},
      YEAR = {2023},
    NUMBER = {1},
     PAGES = {Paper No. 12, 21},
      ISSN = {2766-1334},
   MRCLASS = {05A15 (05A05 05A18 05E45 06A07 26C10)},
  MRNUMBER = {4565299},
MRREVIEWER = {Nikolai\ Volodin},
       DOI = {10.5070/c63160425},
       URL = {https://doi.org/10.5070/c63160425},
}

@incollection {ardila-icm22,
    AUTHOR = {{Ardila-Mantilla}, Federico},
     TITLE = {The geometry of geometries: matroid theory, old and new},
 BOOKTITLE = {I{CM}---{I}nternational {C}ongress of {M}athematicians. {V}ol.
              6. {S}ections 12--14},
     PAGES = {4510--4541},
 PUBLISHER = {EMS Press, Berlin},
      YEAR = {[2023] \copyright 2023},
      ISBN = {978-3-98547-064-8; 978-3-98547-564-3; 978-3-98547-058-7},
   MRCLASS = {52B40 (05B35 14C17 14T15 16T30 52B45)},
  MRNUMBER = {4680412},
}

@article{Stanley1971ModularElements,
  author  = {Stanley, Richard P.},
  title   = {Modular Elements of Geometric Lattices},
  journal = {Algebra Universalis},
  volume  = {1},
  number  = {2},
  pages   = {214--217},
  year    = {1971},
  doi     = {10.1007/BF02944981}
}

@article{Damiani1994,
    AUTHOR = {Damiani, E. and D'Antona, O. and Regonati, F.},
     TITLE = {Whitney numbers of some geometric lattices},
   JOURNAL = {J. Combin. Theory Ser. A},
  FJOURNAL = {Journal of Combinatorial Theory. Series A},
    VOLUME = {65},
      YEAR = {1994},
    NUMBER = {1},
     PAGES = {11--25},
      ISSN = {0097-3165,1096-0899},
   MRCLASS = {06C10},
  MRNUMBER = {1255260},
MRREVIEWER = {Jaroslav\ Libicher},
       DOI = {10.1016/0097-3165(94)90034-5},
       URL = {https://doi.org/10.1016/0097-3165(94)90034-5},
}

@misc{stevens-bachelor,
    AUTHOR = {Stevens, Matthew},
     TITLE = {Real-rootedness conjectures in matroid theory},
      NOTE = {Thesis (Bachelor)--Stanford University, \url{https://mcs0042.github.io/bachelor.pdf}},
      YEAR = {2021},
     PAGES = {27},
}

@incollection {huh-icm22,
    AUTHOR = {Huh, June},
     TITLE = {Combinatorics and {H}odge theory},
 BOOKTITLE = {I{CM}---{I}nternational {C}ongress of {M}athematicians. {V}ol.
              1. {P}rize lectures},
     PAGES = {212--239},
 PUBLISHER = {EMS Press, Berlin},
      YEAR = {[2023] \copyright 2023},
      ISBN = {978-3-98547-059-4; 978-3-98547-559-9; 978-3-98547-058-7},
   MRCLASS = {05A20 (05B35 14C17 14C30 14N10)},
  MRNUMBER = {4680249},
}

@article {eur,
    AUTHOR = {Eur, Christopher},
     TITLE = {Essence of independence: {H}odge theory of matroids since
              {J}une {H}uh},
   JOURNAL = {Bull. Amer. Math. Soc. (N.S.)},
  FJOURNAL = {American Mathematical Society. Bulletin. New Series},
    VOLUME = {61},
      YEAR = {2024},
    NUMBER = {1},
     PAGES = {73--102},
      ISSN = {0273-0979,1088-9485},
   MRCLASS = {05B35 (05E14 14C17 14F43)},
  MRNUMBER = {4678572},
MRREVIEWER = {Maruti\ M.\ Shikare},
       DOI = {10.1090/bull/1803},
       URL = {https://doi.org/10.1090/bull/1803},
}

@article{branden2021symmetric,
  title={Symmetric decompositions and real-rootedness},
  author={Br{\"a}nd{\'e}n, Petter and Solus, Liam},
  journal={International Mathematics Research Notices},
  volume={2021},
  number={10},
  pages={7764--7798},
  year={2021},
  publisher={Oxford University Press}
}

@article{savage2015,
  title={The {$\mathbf{s}$}-Eulerian polynomials have only real roots},
  author={Savage, Carla and Visontai, Mirk{\'o}},
  journal={Transactions of the American Mathematical Society},
  volume={367},
  number={2},
  pages={1441--1466},
  year={2015}
}

@article{braden2020singular,
  title={Singular {H}odge theory for combinatorial geometries},
  author={Braden, Tom and Huh, June and Matherne, Jacob P and Proudfoot, Nicholas and Wang, Botong},
  journal ={arXiv e-prints},
  eprint={2010.06088},
  url={https://arxiv.org/abs/2010.06088},
  year={2020},
  archivePrefix={arXiv},
  pages={arXiv:2010.06088}
}

@book{oxley2006matroid,
    AUTHOR = {Oxley, James G.},
     TITLE = {Matroid theory},
    SERIES = {Oxford Science Publications},
 PUBLISHER = {The Clarendon Press, Oxford University Press, New York},
      YEAR = {1992},
     PAGES = {xii+532},
      ISBN = {0-19-853563-5},
   MRCLASS = {05B35 (90C27)},
  MRNUMBER = {1207587},
MRREVIEWER = {Talmage\ J.\ Reid},
}

@ARTICLE{liao,
       author = {{Liao}, Hsin-Chieh},
        title = "{Equivariant $\gamma$-positivity of Chow rings and augmented Chow rings of matroids}",
      journal = {arXiv e-prints},
         year = 2024,
        month = aug,
          eid = {arXiv:2408.00745},
        pages = {arXiv:2408.00745},
          doi = {10.48550/arXiv.2408.00745},
archivePrefix = {arXiv},
       eprint = {2408.00745},
 primaryClass = {math.CO},
       adsurl = {https://ui.adsabs.harvard.edu/abs/2024arXiv240800745L}
}

@article {gaiffi-serventi,
    AUTHOR = {Gaiffi, Giovanni and Serventi, Matteo},
     TITLE = {Poincar\'e{} series for maximal {D}e {C}oncini-{P}rocesi
              models of root arrangements},
   JOURNAL = {Atti Accad. Naz. Lincei Rend. Lincei Mat. Appl.},
  FJOURNAL = {Atti della Accademia Nazionale dei Lincei. Rendiconti Lincei.
              Matematica e Applicazioni},
    VOLUME = {23},
      YEAR = {2012},
    NUMBER = {1},
     PAGES = {51--67},
      ISSN = {1120-6330,1720-0768},
   MRCLASS = {14N20 (17B22)},
  MRNUMBER = {2924891},
MRREVIEWER = {Satyan\ L.\ Devadoss},
       DOI = {10.4171/RLM/616},
       URL = {https://doi.org/10.4171/RLM/616},
}

@article {stump,
    AUTHOR = {Stump, Christian},
     TITLE = {Chow and augmented {C}how polynomials as evaluations of
              {P}oincar\'e-extended $\mathbf{ab}$-indices},
   JOURNAL = {Adv. Math.},
  FJOURNAL = {Advances in Mathematics},
    VOLUME = {482},
      YEAR = {2025},
     PAGES = {Paper No. 110618},
      ISSN = {0001-8708,1090-2082},
   MRCLASS = {06A07 (05B35 52C40)},
  MRNUMBER = {4975905},
       DOI = {10.1016/j.aim.2025.110618},
       URL = {https://doi.org/10.1016/j.aim.2025.110618},
}

@article{bjorner1983lexicographically,
    AUTHOR = {Bj\"orner, Anders and Wachs, Michelle},
     TITLE = {On lexicographically shellable posets},
   JOURNAL = {Trans. Amer. Math. Soc.},
  FJOURNAL = {Transactions of the American Mathematical Society},
    VOLUME = {277},
      YEAR = {1983},
    NUMBER = {1},
     PAGES = {323--341},
      ISSN = {0002-9947,1088-6850},
   MRCLASS = {06A10 (05A99 52A25 57Q05)},
  MRNUMBER = {690055},
MRREVIEWER = {R.\ P.\ Dilworth},
       DOI = {10.2307/1999359},
       URL = {https://doi.org/10.2307/1999359},
}

@ARTICLE{branden-saud2,
       author = {{Br{\"a}nd{\'e}n}, Petter and {Leite}, Leonardo Saud Maia},
        title = "{Totally nonnegative matrices, chain enumeration and zeros of polynomials}",
      journal = {arXiv e-prints},
         year = 2024,
        month = dec,
          eid = {arXiv:2412.06595},
        pages = {arXiv:2412.06595},
          doi = {10.48550/arXiv.2412.06595},
archivePrefix = {arXiv},
       eprint = {2412.06595},
 primaryClass = {math.CO},
       adsurl = {https://ui.adsabs.harvard.edu/abs/2024arXiv241206595B}
}

@ARTICLE{branden-saud1,
       author = {{Br{\"a}nd{\'e}n}, Petter and {Leite}, Leonardo Saud Maia},
        title = "{On chain polynomials of geometric lattices}",
      journal = {arXiv e-prints},
         year = 2025,
        month = aug,
          eid = {arXiv:2508.13810},
        pages = {arXiv:2508.13810},
          doi = {10.48550/arXiv.2508.13810},
archivePrefix = {arXiv},
       eprint = {2508.13810},
 primaryClass = {math.CO},
       adsurl = {https://ui.adsabs.harvard.edu/abs/2025arXiv250813810B}
}

@article {brylawski,
    AUTHOR = {Brylawski, Tom},
     TITLE = {Intersection theory for embeddings of matroids into uniform
              geometries},
   JOURNAL = {Stud. Appl. Math.},
  FJOURNAL = {Studies in Applied Mathematics},
    VOLUME = {61},
      YEAR = {1979},
    NUMBER = {3},
     PAGES = {211--244},
      ISSN = {0022-2526},
   MRCLASS = {05B25 (05B35)},
  MRNUMBER = {550398},
MRREVIEWER = {A.\ Barlotti},
       DOI = {10.1002/sapm1979613211},
       URL = {https://doi.org/10.1002/sapm1979613211},
}

@article {athanasiadis-douvropoulos-kalampogia,
    AUTHOR = {Athanasiadis, Christos A. and Douvropoulos, Theo and
              Kalampogia-Evangelinou, Katerina},
     TITLE = {Two classes of posets with real-rooted chain polynomials},
   JOURNAL = {Electron. J. Combin.},
  FJOURNAL = {Electronic Journal of Combinatorics},
    VOLUME = {31},
      YEAR = {2024},
    NUMBER = {4},
     PAGES = {Paper No. 4.16, 22},
      ISSN = {1077-8926},
   MRCLASS = {06A07 (05A15 05E45 26C10)},
  MRNUMBER = {4816497},
MRREVIEWER = {Galen\ Dorpalen-Barry},
       DOI = {10.37236/12218},
       URL = {https://doi.org/10.37236/12218},
}

@ARTICLE{athanasiadis-ferroni,
       author = {{Athanasiadis}, Christos A. and {Ferroni}, Luis},
        title = "{A convex ear decomposition of the augmented Bergman complex of a matroid}",
      fjournal = {Arkiv f\"or Matematik},
      journal = {Ark. Math.},
         year = 2025,
        pages = {to appear}
}

@article{stanley1972supersolvable,
    AUTHOR = {Stanley, R. P.},
     TITLE = {Supersolvable lattices},
   JOURNAL = {Algebra Universalis},
  FJOURNAL = {Algebra Universalis},
    VOLUME = {2},
      YEAR = {1972},
     PAGES = {197--217},
      ISSN = {0002-5240,1420-8911},
   MRCLASS = {06A20},
  MRNUMBER = {309815},
MRREVIEWER = {O.\ Frink},
       DOI = {10.1007/BF02945028},
       URL = {https://doi.org/10.1007/BF02945028},
}

@ARTICLE{branden-vecchi2,
       author = {{Br{\"a}nd{\'e}n}, Petter and {Vecchi}, Lorenzo},
        title = "{Chow polynomials of\phantom{z}totally nonnegative matrices and posets}",
      journal = {arXiv e-prints},
         year = 2025,
        month = sep,
          eid = {arXiv:2509.17852},
        pages = {arXiv:2509.17852},
          doi = {10.48550/arXiv.2509.17852},
archivePrefix = {arXiv},
       eprint = {2509.17852},
 primaryClass = {math.CO},
       adsurl = {https://ui.adsabs.harvard.edu/abs/2025arXiv250917852B}
}

@ARTICLE{ferroni-matherne-vecchi,
       author = {{Ferroni}, Luis and {Matherne}, Jacob P. and {Vecchi}, Lorenzo},
        title = "{Chow functions for partially ordered sets}",
      journal = {arXiv e-prints},
         year = 2024,
        month = nov,
          eid = {arXiv:2411.04070},
        pages = {arXiv:2411.04070},
          doi = {10.48550/arXiv.2411.04070},
archivePrefix = {arXiv},
       eprint = {2411.04070},
 primaryClass = {math.CO},
       adsurl = {https://ui.adsabs.harvard.edu/abs/2024arXiv241104070F}
}

@article{BRADEN2022108646,
    AUTHOR = {Braden, Tom and Huh, June and Matherne, Jacob P. and
              Proudfoot, Nicholas and Wang, Botong},
     TITLE = {A semi-small decomposition of the {C}how ring of a matroid},
   JOURNAL = {Adv. Math.},
  FJOURNAL = {Advances in Mathematics},
    VOLUME = {409},
      YEAR = {2022},
     PAGES = {Paper No. 108646, 49},
      ISSN = {0001-8708,1090-2082},
   MRCLASS = {05B35 (14C25)},
  MRNUMBER = {4477425},
MRREVIEWER = {Joseph\ Kung},
       DOI = {10.1016/j.aim.2022.108646},
       URL = {https://doi.org/10.1016/j.aim.2022.108646},
}

@article{feichtner2004chow,
    AUTHOR = {Feichtner, Eva Maria and Yuzvinsky, Sergey},
     TITLE = {Chow rings of toric varieties defined by atomic lattices},
   JOURNAL = {Invent. Math.},
  FJOURNAL = {Inventiones Mathematicae},
    VOLUME = {155},
      YEAR = {2004},
    NUMBER = {3},
     PAGES = {515--536},
      ISSN = {0020-9910,1432-1297},
   MRCLASS = {14C15 (14M25)},
  MRNUMBER = {2038195},
MRREVIEWER = {G.\ K.\ Sankaran},
       DOI = {10.1007/s00222-003-0327-2},
       URL = {https://doi.org/10.1007/s00222-003-0327-2},
}

@article{backman2023simplicial,
  title={Simplicial generation of Chow rings of matroids},
  author={Backman, Spencer and Eur, Christopher and Simpson, Connor},
  journal={Journal of the European Mathematical Society},
  volume={26},
  number={11},
  pages={4491--4535},
  year={2023}
}

@ARTICLE{hoster-stump,
       author = {{Hoster}, Elena and {Stump}, Christian},
        title = "{Chow polynomials of simplicial posets with positive $h$-vector are real-rooted}",
      journal = {arXiv e-prints},
         year = 2025,
        month = aug,
          eid = {arXiv:2508.15538},
        pages = {arXiv:2508.15538},
          doi = {10.48550/arXiv.2508.15538},
archivePrefix = {arXiv},
       eprint = {2508.15538},
 primaryClass = {math.CO},
       adsurl = {https://ui.adsabs.harvard.edu/abs/2025arXiv250815538H}
}

@article{simion1995q,
    AUTHOR = {Simion, Rodica},
     TITLE = {On {$q$}-analogues of partially ordered sets},
   JOURNAL = {J. Combin. Theory Ser. A},
  FJOURNAL = {Journal of Combinatorial Theory. Series A},
    VOLUME = {72},
      YEAR = {1995},
    NUMBER = {1},
     PAGES = {135--183},
      ISSN = {0097-3165,1096-0899},
   MRCLASS = {06A08 (05A30 05E25)},
  MRNUMBER = {1354971},
       DOI = {10.1016/0097-3165(95)90032-2},
       URL = {https://doi.org/10.1016/0097-3165(95)90032-2},
}

@incollection{dowling1973q,
  title={A q-analog of the partition lattice},
  author={Dowling, Thomas A},
  booktitle={A survey of Combinatorial Theory},
  pages={101--115},
  year={1973},
  publisher={Elsevier}
}

@article{dowling1973class,
  title={A class of geometric lattices based on finite groups},
  author={Dowling, Thomas A},
  journal={Journal of Combinatorial Theory, Series B},
  volume={14},
  number={1},
  pages={61--86},
  year={1973},
  publisher={Elsevier}
}

@incollection{wachs2007poset,
    AUTHOR = {Wachs, Michelle L.},
     TITLE = {Poset topology: tools and applications},
 BOOKTITLE = {Geometric combinatorics},
    SERIES = {IAS/Park City Math. Ser.},
    VOLUME = {13},
     PAGES = {497--615},
 PUBLISHER = {Amer. Math. Soc., Providence, RI},
      YEAR = {2007},
      ISBN = {978-0-8218-3736-8; 0-8218-3736-2},
   MRCLASS = {06B30 (05E10 52C35 55R80)},
  MRNUMBER = {2383132},
       DOI = {10.1090/pcms/013/09},
       URL = {https://doi.org/10.1090/pcms/013/09},
}

@incollection{branden2015unimodality,
    AUTHOR = {Br{\"a}nd{\'e}n, Petter},
     TITLE = {Unimodality, log-concavity, real-rootedness and beyond},
 BOOKTITLE = {Handbook of enumerative combinatorics},
    SERIES = {Discrete Math. Appl. (Boca Raton)},
     PAGES = {437--483},
 PUBLISHER = {CRC Press, Boca Raton, FL},
      YEAR = {2015},
      ISBN = {978-1-4822-2085-8},
   MRCLASS = {05Axx (05D40 05E10)},
  MRNUMBER = {3409348},
}

@article{branden2004sign,
    AUTHOR = {Br{\"a}nd{\'e}n, Petter},
     TITLE = {Sign-graded posets, unimodality of {$W$}-polynomials and the
              {C}harney-{D}avis conjecture},
   JOURNAL = {Electron. J. Combin.},
  FJOURNAL = {Electronic Journal of Combinatorics},
    VOLUME = {11},
      YEAR = {2004/06},
    NUMBER = {2},
     PAGES = {Research Paper 9, 15},
      ISSN = {1077-8926},
   MRCLASS = {06A07 (13F55)},
  MRNUMBER = {2120105},
MRREVIEWER = {Axel\ Hultman},
       DOI = {10.37236/1866},
       URL = {https://doi.org/10.37236/1866},
}

@article{gal2005real,
    AUTHOR = {Gal, \'Swiatos\l aw R.},
     TITLE = {Real root conjecture fails for five- and higher-dimensional
              spheres},
   JOURNAL = {Discrete Comput. Geom.},
  FJOURNAL = {Discrete \& Computational Geometry. An International Journal
              of Mathematics and Computer Science},
    VOLUME = {34},
      YEAR = {2005},
    NUMBER = {2},
     PAGES = {269--284},
      ISSN = {0179-5376,1432-0444},
   MRCLASS = {52B70 (05E25 06A07)},
  MRNUMBER = {2155722},
MRREVIEWER = {Wolfgang\ K\"uhnel},
       DOI = {10.1007/s00454-005-1171-5},
       URL = {https://doi.org/10.1007/s00454-005-1171-5},
}

@article {ferroni-schroter,
    AUTHOR = {Ferroni, Luis and Schr\"oter, Benjamin},
     TITLE = {Valuative invariants for large classes of matroids},
   JOURNAL = {J. Lond. Math. Soc. (2)},
  FJOURNAL = {Journal of the London Mathematical Society. Second Series},
    VOLUME = {110},
      YEAR = {2024},
    NUMBER = {3},
     PAGES = {Paper No. e12984},
      ISSN = {0024-6107,1469-7750},
   MRCLASS = {99-06},
  MRNUMBER = {4795885},
       DOI = {10.1112/jlms.12984},
       URL = {https://doi.org/10.1112/jlms.12984},
}

@article {ferroni-matherne-stevens-vecchi,
    AUTHOR = {Ferroni, Luis and Matherne, Jacob P. and Stevens, Matthew and
              Vecchi, Lorenzo},
     TITLE = {Hilbert-{P}oincar\'e{} series of matroid {C}how rings and
              intersection cohomology},
   JOURNAL = {Adv. Math.},
  FJOURNAL = {Advances in Mathematics},
    VOLUME = {449},
      YEAR = {2024},
     PAGES = {Paper No. 109733},
      ISSN = {0001-8708,1090-2082},
   MRCLASS = {05B35 (05A15 05E18 13D40 14C15 30C15 52B40 52B45)},
  MRNUMBER = {4749982},
       DOI = {10.1016/j.aim.2024.109733},
       URL = {https://doi.org/10.1016/j.aim.2024.109733},
}

@article {brenti-welker,
    AUTHOR = {Brenti, Francesco and Welker, Volkmar},
     TITLE = {{$f$}-vectors of barycentric subdivisions},
   JOURNAL = {Math. Z.},
  FJOURNAL = {Mathematische Zeitschrift},
    VOLUME = {259},
      YEAR = {2008},
    NUMBER = {4},
     PAGES = {849--865},
      ISSN = {0025-5874,1432-1823},
   MRCLASS = {52B11},
  MRNUMBER = {2403744},
       DOI = {10.1007/s00209-007-0251-z},
       URL = {https://doi.org/10.1007/s00209-007-0251-z},
}

@article {hameister-rao-simpson,
    AUTHOR = {Hameister, Thomas and Rao, Sujit and Simpson, Connor},
     TITLE = {Chow rings of vector space matroids},
   JOURNAL = {J. Comb.},
  FJOURNAL = {Journal of Combinatorics},
    VOLUME = {12},
      YEAR = {2021},
    NUMBER = {1},
     PAGES = {55--83},
      ISSN = {2156-3527,2150-959X},
   MRCLASS = {05B35},
  MRNUMBER = {4195584},
MRREVIEWER = {Eva\ Ferrara Dentice},
       DOI = {10.4310/JOC.2021.v12.n1.a3},
       URL = {https://doi.org/10.4310/JOC.2021.v12.n1.a3},
}

@ARTICLE{eur-ferroni-matherne-pagaria-vecchi,
       author = {{Eur}, Christopher and {Ferroni}, Luis and {Matherne}, Jacob P. and {Pagaria}, Roberto and {Vecchi}, Lorenzo},
        title = "{Building sets, Chow rings, and their Hilbert series}",
      journal = {arXiv e-prints},
         year = 2025,
        month = apr,
          eid = {arXiv:2504.16776},
        pages = {arXiv:2504.16776},
          doi = {10.48550/arXiv.2504.16776},
archivePrefix = {arXiv},
       eprint = {2504.16776},
 primaryClass = {math.CO},
       adsurl = {https://ui.adsabs.harvard.edu/abs/2025arXiv250416776E}
}

@article {proudfoot-qniform,
    AUTHOR = {Proudfoot, Nicholas},
     TITLE = {Equivariant {K}azhdan-{L}ustzig polynomials of {$q$}-niform
              matroids},
   JOURNAL = {Algebr. Comb.},
  FJOURNAL = {Algebraic Combinatorics},
    VOLUME = {2},
      YEAR = {2019},
    NUMBER = {4},
     PAGES = {613--619},
      ISSN = {2589-5486},
   MRCLASS = {05B35 (20C99)},
  MRNUMBER = {3997514},
MRREVIEWER = {Eva\ Ferrara Dentice},
       DOI = {10.5802/alco.59},
       URL = {https://doi.org/10.5802/alco.59},
}

@incollection {stanley-unimodality,
    AUTHOR = {Stanley, Richard P.},
     TITLE = {Log-concave and unimodal sequences in algebra, combinatorics,
              and geometry},
 BOOKTITLE = {Graph theory and its applications: {E}ast and {W}est ({J}inan,
              1986)},
    SERIES = {Ann. New York Acad. Sci.},
    VOLUME = {576},
     PAGES = {500--535},
 PUBLISHER = {New York Acad. Sci., New York},
      YEAR = {1989},
   MRCLASS = {05E15 (05E10 20C15 52B20)},
  MRNUMBER = {1110850},
MRREVIEWER = {L. Bruce Richmond},
       DOI = {10.1111/j.1749-6632.1989.tb16434.x},
       URL = {https://doi.org/10.1111/j.1749-6632.1989.tb16434.x},
}

@ARTICLE{kannan-kuhne,
       author = {{Kannan}, Siddarth and {K{\"u}hne}, Lukas},
        title = "{The $\mathbb{S}_n$-equivariant Chow polynomial of the Braid matroid}",
      journal = {arXiv e-prints},
         year = 2025,
        month = apr,
          eid = {arXiv:2504.19829},
        pages = {arXiv:2504.19829},
          doi = {10.48550/arXiv.2504.19829},
archivePrefix = {arXiv},
       eprint = {2504.19829},
 primaryClass = {math.AG},
       adsurl = {https://ui.adsabs.harvard.edu/abs/2025arXiv250419829K}
}

@ARTICLE{branden-vecchi1,
       author = {{Br{\"a}nd{\'e}n}, Petter and {Vecchi}, Lorenzo},
        title = "{Chow~polynomials of~uniform matroids are real-rooted}",
      journal = {arXiv e-prints},
         year = 2025,
        month = jan,
          eid = {arXiv:2501.07364},
        pages = {arXiv:2501.07364},
          doi = {10.48550/arXiv.2501.07364},
archivePrefix = {arXiv},
       eprint = {2501.07364},
 primaryClass = {math.CO},
       adsurl = {https://ui.adsabs.harvard.edu/abs/2025arXiv250107364B}
}

@article{adiprasito2018hodge,
    AUTHOR = {Adiprasito, Karim and Huh, June and Katz, Eric},
     TITLE = {Hodge theory for combinatorial geometries},
   JOURNAL = {Ann. of Math. (2)},
  FJOURNAL = {Annals of Mathematics. Second Series},
    VOLUME = {188},
      YEAR = {2018},
    NUMBER = {2},
     PAGES = {381--452},
      ISSN = {0003-486X,1939-8980},
   MRCLASS = {05B35 (05E99 14C25 14T05)},
  MRNUMBER = {3862944},
MRREVIEWER = {Zvi\ Rosen},
       DOI = {10.4007/annals.2018.188.2.1},
       URL = {https://doi.org/10.4007/annals.2018.188.2.1},
}

\end{document}